\documentclass[twoside, reqno]{amsart}

\usepackage[left=2.5cm, right=2.5cm, top=3cm, bottom=2.5cm]{geometry}

\usepackage[colorlinks=true, pdfstartview=FitV, linkcolor=blue,citecolor=blue, urlcolor=blue]{hyperref}

\usepackage[symbol]{footmisc}

\usepackage[usenames]{xcolor}
\definecolor{labelkey}{rgb}{0,0,1}
\definecolor{Red}{rgb}{0.7,0,0.1}
\definecolor{Green}{rgb}{0,0.7,0}

\usepackage{amsfonts, amssymb, amsmath, amsthm, mathrsfs, bbm, cjhebrew, gensymb, textcomp, mathtools, dsfont,tikz}

\usepackage[normalem]{ulem}

\usepackage{commath, setspace, subcaption, parcolumns, multirow, multicol, accents, comment, marginnote, verbatim, empheq, enumerate, stackrel, enumitem}

\usepackage{lineno}

\makeatletter
\def\namedlabel#1#2{\begingroup
    #2%
    \def\@currentlabel{#2}%
    \phantomsection\label{#1}\endgroup
}
\makeatother

\usepackage[capitalize,nameinlink,noabbrev]{cleveref}
\usepackage{graphicx}
\usepackage{epsfig}
\usepackage{psfrag}
\usepackage{float}
\usepackage[active]{srcltx}

\newtheorem{Thm}{Theorem}[section]
\newtheorem{Cor}[Thm]{Corollary}
\newtheorem{Prop}[Thm]{Proposition}
\newtheorem{Lem}[Thm]{Lemma}
\newtheorem{Rmk}{Remark}[section]

\numberwithin{equation}{section}

\newcommand{\al}{\alpha}
\newcommand{\be}{\beta}
\newcommand{\de}{\delta} 
\newcommand{\De}{\Delta}
\newcommand{\eps}{\epsilon}

\newcommand{\ph}{\varphi}
\newcommand{\gam}{\gamma}
\newcommand{\Gam}{\Gamma}

\newcommand{\lam}{\lambda}
\newcommand{\Lam}{\Lambda}
\newcommand{\s}{\sigma}
\newcommand{\xe}{\xi-\eta}

\newcommand{\tht}{\theta}
\newcommand{\Tht}{\Theta}
\newcommand{\om}{\omega}

\newcommand{\ze}{\zeta}
\newcommand{\vs}{\varsigma}
\newcommand{\vr}{\varrho}

\newcommand{\lpj}{\triangle_j}
\newcommand{\lpk}{\triangle_k}
\newcommand{\tlpj}{\triangle_j^e}
\newcommand{\tlpk}{\triangle_k^e}

\newcommand{\ZZ}{\mathbb{Z}}
\newcommand{\RR}{\mathbb{R}}
\newcommand{\NN}{\mathbb{N}}

\newcommand{\lb}{\big\langle}
\newcommand{\rb}{\big\rangle}

\newcommand{\goesto}{\rightarrow}

\newcommand{\Sob}[2]{\lVert#1\rVert_{#2}}
\newcommand{\nrm}[1]{\lVert#1\rVert}

\newcommand{\bdy}{\partial}

\newcommand{\til}[1]{\widetilde{#1}}

\newcommand{\Ae}{\abs{\eta}}

\newcommand{\Hdot}{\dot{H}}
\newcommand{\Bdot}{\dot{B}}

\newcommand{\Acal}{\mathcal{A}}
\newcommand{\Bcal}{\mathcal{B}}

\newcommand{\Ft}{\mathcal{F}}

\DeclareMathOperator{\supp}{supp}

\DeclareMathOperator*{\Div}{div}

\usepackage{tikz}

\title[On well posedness of mildly dissipative active scalars in critical spaces]{On well-posedness of a mildly dissipative family of active scalar equations in borderline Sobolev spaces}
\author{Anuj Kumar, Vincent R. Martinez$^\dagger$}
\date{December 7, 2025\\
\indent\indent $^\dagger$Corresponding author}

\begin{document}

\begin{abstract} 
This paper considers a family of active scalar equations which modify the
generalized surface quasi-geostrophic (gSQG) equations through its constitutive
law and a dissipative perturbation. These modifications are characteristically
mild in the sense that they are logarithmic. The problem of well posedness,
in the sense of Hadamard, is then studied in a borderline setting of regularity
in analogy to the scaling-critical spaces of the gSQG equations. A novelty
of the system considered is the nuanced form of smoothing provided by the
proposed mild form of dissipation, which is able to support global well-posedness
at the Euler endpoint, but in a setting where the inviscid counterpart
is known to be ill-posed. A novelty of the analysis lies in the simultaneous
treatment of modifications in the constitutive law, dissipative mechanism,
and functional setting, which the existing literature has typically treated
separately. A putatively sharp relation is identified between each of the
distinct system-modifiers that is consistent with previous studies that
considered these modifications in isolation. This unified perspective is
afforded by the introduction of a linear model equation, referred to as
the \textit{protean system}, that successfully incorporates the more delicate
commutator structure collectively possessed by the gSQG family and upon
which each facet of well-posedness can effectively be reduced to its study.
\end{abstract}

\maketitle

{\noindent \small {\it {\bf Keywords: generalized surface quasi-geostrophic (SQG) equation, borderline regularity, well-posedness, logarithmic dissipation, instantaneous smoothing, protean system}
  }  }

\setcounter{tocdepth}{1}
\tableofcontents

\section{Introduction}
This article is concerned with the well-posedness of the initial value
problem for a family of dissipative
active scalar equations over the whole plane $\mathbb{R}^{2}$ in a borderline regularity setting:
\begin{align}
\label{eq:dgsqg}
\begin{cases}
\partial _{t}\theta + m(D)\theta +u \cdot \nabla \theta =0, \quad
\theta (0,x)=\theta _{0}(x),\quad x\in \mathbb{R}^{2},\quad t>0,
\\
u=\nabla ^{\perp}{\psi},\quad \Delta {\psi}=\Lambda ^{\beta}p(D)
\theta ,
\end{cases}
\end{align}
where $\beta \in [0,2]$ and $\Lambda =(-\Delta )^{1/2}$; the operators $m(D)$ and $p(D)$ denote Fourier multiplier operators, which
are assumed to be radial and of logarithmic type. Roughly speaking,
$p(D)$ will belong to a class of multipliers that may decay at most logarithmically
up to some order, while $m(D)$ will essentially be assumed to have non-negative Fourier transform, and thus serve as a mechanism to dissipate energy from the system.

When $p(D)=I$, {\eqref{eq:dgsqg}} corresponds to a dissipative perturbation
of the inviscid generalized surface quasi-geostrophic (gSQG) equation,
which is given by
\begin{align}
\label{eq:gsqg}
\partial _{t}\theta +u \cdot \nabla \theta =0,\quad u=\nabla ^{\perp}{
\psi},\quad \Delta {\psi}=\Lambda ^{\beta}\theta ,\quad 0\le \beta
\le 2.
\end{align}
For $\beta \in [0,1]$, {\eqref{eq:gsqg}} interpolates between
the 2D incompressible Euler equation in vorticity form ($\beta =0$) and
the SQG equation ($\beta =1$). The regime $\beta \in (1,2)$ represents
a family of active scalar equations with constitutive laws more singular
than SQG, while the endpoint $\beta =2$ constitutes a trivial case where
the streamfunction can effectively be identified with the advected scalar, resulting in
the reduction of {\eqref{eq:gsqg}} to the stationary equation
$\partial _{t}\theta =0$. This endpoint can be made nontrivial by modifying
the streamfunction equation with a positive power of a logarithmic multiplier
as in {\eqref{eq:dgsqg}}; the endpoint case modified in this way is often referred to as the Ohkitani model (see
\cite{ChaeConstantinCordobaGancedoWu2012,Okhitani2011,Okhitani2012}). The issue of whether singularities can
develop in finite-time from smooth initial data remains an outstanding
open problem for {\eqref{eq:gsqg}} when $\beta \in (0,2]$ (with the $\beta =2$ endpoint modified accordingly). Nevertheless, much progress has been made in the understanding of well-posedness or ill-posedness of the initial value problem associated to {\eqref{eq:gsqg}}; a detailed discussion of this progress in relation to {\eqref{eq:dgsqg}} is provided below.

The case where $p(D)$ provides a regularizing effect in the constitutive law, in the sense that its Fourier transform decays sufficiently fast at infinity, was originally studied in
\cite{ChaeWu2012} for the purpose of locating a minimal degree of regularization
to support a local existence theory for the 
inviscid system {\eqref{eq:gsqg}} in a borderline regularity setting;
the main example of interest is the 2D Euler equation (in vorticity
form) in the scaling-critical Sobolev space $\dot{H}^{1}(\mathbb{R}^{2})$; this Sobolev space preserves the natural scaling symmetry of the 2D Euler equation and indicates a regularity threshold for which the velocity fails, barely, to
be Lipschitz. In a similar spirit to \cite{ChaeWu2012}, we propose an alternative
mechanism for regularization that is \textit{dissipative}, as captured by
the multiplier, $m(D)$, and wish to locate the smallest possible degree
of such dissipation that supports a standard local solution theory (in the sense
of Hadamard). Such a form of regularization is categorically different
from an inviscid regularization mechanism since the dissipativity of
$m(D)$ may instantaneously confer additional regularity to the
solution. Ultimately, we show that local well-posedness
in the borderline Sobolev regularity setting holds by precisely quantifying, and subsequently exploiting, this instantaneous
smoothing effect, the main novelty here being that this smoothing effect can be very weak. Generally speaking, the overarching
goal of this work is to identify a putatively minimal relation between
$m(D)$ and $p(D)$ that guarantees a local existence theory in a regularity
setting that is scaling-critical for the corresponding inviscid
equation and is either known or expected to be ill-posed, while simultaneously quantifying the more subtle gain of regularity from the linear dissipative component.

The choice of dissipative perturbation will take on a logarithmic
form. We consider such perturbations as being
\textit{mildly dissipative}. This terminology is intended to distinguish from
\textit{weakly dissipative} perturbations, which refer to dissipative mechanisms
that are non-regularizing, such as damping effects due to friction, and
\textit{strongly dissipative} perturbations, which we interpret as typically
referring to dissipative operators like the fractional Laplacian,
${\Lambda }^{\gamma }$, that instantaneously
regularize the solution to become smooth in space. Thus, \textit{mild dissipation} indicates an intermediate form
of dissipation that lies between \textit{weak} and \textit{strong}. As we
will see below, \textit{mild dissipation} instantaneously confers additional
regularity to the solution, albeit at a categorically weaker level than
\textit{strong dissipation}. Note that as with the inviscid regularization considered in \cite{ChaeWu2012}, the dissipatively modified equation {\eqref{eq:dgsqg}} considered here does not possess a scaling symmetry. We thus
adopt the terminology of \textit{borderline regularity} throughout the manuscript in lieu of \textit{scaling-critical}
regularity to describe our functional setting.

The (homogeneous) Sobolev spaces,
$\dot{H}^{1+\beta }(\mathbb{R}^{2})$, where $\beta \in [1,2]$, is a scaling-critical
space for {\eqref{eq:gsqg}}. As in the case of the 2D Euler equation in vorticity
form, \textit{criticality} refers to the threshold of Sobolev
regularity where the velocity field barely fails to be Lipschitz; it is moreover characterized by the regularity level for which the norm remains invariant with respect to the natural scaling symmetry of the system.
From this point of view, the scaling-critical spaces identify a threshold for regularity
above which local well-posedness of the corresponding initial value problem
is expected to hold (the \textit{subcritical regime}), and below which some
form of ill-posedness can be expected to emerge (the \textit{supercritical regime)}. Classically,
local well-posedness above the critical regularity threshold at the endpoint case $\beta =0$ has been known at least since
\cite{Kato1967}. In fact, in this subcritical setting, local well-posedness holds 
in the spaces $H^{s}(\mathbb{R}^{d})$, $s>1+d/2$, (stated for the velocity) for all dimensions
$d\geq 2$ (see
\cite{BourguignonBrezis1974,EbinMarsden1969,Kato1967,Kato1972,Temam1974}).
The analogous result for the family {\eqref{eq:gsqg}} was established in
the works
\cite{ChaeConstantinCordobaGancedoWu2012,HuKukavicaZiane2015}, where local
well-posedness was established in
$H^{s}(\mathbb{R}^{2})$, $s>1+\beta$, for all $\beta\in(0,2)$. In the
presence of strong dissipation, global regularity of solutions in the case
$\beta \in (0,1]$ with $m(D)=\Lambda ^{\gamma }$, where
$\gamma =\beta $, was collectively established by the mathematical community
in the works
\cite{Resnick1995,ConstantinWu1999,KiselevNazarovVolberg2007,CaffarelliVasseur2010,KiselevNazarov2009,ConstantinVicol2012,ConstantinTarfuleaVicol2015,ConstantinIyerWu2008,MiaoXue2011}.

The issue of well-posedness for the Euler endpoint, $\beta =0$, at the critical
regularity threshold remained an outstanding open problem until the seminal work
of J. Bourgain and D. Li \cite{BourgainLi2015}, where a mechanism for norm
inflation was identified to establish strong ill-posedness in
$H^{1}(\mathbb{R}^{2})$. An alternative approach to establishing ill-posedness
was subsequently developed by T. Elgindi and M. Masmoudi in
\cite{ElgindiMasmoudi2020}. In a series of recent works, strong ill-posedness
was also established in the range $\beta \in (0,2]$ by D. C\'ordoba and
L. Mart\'inez-Zoroa
\cite{CordobaZoroa-Martinez2021,CordobaZoroa-Martinez2022} and I.-J. Jeong
and J. Kim \cite{JeongKim2021}. The Ohkitani model represented by the modified endpoint case $\beta =2$ in {\eqref{eq:gsqg}} was recently shown to be ill-posed in $H^{s}$, for $s>3$
\cite{ChaeJeongOh2023b}, by D. Chae, I.-J. Jeong, and S.-J. Oh, but globally well-posed when the model possessed
logarithmic-order dissipation \cite{ChaeJeongNaOh2023a} by the same authors with J. Na; similar ill-posedness
results for a class of models generalizing the Okhitani case, as well as
some dissipatively perturbed counterparts, were also obtained  \cite{ChaeJeongOh2023b}.
We emphasize that the framework developed in the present article is complementary
to that studied in \cite{ChaeJeongNaOh2023a,ChaeJeongOh2023b}; the reader
is referred to {\cref{rmk:okhitani}} for a more detailed discussion.
\par
Complementary to the ill-posedness results mentioned above are well-posedness the results
of D. Chae and J. Wu \cite{ChaeWu2012} and M.S. Jolly with the present
authors \cite{JollyKumarMartinez2020b}, where a mild inviscid regularization
of {\eqref{eq:gsqg}} is studied in order to recover well-posedness in the
borderline spaces. This regularization modifies the constitutive law for
the velocity with a logarithmic multiplier operator as
\begin{align}
\label{def:mod:gsqg}
u=-\nabla ^{\perp}{\Lambda }^{\beta -2}\left (\ln \left (e-\Delta
\right )\right )^{\widetilde{\mu }}\theta , \quad \widetilde{\mu }<0.
\end{align}
In \cite{ChaeWu2012}, local existence and uniqueness was shown for
$\beta \in [0,1]$ provided that $\widetilde{\mu }<-1/2$, while
\cite{JollyKumarMartinez2020b} established local well-posedness for the
more singular range $\beta \in (1,2)$ provided, again, that
$\widetilde{\mu }<-1/2$. The key difference between the regimes
$\beta \in [0,1]$ and $\beta \in (1,2)$ are in the identification of a
suitable linear system that allows one to accommodate the more nuanced
commutator structure of {\eqref{eq:gsqg}} when establishing stability-type
estimates \textit{at the critical regularity level}. In particular, the
need for such a system is crucial when establishing
\textit{continuity of the solution map} since the classical estimates for
the transport equation require control of
$\lVert \nabla u\rVert _{H^{\beta }}$; one must thus appeal to additional
cancellation in the form of commutators. In an extension of the seminal
work of J. Bourgain and D. Li, it was shown by H. Kwon in
\cite{Kwon2020} that the threshold $\widetilde{\mu }=-1/2$ is in fact sharp
in the endpoint case of the 2D Euler equation by demonstrating strong ill-posedness
in $H^{1}(\mathbb{R}^{2})$ of the corresponding initial value problem for
all $-1/2\leq \widetilde{\mu }<0$. These considerations were
subsequently extended to the inviscid gSQG family by I.-J. Jeong and J.
Kim in \cite{JeongKim2021} and D. C\'ordoba and L. Mart\'inez-Zoroa in
\cite{CordobaZoroa-Martinez2022}. On the one hand, since the ill-posedness
phenomenon arises instantaneously in time, one cannot expect a simple damping
mechanism, for instance, in the form of adding $-\gamma \theta $, where
$\gamma >0$ to the right-hand side of {\eqref{eq:gsqg}}, to preclude the
behavior leading to ill-posedness. On the other hand, as previously mentioned,
if one adds dissipation in the form of $-\gamma \Lambda ^{\kappa }$, where
$\gamma ,\kappa >0$, then the initial value problem becomes locally well-posed
in the corresponding scaling-critical spaces and {globally} well-posed for
sufficiently small data. A natural question to ask, therefore, is what
is the \textit{weakest} form of dissipation that one could add to support
a well-posedness theory in a borderline regularity setting?

In this paper, we develop a general framework of mild dissipation
to comprehensively address this issue for the full range of {\eqref{eq:gsqg}}, $\beta \in [0,2]$, in a borderline regularity setting
that encompasses the classical Sobolev spaces $H^{1+\beta }(\mathbb{R}^{2})$ and frequency-weighted Sobolev spaces $H^{1+\beta}_{\omega}(\mathbb{R}^2)$. The addition of mild dissipation allows us to also consider
constitutive laws that balance the dissipation with logarithmically more
singular velocities. The incorporation of frequency-weights, $\omega$, provide us with further flexiblity to balance the effects from both the modifications of the dissipative term and constitutive law 
through the functional setting. With all
such modifications in place, we develop a unified analysis
of local well-posedness in a general borderline regularity setting
to {\eqref{eq:dgsqg}}, for all
$\beta \in [0,2]$, where $\omega $ belongs an appropriate set of weights.
Allowing for a set of weights allows one to navigate the lack
of a scaling-symmetry in the presence of either mild dissipation or logarithmically-modified
velocity and ultimately provides a sharper family of borderline spaces
in this context. From this point of view, our results touch upon those
of M. Vishik \cite{Vishik1999}, where Yudovich's classical uniqueness theorem
was extended to a borderline Besov space setting where growth in the norm
is allowed, but in a controlled way that is characterized by frequency weights; the role of
the weights $\omega $ {plays a} similar role for us, albeit in an
$L^{2}$ setting.

As in \cite{ChaeWu2012,JollyKumarMartinez2020b}, we identify a putatively
minimal condition for local well-posedness that jointly involves the dissipation,
modification of the constitutive law, as well as the weights
$\omega $ that modify the functional setting. The class of multipliers
considered here are sufficiently broad to accommodate powers of logarithms
or iterated logarithms. One particular special case that is covered is
$m(D)=(\ln (I-\Delta ))^{\mu}$ and
$p(D)=(\ln (I-\Delta ))^{\widetilde{\mu }}$, where we are able to establish
local well-posedness in $H^{1+\beta }(\mathbb{R}^{2})$, for all
$\beta \in [0,2]$, provided that $\mu >\widetilde{\mu }+1/2$, where
$\widetilde{\mu }>-1/2$. Our result therefore interpolates between the
well-posedness results in \cite{ChaeWu2012,JollyKumarMartinez2020b} up
to the sharp thresholds where the equation is known to experience ill-posedness
\cite{JeongKim2021,Kwon2020}, at least in the range
$\beta \in [0,1]$.

We additionally establish a
\textit{mild instantaneous smoothing effect} for {\eqref{eq:dgsqg}} that derives
from the regularizing mechanism of the corresponding linear equation. Indeed,
let $\nu (D)$ denote a multiplier operator that belongs to the class
$\mathscr{M}_{S}(m)$ defined in {\eqref{def:M:m}}. Then we establish the
following result for {\eqref{eq:dgsqg}}:
\begin{equation*}
e^{\lambda t \nu (D)}\theta (t,\cdotp )\in H^{\sigma },\quad
\text{whenever}\quad \theta _{0}\in H^{\sigma },
\end{equation*}
for all $\lambda \in (0,1)$ sufficiently small. In particular, with
$m(D)=(\ln (I-\Delta ))^{\mu}$ where $\mu \ge 1$, this implies
\begin{equation*}
\theta (t,\cdotp )\in H^{\sigma +\lambda t},\quad \text{whenever}
\quad \theta _{0}\in H^{\sigma }.
\end{equation*}
This observation along with a suitable $L^{\infty}$ maximum principle for
the dissipative term allows us to obtain global-in-time existence in the
case of $\beta =0$ (Euler endpoint). This complements the result in H.
Dong and D. Li \cite{DongLi2015}, where global well-posedness in the borderline
space $\dot{H}^{1}\cap \dot{H}^{-1}$ was also established for the corresponding
logarithmic inviscid regularization for a sufficiently large power of the
logarithm, as well as M. Vishik \cite{Vishik1998}, where global existence
for the Euler endpoint was established, but by varying the functional setting,
specifically considering an $L^{p}$--based borderline Besov space. In addition
to providing an alternative mechanism for global existence, our framework
explores the simultaneous effect of regularization or ``singularization''
in the constitutive law, functional setting, and dissipation.

When {\eqref{eq:gsqg}} is dissipatively perturbed by the strong dissipation,
${\Lambda }^{\kappa }$, where $\kappa \in (0,2)$, the smoothing effect
conferred on its solutions is much stronger than the one conferred by mild
dissipation. In a series of works
\cite{DongLi2008,DongLi2010,Biswas2014,BiswasMartinezSilva2015,Li2021,JollyKumarMartinez2020a},
it was shown that the unique solution of the strongly dissipative gSQG
equations emanating from initial data belonging to the scaling-critical
Sobolev space $H^{1+\beta -\kappa }(\mathbb{R}^{2})$, instantaneously enters
a class of smooth functions, known as the Gevrey class, which characterizes
a scale of regularity between the $C^{\infty}$ class of smooth functions
and the $C^{\omega }$ class of real-analytic functions. Such results strengthened
the existing well-posedness results in critical regularity settings
\cite{ChaeLee2003,Miura2006,ChenMiaoZhang2007}. In the setting of euclidean
space, the Gevrey class enforces exponential decay of the frequency spectrum
of its members at some rate. For solutions of the strongly dissipative
gSQG equations, the rate of exponential decay is shown to grow in time.
The series of works mentioned above ultimately culminated in the recent
work \cite{JollyKumarMartinez2020a}, where this
\textit{strong smoothing} phenomenon was established for the range
$\beta \in (1,2]$, with the endpoint $\beta =2$ accordingly modified by
a logarithm (see discussion following {\eqref{eq:gsqg}}), thus completing
this line of investigation for the strongly dissipative gSQG equation in
the setting of the scaling-critical Sobolev spaces. In contrast, the smoothing
effect observed for the mildly dissipative equation given by {\eqref{eq:dgsqg}} is categorically weaker than the one observed for its
strongly dissipative counterpart. In establishing this form of smoothing,
we expand upon the celebrated Gevrey-norm technique of C. Foias and R.
Temam \cite{FoiasTemam1989} to accommodate logarithmic-type multipler operators.

It should be emphasized that the setting of {\eqref{eq:gsqg}} when
$\beta \in (1,2]$ and $\beta >\kappa +1$ exhibits a
\textit{strongly quasilinear} structure due to the fact that the velocity,
treated as a coefficient of the gradient, is of higher order than the linear
dissipative term; from this point of view, the case of mild dissipation
is significantly supercritical. Although the quasilinearity in {\eqref{eq:dgsqg}} is much stronger than the fractionally dissipative counterpart
of {\eqref{eq:gsqg}}, we point out the lack of a bona fide scaling-critical
space effectively places our setting in a subcritical regime, albeit barely
so. The quasilinearity is therefore a source of difficulty that forces
us to exploit the nuanced commutator structure of {\eqref{eq:dgsqg}}, but
is ultimately overcome by jointly exploiting the
\textit{barely subcritical} functional setting of the problem. A notable
related work that may be viewed as somewhat ``dual'' to the considerations
here is that of O. Lazar and L. Xue \cite{LazarXue2019}, where global regularity
of solutions is established for a strongly dissipative counterpart of {\eqref{eq:gsqg}} with constitutive laws that are logarithmically modified
(via $p(D)$ as in {\eqref{eq:dgsqg}}) in a similar fashion to framework developed
in the present article, but for the purpose of considering a
\textit{slightly supercritical} scenario, where the dissipation is chosen
in a particular relation to the constitutive law; in this scenario, global
regularity is known without the logarithmic modification (\cite{ConstantinIyerWu2008,MiaoXue2011})
of the constitutive law. The present article, in contrast, identifies a
general condition between $m(D)$ and $p(D)$ that guarantees local well-posedness
at borderline regularity and the smoothing effects naturally induced by
$m(D)$. Moreover, this general condition specifically accommodates logarithmic-type
{multipliers} that characterize a class of dissipative operators that are
strictly weaker than the strongly dissipative operators previously considered
in the literature. The reader is referred to {\cref{rmk:lazar}} for further
discussion.

Lastly, central to our framework is a model equation that we refer to as
the \textit{protean system}. The protean system is linear system upon which
the proof of well-posedness is found to turn around entirely. This system
takes its inspiration from the role played by the linear transport equation
in the study of the 2D Euler and SQG equations. Indeed, from our perspective,
the linear transport equation can be realized as the protean system of
the 2D Euler equation. However, when considering the gSQG family beyond
the SQG point, i.e., $1<\beta \leq 2$, the paradigm of the transport equation
breaks down. This phenomenon was observed in the most recent work of the
authors with M.S. Jolly \cite{JollyKumarMartinez2020b}, where continuity
of the solution operator required additional modifications to the transport
equation. This modification was also required for establishing local well-posedness
for large data in the critical Sobolev space setting for the supercritically
fractional dissipative gSQG equation by the same authors in
\cite{JollyKumarMartinez2020a}. We ultimately realize these key modifications
in the form of the protean system proposed here (see
\cref{sect:mod:flux}). The present article therefore represents a conclusion
to a trilogy works starting from
\cite{JollyKumarMartinez2020a,JollyKumarMartinez2020b} that forms a culmination
in the understanding gained therein about the gSQG family in borderline
regularity settings.

\subsection{The Protean System}\label{sect:mod:flux}
Our main apriori analysis will be centered around a model equation from which all the estimates relating to the well-posedness and instantaneous smoothing property for the original system \eqref{eq:dgsqg} will effectively follow as a special case; for this reason, we refer to the model equation as the \textit{protean system}. For \eqref{eq:dgsqg}, the protean system is given by a \textit{linear conservation law}, which becomes nonlinear upon appropriately substituting for the solution in its flux. The need for this structure arises from the more singular nature of the constitutive law of the equations. Indeed, in contrast with the regime $\be\in(1,2]$, the protean system reduces simply to a \textit{linear transport equation} (see \eqref{eq:mod:claw} below). We first introduce the protean system here and detail the ways in which it is used to demonstrate well-posedness and smoothing for \eqref{eq:dgsqg}. We develop the apriori analysis for the protean system and the well-posedness of its initial value problem in \cref{sect:apriori}.

Suppose that $q=q(t,x)$, $G=G(t,x)$ are sufficiently smooth functions. For $\be\in[0,2]$, let $m(D)$ and $p(D)$ be Fourier multiplier operators, i.e.,
    \[
        (\mathcal{F}m(D)\ph)(\xi)=m(\xi)(\mathcal{F}\ph)(\xi),\qquad(\mathcal{F}p(D)\ph)(\xi)=p(\xi)(\mathcal{F}\ph)(\xi),\quad m(\xi), p(\xi)\geq0.
    \]
Then define a Fourier multiplier $a(D)$ by
    \begin{align}\label{def:aD}
        a(D):={\Lam}^{\be-2}p(D).
    \end{align}
The precise restrictions on $m(D)$ and $p(D)$ will be specified later. Given $\tht_0=\tht_0(x)$, we will consider the following initial value problem for a linear conservation law:
    \begin{align}\label{eq:mod:claw}
        {\partial_{t}\tht+  \Div F_q(\tht)=-m(D)\tht}+G,\quad\tht(0,x)=\tht_0(x),
    \end{align}
where the flux is given by
    \begin{align}\label{def:mod:flux}
        F_q(\tht)=
            \begin{cases}
                (\nabla^{\perp}a(D)q) \tht,\quad &\text{if}\quad \be \in [0,1],\\
                (\nabla^{\perp}a(D)q) \tht+a(D)(({\nabla}^\perp\tht)q), \quad &\text{if}\quad \be \in (1,2].
            \end{cases}
    \end{align}
Observe that the flux satisfies the following identity:
    \begin{align}\label{eq:cancel}
        \Div F_{-\tht}(\tht)=-({\nabla}^\perp a(D)\tht)\cdotp{\nabla}\tht=v\cdotp{\nabla}\tht.
    \end{align}
where
    \begin{align}\label{def:v}
    v=v(q):=-\nabla^{\perp}a(D)q.
    \end{align}
Observe also that ${\nabla}\cdotp v=0$, for all $q$ sufficiently smooth, which yields the identity
    \begin{align}\label{eq:v:cancel}
        \lb v\cdotp\nabla h,h\rb=0,
    \end{align}
for any sufficiently smooth functions $h$. Hence, one recovers \eqref{eq:dgsqg} from \eqref{eq:mod:claw} when  $q=-\tht$, and $G\equiv0$. In particular, one may obtain apriori estimates for \eqref{eq:dgsqg} by obtaining them for \eqref{eq:mod:claw} and simply specializing to the case $q=-\tht$.

For uniqueness, observe that if  $\tht^{(1)}, \tht^{(2)}$ are distinct solutions, then the difference $\Tht:=\tht^{(1)}-\tht^{(2)}$ is governed by
    \begin{align}\label{eq:difference}
        \partial_{t}\Tht+m(D) \Tht+\Div F_{-\tht^{(1)}}(\Tht)=\Div F_{\Tht}(\tht^{(2)}),\quad\Tht(0,x)=\Tht_0(x)
    \end{align}
with $q=-\tht^{(1)}$ and $G=\Div F_{\Tht}(\tht^{(2)})$. The related issue of continuity with respect to initial data is more delicate. Indeed, assessing stability of \eqref{eq:mod:claw} in the borderline space $H^{1+\be}$ is not possible through a direct analysis of \eqref{eq:difference:intro} due to the loss of derivatives experienced through the flux term in the regime $\be\in(1,2]$. A direct analysis is ultimately limited to establishing continuity in the \textit{weaker topology} of $H^\be$. Bootstrapping from the weaker topology to the topology of the phase space can effectively be carried out by making use of a splitting technique of Kato that analyzes the \textit{gradient} of difference in the weaker topology. However, while this scheme is well-adapted to the transport equation, it is more delicate to accommodate the case of \eqref{eq:mod:claw} when $\be\in(1,2]$.  

To see this, we consider a decomposition $\nabla\tht=\vs+\ze$ and aim to show that for a sequence $\tht_0^n$ converging to $\tht_0$, that the corresponding solution $\nabla\tht^n=\vs^n+\ze^n$ converges to $\nabla\tht$ in $H^{1+\be}$ by showing that $\vs^n, \ze^n$ converge to $\vs,\ze$ in the weaker topology $H^\be$. We specifically assume that $\vs^n,\ze^n,\vs,\ze$ satisfy
    \begin{align}\label{eq:vs:protean}
            \partial_{t}\vs^{n}+m(D)\vs^{n}+ \Div F_{-\tht^{n}}(\vs^{n})&=\Div F_{\nabla\tht}(\tht),\quad\vs^{n}(0,x)=\nabla\tht_0(x),\\
            \partial_{t}\ze^{n}+m(D)\zeta^{n}+ \Div F_{-\tht^n}(\ze^{n})&=\Div F_{\nabla\tht^n}(\tht^{n})-\Div F_{\nabla\tht}(\tht),\quad \ze^{n}(0,x)=\nabla\tht^{n}_0(x)-\nabla\tht_{0}(x),\label{eq:ze:protean}
         \end{align}
where we identify $\tht_0,\vs,\ze$ with $\tht_0^\infty, \vs^\infty, \ze^\infty$. Now observe that \eqref{eq:vs:protean} has the structure of \eqref{eq:mod:claw}  upon making the replacements $q\mapsto -\tht^n$ and $G\mapsto \Div F_{\nabla\tht}(\tht)$. Similarly, \eqref{eq:ze:protean} has the structure of \eqref{eq:mod:claw} upon making the replacement $q\mapsto -\tht^n$ and $G\mapsto \Div F_{\nabla\tht^n}(\tht^{n})-\Div F_{\nabla\tht}(\tht)$. In particular, the systems corresponding to the differences $\vs^n-\vs$ and $\ze^n-\ze$ satisfy systems analogous to \eqref{eq:difference:intro}, which \textit{once again} possess the structure of the protean system and ultimately allow estimates to close in the desired manner.

In order to leverage the estimates obtained for \eqref{eq:mod:claw} and  establish existence and smoothing of solutions to \eqref{eq:dgsqg}, one must identify a suitable approximation of \eqref{eq:dgsqg} that also satisfies the desired estimates, but uniformly in the approximation. Under the assumptions we make for $m(D)$ and $p(D)$, an artificial viscosity approximation scheme will be sufficient; this will be self-evident from the apriori analysis we perform below in \cref{sect:apriori}. Of course, once existence has been established, the proofs of uniqueness and continuity with respect to initial conditions can be performed on the equations rigorously.

\section{Mathematical Preliminaries}\label{section:notation:preliminaries}

Let $d\geq1$ and denote by $\mathscr{S}(\RR^d)$ the space of Schwartz class functions on $\RR^d$ and by $\mathscr{S}'(\RR^d)$ the space of tempered distributions. We will denote by $\hat{f}$ or $\mathcal{F}(f)$, the Fourier transform of a tempered distribution $f$, defined as
	\[
	    \hat{f}(\xi):=\int_{\RR^d}e^{-2\pi i x\cdot \xi}f(x)dx.
    \]
Recall that $\mathcal{F}$ is an isometry on $L^2$, i.e.,
    \begin{align*}
        \lb \hat{f},\hat{g}\rb= \lb f, g\rb:=\int_{\RR^d}f(x)\overline{g(x)}dx.
    \end{align*}
The fractional laplacian operators, denoted ${\Lam}^\s$, $\s\in\RR$, are defined in terms of the Fourier transform by
    \begin{align}\label{def:Lam}
        \mathcal{F}({\Lam}^\s f)(\xi)=|\xi|^\s\mathcal{F}(f).
    \end{align} 
For $d\geq1$ and $\s\in\RR$, the  homogeneous and the inhomogeneous Sobolev spaces are defined as
    \begin{align}
        &\Hdot^\s(\RR^d):=\left\{f\in \mathscr{S}'(\RR^d):\hat{f}\in L^1_{loc}(\RR^d),\quad \Sob{f}{\Hdot^\s}:=\Sob{{\Lam}^\s f}{L^2}<\infty\right\},\label{def:hom:Sob:norm}\\
        &H^\s(\RR^d):=\left\{f\in \mathscr{S}'(\RR^d):\hat{f}\in L^1_{loc}(\RR^d),\quad \Sob{f}{H^\s}:=\Sob{(I-\De)^{\s/2} f}{L^2}<\infty\right\}.\label{def:inhom:Sob:norm}
    \end{align}

Recall that $H^\s(\RR^d)$ is a Hilbert space for all $\s\in\RR$, whereas  $\Hdot^\s(\RR^d)$ is a Hilbert space if and only if $\s<d/2$. The inhomogeneous spaces are nested $H^{\s'}(\RR^d)\subset H^\s(\RR^d)$, whenever $\s'>\s$, and moreover the embedding is compact over compact sets. The homogeneous spaces on the other hand are in general only related by the following interpolation inequality: Given $\s_1\leq\s\leq \s_2$, we have
	\begin{align}\label{est:interpolation}
	    \nrm{f}_{\Hdot^{\s}(\RR^d)}\leq \Sob{f}{\Hdot^{\s_1}(\RR^d)}^{\frac{\s_2-\s}{\s_2-\s_1}}\Sob{f}{\Hdot^{\s_2}(\RR^d)}^{\frac{\s-\s_1}{\s_2-\s_1}}.
	\end{align}
Lastly, observe that for each $\s\geq0$, there exists $C>0$ (depending also on $d$) such that
    \begin{align}\label{eq:Sob:equiv}
        C^{-1}\Sob{f}{H^\s(\RR^d)}\leq \Sob{f}{\Hdot^\s(\RR^d)\cap L^2(\RR^d)}\leq C\Sob{f}{H^\s(\RR^d)},
    \end{align}
where we have adopted the convention
    \[
        \Sob{f}{X\cap Y}^2:=\Sob{f}{X}^2+\Sob{f}{Y}^2,
    \]
where $X, Y$ are (semi)normed vector spaces.
    
\subsubsection{Frequency-weighted Sobolev spaces} Due to the presence of a logarithmic order dissipation and a logarithmic modification in the constitutive law in \eqref{eq:dgsqg}, it will be natural to consider function spaces that accommodate logarithmic decay at infinity in frequency; they have been referred to in the literature as the \textit{log-Sobolev} spaces. Spaces such as these and their natural generalizations have been an object of study in recent years. We refer the reader to \cite{CobosDominguez2015, ColombiniDelSantoFanelliMetivier2015} and the references therein for a rigorous development of these spaces. For our purposes, it will be convenient to introduce these spaces in a greater generality, then identify precise conditions on the weights afterwards.

Let $\om: [0,\infty)\mapsto(0,\infty)$. We denote its associated multiplier operator by $\om(D)$; note that $\om(D)=\om(|D|)$, i.e., the symbol of $\om(D)$ is a radial function. Recall that $\om(D)$ is defined by the relation  $\mathcal{F}(\om(D)f)(\xi)=\om(|\xi|)(\mathcal{F}f)(\xi)$. Then for $\s \in \mathbb{R}$, we define the $\om$-weighted Sobolev spaces on $\RR^{d}$ by
    \begin{align}
        \Hdot^{\s}_{\om}(\RR^d)&:=\left\{f\in \mathscr{S}'(\RR^d):\hat{f}\in L^1_{loc}(\RR^d),\quad \Sob{f}{\Hdot^{\s}_{\om}}:=\Sob{\om(D)f}{\Hdot^\s}<\infty\right\},\label{def:hom:Sob:norm:log}\\
        H^{\s}_{\om}(\RR^d)&:=\left\{f\in \mathscr{S}'(\RR^d):\hat{f}\in L^1_{loc}(\RR^d),\quad \Sob{f}{H^{\s}_{\om}}:=\Sob{\om(D)f}{H^\s}<\infty\right\},\label{def:inhom:Sob:norm:log}
    \end{align}
In this setting, an interpolation inequality analogous to \eqref{est:interpolation} also holds: for all $\gam_1\leq \gam\leq \gam_2$, we have
    \begin{align}\label{est:interpolation:weighted}
        \Sob{\om(D)^{\gam}f}{L^2(\RR^d)}\leq \Sob{\om(D)^{\gam_1}f}{L^2(\RR^d)}^{\frac{\gam_2-\gam}{\gam_2-\gam_1}}\Sob{\om(D)^{\gam_2}f}{L^2(\RR^d)}^{\frac{\gam-\gam_1}{\gam_2-\gam_1}}.
    \end{align}
The proof is omitted since the same argument for proving \eqref{est:interpolation} can be applied to prove \eqref{est:interpolation:weighted}.

\begin{Rmk}
We point out that for our main results (see \cref{sect:main:results}), we will restrict the class of weights, $\om$, that we consider by imposing certain restrictions on them that will be useful for the analysis (see \eqref{def:MM} in \cref{sect:op:classes}). These restrictions will imply the following growth condition: there exist constants $C, N>0$ such that
  \begin{align}\notag
      \om(r)\le C(\ln(e+r))^{N},
    \end{align}
for all $r\geq0$. For example, we may consider $\om(r)=(\ln(e+r^2))^{\rho_1}$, $\rho_1\in \mathbb{R}$. In this particular case, we will denote the corresponding inhomogeneous and homogeneous log-Sobolev spaces by $H^{\s}_\om=H^{\s, \rho_1}$ and $\Hdot^{\s}_\om=\Hdot^{\s,\rho_1}$. Similarly, for $\om(r)=(\ln(e+r^2))^{\rho_1}(\ln(e+\ln(1+r^2))^{\rho_2}$, $\rho_1, \rho_2\in \mathbb{R}$, we will denote the corresponding inhomogeneous and homogeneous log-Sobolev spaces by $H^{\s, \rho_1, \rho_2}$ and $\Hdot^{\s,\rho_1, \rho_2}$, and so on.
\end{Rmk}

\subsection{Littlewood-Paley Decomposition}\label{sect:lp:decomposition}
In this section, we provide a brief review of the Littlewood-Paley decomposition and refer the reader to \cite{BahouriCheminDanchinBook2011, CheminBook1995} for additional details. First, we introduce the space 
    \begin{align*}
        \mathscr{Q}(\RR^d):=\left\{f\in \mathscr{S}(\RR^d): \int_{\RR^d}f(x)x^{\tau}\, dx=0, \quad \abs{\tau}=0,1,2,\cdots \right\}.
    \end{align*}
The topological dual space, $\mathscr{Q}'(\RR^d)$, of $\mathscr{Q}(\RR^d)$, can be identified with the space of tempered distributions modulo polynomials, that is, as
    \begin{align*}
        \mathscr{Q}'(\RR^d)\cong\mathscr{S}'(\RR^d)/\mathscr{P}(\RR^d).
    \end{align*}
{where $\mathscr{P}(\RR^d)$ denotes the vector space of polynomials}.
	
Given $d\geq1$, fixed, let  ${\Bcal}(r)$ denote the open ball in $\RR^d$ of radius $r$ centered at the origin and ${\Acal}(r_{1},r_{2})$ denote the open annulus in $\RR^d$ with inner and outer radii $r_{1}$ and $r_{2}$, and centered at the origin. One can construct two non-negative radial functions $\chi,\phi\in\mathscr{S}(\RR^d)$ with $\supp\chi\subset{\Bcal}(1)$ and $\supp\phi\subset{\Acal}(2^{-1},2)$ such that the following properties are satisfied. For $\chi_j(\xi):=\chi(2^{-j}\xi)$ and $\phi_j(\xi):=\phi(2^{-j}\xi)$,
    \begin{align}\label{eq:partition}
	    \begin{cases}
	     &\sum_{j\in\ZZ}\phi_j(\xi)=1,\quad \forall\,\xi\in\RR^d\setminus\{\mathbf{0}\},\\
	     &\chi+\sum_{j\geq0}\phi_j\equiv 1,\\
	     &\supp\phi_i\cap\supp\phi_j=\varnothing,\quad |i-j|\geq2,\\
	     &\supp\phi_i\cap\supp\chi =\varnothing,\quad i\geq1.
	    \end{cases}
	\end{align} 
It will be convenient to define the following notation:
    \begin{align}\label{def:ann:ball}
        {\Acal}_{j}= {\Acal}(2^{j-1},2^{j+1}),\quad{\Acal}_{\ell,k}= {\Acal}(2^{\ell},2^{k}),\quad  {\Bcal}_j={\Bcal}(2^j),
    \end{align}
so that, in particular, ${\Acal}_j={\Acal}_{j-1,j+1}$. With this notation, observe that
    \begin{align}\label{eq:rewrite:supp}
        \supp\phi_j\subset{\Acal}_j,\quad \supp\chi_j\subset{\Bcal}_j.
    \end{align}
We denote the homogeneous Littlewood-Paley dyadic blocks by ${\lpj}$ and $S_{j}$. These are both defined in terms of its Fourier transform by
    \begin{align}\notag
	    \mathcal{F}({\lpj}f)=\phi_{j}\mathcal{F}(f), \quad \mathcal{F}(S_{j}f)=\chi_{j}\mathcal{F}(f). 
	\end{align}

One then has the following identity 
    \begin{align*}
        f&=S_if+\sum_{j\geq i}\lpj f,\quad i\in\ZZ,\quad f\in\mathscr{S}'(\RR^d).
    \end{align*}
In fact, when restricted to $\mathscr{Q}'$, one has
    \begin{align*}
        f&=\sum_{j\in\ZZ}\lpj f,\quad f\in\mathscr{Q}'(\RR^d).
    \end{align*}
Then a useful characterization of Sobolev norms is as follows: given $\s\in\RR$, there exists $C>0$ such that
    \begin{align}\label{eq:Sob:Bes}
        C^{-1}\Sob{f}{\Hdot^{\s}(\RR^d)}\leq \left(\sum_{j\in \mathbb{Z}}2^{2\s j}\Sob{\triangle_{j}f}{L^2(\RR^d)}^{2}\right)^{1/2}\le C\Sob{f}{\Hdot^{\s}(\RR^d)},
    \end{align}
Note that this equivalence also holds when the support of $\lpj$ is rescaled by any fixed number (see \cite{BahouriCheminDanchinBook2011}). We refer to the intermediate quantity as the homogeneous Besov norm $\Sob{f}{\dot{B}^{\s}_{2,2}(\RR^d)}$; whenever $\s\geq0$, we define the corresponding \textit{inhomogeneous} Besov norm by 
    \[
        \Sob{f}{B^{\s}_{2,2}(\RR^d)}^2:=\Sob{f}{\dot{B}^{\s}_{2,2}(\RR^d)\cap L^2(\RR^d)}^2=\Sob{f}{\dot{B}^{\s}_{2,2}(\RR^d)}^2+\Sob{f}{L^2(\RR^d)}^2.
    \]
With this notation, we shall also make use of the following notation for the Besov-space analog of the frequency-weighted Sobolev spaces:
    \begin{align}\label{def:Bes:om}
        \Sob{f}{\dot{B}^{\s}_{\om}(\RR^d)}=\Sob{\om(D)f}{\dot{B}^{\s}_{2,2}(\RR^d)}, \quad \Sob{f}{{B^\s_\om}(\RR^d)}=\Sob{\om(D)f}{\dot{B}^{\s}_{2,2}(\RR^d)\cap L^2(\RR^d)},
    \end{align}
where the latter quantity denotes the corresponding inhomogeneous Besov norm whenever $\s\geq0$.

The relation between the Littlewood-Paley blocks and the fractional laplacian is captured by the following Bernstein-type inequalities.	
\begin{Lem}[Bernstein inequalities]\label{lem:Bernstein}
Let $\s\in\RR$ and $1\le p \le q\le \infty$. Then there exists $C>0$ such that
	\begin{align*}
	    C^{-1}2^{\s j}\nrm{{\lpj}f}_{L^q(\RR^d)}\le \nrm{{\Lam}^{\s}{\lpj}f}_{L^q(\RR^d)}\le C 2^{\s j+dj(\frac{1}{p}-\frac{1}{q})}\nrm{{\lpj}f}_{L^p(\RR^d)},
	\end{align*}
for all $j\in\ZZ$ and $f\in\mathscr{S}'(\RR^d)$.
\end{Lem}

Let us recall the following classical product estimate in homogeneous Sobolev spaces (see \cite{BahouriCheminDanchinBook2011,RunstSickel1996}).
\begin{Lem}\label{lem:Sobolev}
Suppose that $\s_1,\s_2 \in (-d/2,d/2)$ and $\s_1+\s_2>0$. Let $f\in \Hdot^{\s_1}(\RR^d)$ and $g\in \Hdot^{\s_2}(\RR^d)$. Then there exists $C>0$ such that
\begin{align*}
    \Sob{fg}{\Hdot^{\s_1+\s_2-\frac{d}{2}}(\RR^d)}\le C\Sob{f}{\Hdot^{\s_1}(\RR^d)}\Sob{g}{\Hdot^{\s_2}(\RR^d)}.
\end{align*}
\end{Lem}

\begin{Rmk}\label{rmk:convention} For the remainder of the manuscript, we adopt the convention that whenever $d=2$, we will denote $H^\s(\RR^2), {H}^{\s,\rho}(\RR^2)$ simply as $H^\s, {H}^{\s,\rho}$, and similarly for their homogeneous counterparts and related spaces such as $\mathscr{S}(\RR^2)$, $\mathscr{S}'(\RR^2)$, etc. Whenever results hold for $d\geq1$, we will explicitly write $H^\s(\RR^d)$, $\Hdot^{\s,\rho}(\RR^d)$, $\mathscr{S}(\RR^d)$, etc., in their statements. However, we will always suppress the domain when performing estimates. Lastly, whenever the parameter, $d$, appears, it is understood that $d\geq1$ unless stated otherwise.
\end{Rmk}

\subsection{Multiplier Classes
}\label{sect:op:classes}In this section, we {identify a minimal set of assumptions that} define general classes of Fourier multiplier operators which characterize the regularity and growth properties of the multiplier operators $m(D)$, $p(D)$, $\om(D)$ and $\nu(D)$ that feature in the model \eqref{eq:dgsqg} and our main results. A central preoccupation of this article is in identifying the precise interrelation between these operators in establishing local well-posedness (in the sense of Hadamard). Indeed, we deconstruct the various effects arising from the dissipative operator, $m(D)$, and the operator $p(D)$ which multiplicatively modifies the constitutive relation between the advecting velocity field and the transported scalar in such a way that provides either regularization or de-regularization. To carry out this ``deconstructive analysis," we introduce the operators $\om(D)$ and $\nu(D)$. On an intuitive level, the role of these operators can be described as follows:
    \begin{align}
        \om(D)&\sim\ \text{logarithmically adjusts the regularity of the \textit{phase space}}\notag\\
        p(D)&\sim\ \text{represents the logarithmic modification of the \textit{constitutive law}}\notag\\
        m(D)&\sim\ \text{represents the \textit{dissipation mechanism}}\notag\\
        \nu(D)&\sim\ \text{captures the \textit{smoothing mechanism} associated to $m(D)$}\notag.
    \end{align}
{Ultimately, the multiplier $\om(D)$ enables an additional degree of flexibility for the local existence theory. In particular, it allows us to accommodate additional logarithmic losses of derivatives in the initial data without leaving the setting of borderline regularity; one of the main observations represented by our results is that this effect can be balanced by appropriately adjusting the regularizing or singularizing effects of the constitutive law or dissipation.} On the other hand, the multiplier $\nu(D)$ enables us to quantify the regularizing mechanism of the dissipation operator. The precise inter-relation between these operators that admit well-posedness of the corresponding active scalar transport system are stated in our main theorems in \cref{sect:main:results}. 

\subsubsection{Frequency weights associated to the regularity of the phase space} Let us first introduce the following properties:
    \begin{description}
        \item[\namedlabel{item:O1}{(O1)}]
            $\om \in C^1([0,\infty))$ is positive,  and satisfies  $\om'\geq0$.
        \item[\namedlabel{item:O2}{(O2)}] 
            There exists $C>0$ such that ${r\om'(r)}\le {C{\om(r)}}$, for all $r\ge 0$.
        \item[\namedlabel{item:O3}{(O3)}] There exists $C>0$ such that 
    \begin{align}\notag
       \om(r_1r_2)\leq C(\om(r_1)+\om(r_2)),\quad\text{for all}\ r_1,r_2\geq0.
    \end{align}
    \end{description}
From \ref{item:O1} and \ref{item:O2}, we deduce the following property: for any integer $k_1,k_2>0$, there exists $C>0$ such that for all $j\in\ZZ$
    \begin{align}\label{est:omega}
        \begin{split}
         C^{-1}\om(2^j)&\le \om(r)\le C\om(2^j).
         \end{split}
    \end{align}
for all $r\in[2^{j-k_1},2^{j+k_2}]$.

Now, let us see how \ref{item:O3} limits the growth at infinity provided that \ref{item:O1}, \ref{item:O2} are also satisfied. Indeed, observe that \ref{item:O2} implies
    \begin{align}\label{eq:om:upper}
        \om(r)\le \om(1)r^{C},\quad \text{for all}\ r\ge 1,
    \end{align}
where $C$ is the same constant from \ref{item:O2}. In particular, \ref{item:O2} limits growth at infinity to be at most algebraic. However, upon applying \ref{item:O3} with $r_1=r_2=\sqrt{r}$ and iterating, we obtain
    \begin{align}\notag
        \om(r)\leq 2C\om(r^{1/2})\leq(2C)^2\om(r^{1/4})\leq\dots\leq (2C)^n\om(r^{1/2^n})\leq {(2C)^n\om(1)r^{C/{2^n}}},
    \end{align}
for all $n>0$, where we have applied \eqref{eq:om:upper} in obtaining the final inequality. Hence, it follows that for any $\eps>0$, there exist $c_{\eps}>0$ such that
    \begin{align}\label{eq:omega:eps}
        \om(r)\le c_{\eps}r^{\eps}\quad\text{for all}\ r\geq1.
    \end{align}
Moreover
    \begin{align}\label{eq:omega:eps:bdd}
        \om(r)\leq \left(\sup_{0\leq r\leq 1}\om(r)\right)+c_\eps  r^{\eps}\leq C_\eps(1+r)^\eps,
    \end{align}
for all $r\geq0$, for some $C_\eps>0$.
    
Finally, we observe that if $\om_1,\om_2$ satisfies  \ref{item:O1}, \ref{item:O2}, and \ref{item:O3}, then the pointwise product $\om_1\om_2$ also satisfies \ref{item:O1}, \ref{item:O2}, and \ref{item:O3}.

Let us then define the class $\mathscr{M}_W$ by
    \begin{align}\label{def:MM}
        \mathscr{M}_W:=\left\{\frac{\om_a(D)}{\om_b(D)}:\om_a,\om_b\in C^1([0,\infty))\ \text{satisfy \ref{item:O1}, \ref{item:O2},\ \ref{item:O3}}\right\}.
    \end{align}
It follows by an application of Plancherel's theorem and \eqref{est:omega} that whenever $\om\in\mathscr{M}_W$, there exists $c,C>0$ such that
    \begin{align}\label{eq:bernstein:om}
	    c\om(2^j)\nrm{{\lpj}f}_{L^2(\RR^d)}\le \nrm{\om(D){\lpj}f}_{L^2(\RR^d)}\le C \om(2^j)\nrm{{\lpj}f}_{L^2(\RR^d)}.
	\end{align}
We will make use of \eqref{eq:bernstein:om} in a crucial way in the product and commutator estimates. { Moreover, observe that by \ref{item:O1}, we have $\om_b(r)\ge c$ for some positive constant $c$, hence \eqref{eq:omega:eps} holds for all $\om\in\mathscr{M}_W$.} Note that a prototypical example of $\om\in\mathscr{M}_W$ is given by $\om(r)=\left(\ln(e+r^2)\right)^p$, where $p$ is a nonzero real number.

\subsubsection{Multipliers associated to the constitutive law}
{For the multipliers, $p(D)$, associated to the constitutive law we will first introduce a slightly generalized class  $\tilde{\mathscr{M}}_W$, which we will then supplement with a property that limits the rate of decay at infinity. We recall that this property is what ultimately prevents the velocity from automatically being Lipschitz}. Let
    \begin{align}\label{def:MMt}
        \tilde{\mathscr{M}}_W:=\left\{\frac{p_a(D)}{p_b(D)}:p_a,p_b\in C^1([0,\infty))\ \text{satisfy \ref{item:O1}, \ref{item:O2}}\right\}.
    \end{align}
We point out that \eqref{est:omega} and \eqref{eq:bernstein:om} still hold for any $p\in\tilde{\mathscr{M}}_W$ since they only rely on the properties \ref{item:O1}, \ref{item:O2}.
{It follows, upon solving the differential inequality in \ref{item:O2}, that for $0<r_1\le r_2$ we have
    \[
    \ln \left(\frac{p_a(r_1+r_2)}{p_a(r_2)}\right)\le C\ln \left(1+\frac{r_1}{r_2}\right)\le C\ln2.
    \]
Hence
    \[
        p_a(r_1+r_2)\le 2^{C} p_a(r_2),\quad 0<r_1\leq r_2.
    \]
Similarly for $0<r_2\le r_1$, we have
    \[
        p_a(r_1+r_2)\le 2^C p_a(r_1).
    \] 
Upon combining the relations from both cases, we deduce that there exists $C>0$ such that 
    \begin{align}\label{eq:p:triangle:M_1}
       p_a(r_1+r_2)\leq C \max\{p_a(r_1), p_a(r_2)\}\le C(p_a(r_1)+p_a(r_2)),\quad \text{for all}\ r_1,r_2 \ge  0.
    \end{align}
}
Moreover, from \eqref{eq:p:triangle:M_1}, we may make additional use of \ref{item:O1} to show that
    \begin{align}\label{def:gen:tri:prod}
        p(|\xi|+|\eta|)\leq C\left(p_a(|\xi|)+p_a(|\eta|)\right),
    \end{align}
for some constant $C>0$, which depends on the value of $p_b(0)$.

Finally, we incorporate the decay restriction and define
  \begin{align}\label{def:MC}
        \mathscr{M}_C:=\left\{p(D)\in {\tilde{\mathscr{M}}_W} : \int_{1}^{\infty}\frac{p^{2}(r)}{r}dr=\infty\right\},
    \end{align}
We see that \eqref{def:MC} limits the symbol of the multiplier to decay at most logarithmically to a certain degree at infinity. A prototypical example is given by $p(r)=\left(\ln(e+r^2)\right)^{-1/2}$.

\subsubsection{Multipliers associated to the dissipation and its smoothing effect} The next class of multipliers we define are introduced to capture the smoothing effects of the dissipation represented by the multiplier operator $m(D)$. These effects are typically captured by an operator of the form
    \begin{align}\label{def:E}
    E_\nu^{\lam}f=e^{\lam \nu(D)}f,
    \end{align}
where $\nu(D)$ is a radial multiplier operator, i.e., $\nu(D)=\nu(|D|)$. We will refer to \eqref{def:E} as the \textit{smoothing operator induced by} $\nu(D)$.

Define a class of dissipation operators, denoted by $\mathscr{M}_D$, as
\begin{align}\label{def:MD}
    \mathscr{M}_D:=\left\{m(D): I+m(D)\in{\mathscr{M}}_W \right\}.
\end{align}
Let us assume that the scalar function, $\nu(r)$, associated to the multiplier operator $\nu(D)$ satisfies:
\begin{description}
    \item[\namedlabel{item:S1}{(S1)}]$\nu\in C^1([0,\infty))$ and $\nu(r), \nu'(r)\geq0$, for all $r\geq0$.
      \item[\namedlabel{item:S2}{(S2)}] There exists a constant $C>0$ such that $r\nu'(r)\le C$, whenever $r\ge 0$.
    \end{description}
Now given a multiplier $m(D)\in\mathscr{M}_D$, we associate its corresponding smoothing effect by introducing the multiplier class
\begin{align}\label{def:M:m}
    \mathscr{M}_{S}(m):=\left\{\nu(D) :\nu\, \text{satisfies \ref{item:S1}, \ref{item:S2} and}\, \nu(r)\le C(1+m(r)),\ \text{for some}\ C>0,\ \text{for all}\ r\ge 0  \right\}.
\end{align}

Finally, we define a class of functions which can be thought of as a logarithmic analog of the Gevrey classes. Given $\lam>0$, define
\begin{align}\label{def:chern:foias}
    \dot{E}^{\lam,\s}_{\nu,\om}:=\left\{f\in L^2 : \Sob{f}{\dot{E}^{\lam,\s}_{\nu,\om}}:= \Sob{E_\nu^{\lam}f}{\Hdot^{\s}_{\om}}<\infty \right\}.
\end{align}

\section{Statements of Main Results}\label{sect:main:results}
Our main local well-posedness results for the family \eqref{eq:dgsqg} are captured by the following two theorems, the first of which considers the case $\be \in (0,2]$ and the second which considers the endpoint case, $\be=0$, representing the 2D (mildly dissipative) Euler equation. To state these results, we recall that the class of multipliers that modify the constitutive law is denoted by $\mathscr{M}_C$ and is defined in \eqref{def:MC}.  The class of multipliers, $ \mathscr{M}_W$, {adjusts the regularity of the phase space and is defined in \eqref{def:MM}.} The class of multipliers characterizing the dissipation operator is denoted by $\mathscr{M}_D$ and defined in \eqref{def:MD}. Lastly, the class of multipliers, $\nu(D)$, that captures the smoothing effect associated to the dissipation operator, $m(D)$, is denoted by $\mathscr{M}_{S}(m)$, and is defined in \eqref{def:M:m}. { Henceforth, it will be convenient to introduce the notation
    \begin{align}\label{def:Md}
        m_1(D):=I+m(D).
    \end{align}
}
\begin{Thm}\label{thm:main:dgsqg}
Let $\be \in (0,2]$. Let $p(D)\in\mathscr{M}_C$, $m(D) \in \mathscr{M}_D$, and $\om(D)\in \mathscr{M}_W$, where $p=p_ap_b^{-1}$. Suppose there exists $\gamma<1$ such that
    \begin{align}\label{cond:main}
       \sup_{y>0}\left\{\frac{1}{m^{\gamma}_1(y)} \left(\int_{0}^{y}\frac{r(p^2(y)+p^2(r))}{(1+r^2){\om^2(r)}}dr\right)^{\frac{1}{2}},\quad \frac{p_a(y)\om_b(y)}{m^{\gamma}_1(y)}\right\}<\infty.
    \end{align}
Then for each $\tht_{0}\in H^{1+\be}_{\om}$, there exists a positive $T=T(\Sob{\tht_{0}}{H^{1+\be}_\om})$ and a unique function, $\tht(\cdotp;\tht_0)$ satisfying \eqref{eq:dgsqg}, such that
    \[
        \tht \in C([0,T];H^{1+\be}_{\om})\cap L^{2}(0,T; H^{1+\be}_{\om {m}^{1/2}}).
    \]
Moreover, the data-to-solution map, $\Phi$, defined by
	{\par\nobreak\noindent}\begin{align}\label{def:flowmap:dsqg}
	    \Phi: H^{1+\be}_{\om}\goesto  \bigcup_{T>0}C([0,T]; H^{1+\be}_\om),\quad \tht_{0}\mapsto\tht(\cdotp;\tht_0),
	\end{align}
is continuous. Lastly, if $\nu(D)\in \mathscr{M}_{S}(m)$, then there exists $\lam>0$ such that $\tht$ satisfies
        \begin{align}\label{est:dsqg:smoothing}
        \sup_{0\le t\le T}\Sob{E^{\lam t}_{\nu}\tht(t)}{\Hdot^{ 1+\be}_{\om}}\le C(T,\Sob{\tht_0}{H^{1+\be}_{\om}}),
        \end{align}
for a positive constant $C$.
\end{Thm}

For the case $\be=0$, we have the following.

\begin{Thm}\label{thm:main:deuler} 
Suppose that $p(D)\in\mathscr{M}_C$, $m(D)\in \mathscr{M}_D$ and $\om(D)\in \mathscr{M}_W$ satisfy \eqref{cond:main}, where $p=p_ap_b^{-1}$. For each $\tht_{0}\in H^{1}_{\om}\cap \Hdot^{-1}_\om$, there exists a positive $T=T(\Sob{\tht_{0}}{H^{1}_\om\cap\Hdot^{-1}_\om})$ and a unique solution, $\tht(\cdotp;\tht_0)$, of \eqref{eq:dgsqg} when $\be=0$, such that
	\[
	    \tht \in C([0,T];H^{1}_{\om}\cap \Hdot^{-1}_\om)\cap L^{2}(0,T; H^{1}_{\om {m}^{1/2}}).
    \]
Moreover, if $\nu(D)\in \mathscr{M}_{S}(m)$, then there exists $\lam>0$ such that $\tht$ satisfies
        \begin{align}\label{est:deuler:smoothing}
        \sup_{0\le t\le T}\Sob{E^{\lam t}_{\nu}\tht(t)}{\Hdot^{1}_{\om}}\le C(T,\Sob{\tht_0}{H^{1}_{\om}}),
        \end{align}
for a positive constant $C$. Lastly, the data-to-solution map, $\Phi$, defined by
	\begin{align}\label{def:flowmap:deuler}
	    \Phi: H^{1}_{\om}\cap \Hdot^{-1}_\om\goesto  \bigcup_{T>0}C([0,T]; H^{1}_\om \cap\Hdot^{-1}_\om),\quad \tht_{0}\mapsto\tht(t;\tht_0),
	\end{align}
is continuous. 
\end{Thm}

\begin{Rmk}\label{rmk:lazar}
The work \cite{LazarXue2019} considers \eqref{eq:dgsqg}, where $m(D)=\Lam^\be$, for $\be\in(0,1)$, and $p(D)$ is given by a multiplier with logarithmic-type growth. The conditions satisfied by $p(D)$ are very similar to those imposed here by $\mathscr{M}_W$, but with a few technical differences. {One notable difference, however, is our algebraic-type condition \textbf{(O3)}, which in contrast to the analytic-type conditions in \cite{LazarXue2019}, do not impose higher-order regularity constraints on our multipliers}.  We emphasize, however, that our results are first and foremost concerned with local existence and stability-type estimates in borderline topologies, particularly in the full parameter range $\be\in[0,2]$ of the gSQG family, whereas \cite{LazarXue2019} focuses on the issue of global regularity. In this regard, the results of the present article  complement those in \cite{LazarXue2019}. 

Nevertheless, it would be interesting if the results of \cite{LazarXue2019} could be established for the range $\be\in(1,2)$ or if global existence of weak solutions and their eventual regularity, proved in \cite{LazarXue2019}, can be extended to the class of models addressed by \cref{thm:main:dgsqg}, \cref{thm:main:deuler}. We refer the reader to the notable recent works \cite{ChaeJeongNaOh2023a, ChaeJeongOh2023b} on the $\be=2$ endpoint; see also \cref{rmk:okhitani} for further discussion.
\end{Rmk}

Next, we present a selection of choices for $m,p,\om,\nu$ to demonstrate that \cref{thm:main:dgsqg} and \cref{thm:main:deuler} contain various non-trivial and interesting consequences. The following result establishes well-posedness under a logarithmic form of singularity and dissipation.

\begin{Cor}\label{cor:l:main:dgsqg}
	Let $p(D)$ and $m(D)$ be defined by
	\begin{align}\label{def:log:pm}
	    m(D)=\ln(I-\Delta)^{{\mu}},\quad p(D)=\ln(e-\Delta)^{\til{\mu}},\quad \til{\mu}\ge-1/2. 
	\end{align}
	Let $\be \in (0,2]$ and ${{\mu}}>\til{\mu}+1/2$. For each $\tht_{0}\in H^{1+\be}$, there exists a positive $T=T(\Sob{\tht_{0}}{H^{1+\be}})$ and a unique solution, $\tht$, of \eqref{eq:dgsqg}, such that
	\[\tht \in C([0,T];H^{1+\be})\cap L^{2}(0,T; H^{1+\be,\frac{\mu}{2}}).\]
	and the data-to-solution map, $\Phi$, such that
	{\par\nobreak\noindent}\begin{align}\label{cor:l:def:flowmap:dsqg}
	    \Phi: H^{1+\be}\goesto  \bigcup_{T>0}C([0,T]; H^{1+\be}),\quad \tht_{0}\mapsto\tht(t;\tht_0),
	\end{align}
is continuous. Furthermore, there exists $\lam>0$ such that $\tht$ satisfies
        \begin{align}\label{cor:l:est:dsqg:smoothing}
        \sup_{0\le t\le T}\Sob{E^{\lam t}_{(\ln(I-\De))^\al}\tht(t)}{\Hdot^{ 1+\be}}\le C(T,\Sob{\tht_0}{H^{1+\be}}),
        \end{align}
for any $0<\al \le \min\{1,\mu\}$, and a positive constant $C$. In particular, if $\mu\ge 1$, we have 
    \[
        \sup_{0\le t\le T}\Sob{\tht(t)}{\Hdot^{ 1+\be+\lam t}}\le C(T,\Sob{\tht_0}{H^{1+\be}})
    \]
\end{Cor}

The corresponding result in the case of $\be=0$ is as follows.
\begin{Cor}\label{cor:l:main:deuler}
	Let $p(D)$ and $m(D)$ be as in \eqref{def:log:pm} with $\mu>\til{\mu}+1/2$. For each $\tht_{0}\in H^{1}\cap \Hdot^{-1}$, there exists a positive $T=T(\Sob{\tht_{0}}{H^{1}\cap \Hdot^{-1}})$ and a unique solution, $\tht$, of \eqref{eq:dgsqg} corresponding to $\be=0$, such that
	\[\tht \in C([0,T];H^{1}\cap \Hdot^{-1})\cap L^{2}(0,T; H^{1,\frac{\mu}2}).\]
	and the data-to-solution map, $\Phi$, such that
	{\par\nobreak\noindent}\begin{align}\label{cor:l:def:flowmap:deuler}
	    \Phi: H^{1}\cap \Hdot^{-1}\goesto  \bigcup_{T>0}C([0,T]; H^{1}\cap \Hdot^{-1}),\quad \tht_{0}\mapsto\tht(t;\tht_0),
	\end{align}
is continuous. Furthermore, there exists $\lam>0$ such that $\tht$ satisfies
        \begin{align}\label{cor:l:est:deuler:smoothing}
        \sup_{0\le t\le T}\Sob{E^{\lam t}_{(\ln(I-\De))^\al} \tht(t)}{\Hdot^{1}} \le C(T,\Sob{\tht_0}{H^{1}}),
        \end{align}
for any $0<\al \le \min\{1,\mu\}$, and a positive constant $C$. In particular, if $\mu\ge 1$, we have 
    \begin{align}\label{cor:l:est:deuler:smoothing:alpha:1}
        \sup_{0\le t\le T}\Sob{\tht(t)}{\Hdot^{ 1+\lam t}}\le C(T,\Sob{\tht_0}{H^{1}})
    \end{align}
\end{Cor}

\begin{Rmk}\label{rmk:okhitani}
In the recent works \cite{ChaeJeongNaOh2023a, ChaeJeongOh2023b}, the well-posedness and ill-posedness of several models within the scope of \eqref{eq:dgsqg} were studied, but ultimately complementary to the class of models studied in the present article.

In \cite{ChaeJeongNaOh2023a}, the Okhitani model, i.e., $\be=2$, $p(D)=\ln(10+\Lam)$, and $m(D)\equiv0$ in \eqref{eq:dgsqg}, was studied. In addition to local well-posedness, losing estimates in $H^{s(t)>4}$, where $s(t)$ is a decreasing function, were shown to be a fundamental feature of solutions under the evolution of the system. The derivation of the Okhitani model as a limit of regularized models in a time re-scaled sense was also subsequently justified. Local well-posedness in a fixed Sobolev space $H^{s>4}$, i.e., \textit{without losing estimates}, was then established in the presence of mild dissipation $m(D)=(\ln(10+\Lam))^{\mu\geq1}$. On the other hand, in \cite{ChaeJeongOh2023b}, ill-posedness in $H^{s>3}$ in the form of norm inflation and non-existence was established for the mildly dissipative Okhitani model in the regime where $p(D)=(\ln(10+\Lam))^{\til{\mu}>0}$, $m(D)=(\ln(10+\Lam))^{\mu}$, where $\mu<\til{\mu}$.

In contrast, the results proved in the present article identify conditions that guarantee local well-posedness, but specifically in the borderline regularity setting, $H^{1+\be}$. In the particular case of the mildly regularized Okhitani model, i.e.,  $\be=2$, $p(D)=(\ln(10+\Lam))^{\til{\mu}=1}$, $m(D)=(\ln(10+\Lam))^{\mu}$, where $\mu>\til{\mu}$, in \eqref{eq:dgsqg}, \cref{cor:l:main:dgsqg} establishes local well-posedness in $H^{3}$, under the proviso that $\mu>\til{\mu}+1/2=3/2$. Thus, the results established here complement the well-posedness and ill-posedness results established in \cite{ChaeJeongNaOh2023a, ChaeJeongOh2023b}. 
\end{Rmk}

\begin{Rmk}\label{rmk:ill:posed}
Observe that when $\be=1$, we obtain the local well-posedness (in the Hadamard sense) of the mildly dissipative SQG equation. This complements the recent ill-posedness results in borderline (critical) Sobolev spaces for the inviscid SQG equation obtained in \cite{CordobaZoroa-Martinez2021, JeongKim2021}, as well as the local well-posedness result for the logarithmically regularized inviscid SQG equation obtained in \cite{ChaeWu2012}. We emphasize once again that our results extend beyond the SQG, for $1<\be\leq2$, where the $\be=2$ endpoint is suitably modified. Thus, the range $\mu>\til{\mu}+1/2$ identified in \cref{cor:l:main:dgsqg} draws out the putatively sharp threshold for well-posedness at borderline regularity beyond the SQG case.
In particular, the problem of whether \eqref{eq:dgsqg} in the setting of \eqref{cor:l:main:dgsqg}, \eqref{cor:l:main:deuler} is well-posed or not in the borderline Sobolev spaces when $0<\mu\leq\til{\mu}+1/2$ is an open consideration even when $\be\in[0,1]$, i.e., including the Euler endpoint $\be=0$.
\end{Rmk}

 We also study the initial value problem for \eqref{eq:dgsqg} in a log-Sobolev borderline space. The multiplier operators $m, p$ in this case are assumed to be of the form of an iterated logarithm. The results are stated below.

\begin{Cor}\label{cor:ll:main:dgsqg}
	Let $p(D)$ and $m(D)$ be defined by
	\begin{align}\label{def:loglog:pm}
	    m(D)&=\ln(I+\ln(I-\Delta))^{\mu},\notag\\
	p(D)&=\ln(e+\ln(I-\Delta))^{\til{\mu}},\quad \til{\mu}\ge -1/2. 
	\end{align}
	Let $\be \in (0,2]$ and $\mu>\til{\mu}+1/2$. For each $\tht_{0}\in H^{1+\be,\frac{1}{2}}$, there exists a positive $T=T(\Sob{\tht_{0}}{H^{1+\be,\frac{1}{2}}})$ and a unique solution, $\tht$, of \eqref{eq:dgsqg}, such that
	   \begin{align}\notag
	        \tht \in C([0,T];H^{1+\be,\frac{1}{2}})\cap L^{2}(0,T; H^{1+\be,\frac{1}{2},\frac{\mu}{2}}),
	   \end{align}
	and the corresponding data-to-solution map
	{\par\nobreak\noindent}\begin{align}\label{cor:ll:def:flowmap:dsqg}
	    \Phi: H^{1+\be,\frac{1}{2}}\goesto  \bigcup_{T>0}C([0,T]; H^{1+\be,\frac{1}{2}}),\quad \tht_{0}\mapsto\tht(t;\tht_0),
	\end{align}
is continuous. Furthermore, there exists $\lam>0$ such that $\tht$ satisfies
        \begin{align}\label{cor:ll:est:dsqg:smoothing}
        \sup_{0\le t\le T}\Sob{E^{\lam t}_{(\ln(I+\ln(I-\Delta)))^{\alpha}}\tht(t)}{\Hdot^{ 1+\be, \frac{1}{2}}}\le C(T,\Sob{\tht_0}{H^{1+\be,\frac{1}{2}}}),
        \end{align}
for any $0<\al \le \min\{1,\mu\}$, and a positive constant $C$. In particular, if $\mu\ge 1$, we have 
    \begin{align}\notag
    \sup_{0\le t\le T}\Sob{\tht(t)}{\Hdot^{ 1+\be,\frac{1}{2}+\lam t}}\le C(T,\Sob{\tht_0}{H^{1+\be,\frac{1}{2}}}).
    \end{align}
\end{Cor}
The corresponding result in the case of $\be=0$ is as follows.
\begin{Cor}\label{cor:ll:main:deuler}
	Let $p(D)$ and $m(D)$ be as in \eqref{def:loglog:pm} with $\mu>\til{\mu}+1/2$. For each $\tht_{0}\in H^{1,\frac{1}{2}}\cap \Hdot^{-1,\frac{1}{2}}$, there exists a positive $T=T(\Sob{\tht_{0}}{H^{1,\frac{1}{2}}\cap \Hdot^{-1,\frac{1}{2}}})$ and a unique solution, $\tht$, of \eqref{eq:dgsqg} ($\be=0$), such that
	\[\tht \in C([0,T];H^{1,\frac{1}{2}}\cap \Hdot^{-1,\frac{1}{2}})\cap L^{2}(0,T; H^{1,\frac{1}{2},\frac{\mu}{2}}).\]
and the data-to-solution map, $\Phi$, such that
	\begin{align}\label{cor:ll:def:flowmap:deuler}
	    \Phi: H^{1,\frac{1}{2}}\cap \Hdot^{-1,\frac{1}{2}}\goesto  \bigcup_{T>0}C([0,T]; H^{1,\frac{1}{2}}\cap \Hdot^{-1,\frac{1}{2}}),\quad \tht_{0}\mapsto\tht(t;\tht_0),
	\end{align}
is continuous. Furthermore, there exists $\lam>0$ such that $\tht$ satisfies
        \begin{align}\label{cor:ll:est:deuler:smoothing}
        \sup_{0\le t\le T}\Sob{E^{\lam t}_{ (\ln(I+\ln(I-\Delta)))^{\alpha}}\tht(t)}{\Hdot^{ 1, \frac{1}{2}}}\le C(T,\Sob{\tht_0}{H^{1,\frac{1}{2}}}),
        \end{align}
for any $0<\al \le \min\{1,\mu\}$, and a positive constant $C$. In particular, if $\mu\ge 1$, we have 
    \begin{align}\notag
    \sup_{0\le t\le T}\Sob{\tht(t)}{\Hdot^{ 1,\frac{1}{2}+\lam t}}\le C(T,\Sob{\tht_0}{H^{1,\frac{1}{2}}}).
    \end{align}
\end{Cor}

The main theorem, \cref{thm:main:dgsqg} will be proved in \cref{sect:wellposed}. Before proceeding to develop its proof, we will discuss an application of \cref{thm:main:dgsqg} to the so-called Euler endpoint, $\be=0$.

\section{Application to the Mildly Dissipative 2D Euler Equation: Global regularity}
A natural consideration in light of the above local well-posedness results is the issue of global regularity of the considered models. In the endpoint case, $\be=0$, which represents the dissipatively perturbed Euler equations, we establish global regularity of solutions as an application of the smoothing effect conferred by the mild dissipation and the existence of a maximum principle. In particular, we prove the following result.

\begin{Thm}\label{thm:deuler:global}
	Given $\tht_0 \in H^{1}\cap \Hdot^{-1}$, consider the initial value problem \eqref{eq:dgsqg} for $\be=0$, where $p(D)=(\ln(I+\ln(I-\De))^\gam$, $\gam\in[0,1]$, and  $m(D)=\ln(I-\Delta)$, i.e.
 \begin{align}\label{eq:Euler:endpoint}
	        \partial_{t}\theta + \ln(I-\Delta)\tht+u \cdot \nabla \theta=0, \quad
	    u=\nabla^{\perp}{\psi},\quad \Delta {\psi}=(\ln(I+\ln(I-\De))^\gam\theta.
 \end{align}
 Then the unique solution satisfies
	\begin{align*}
	    \tht \in C([0,T]; H^{1}\cap\Hdot^{-1}),\quad \sup_{0\leq t\leq T}\Sob{\tht(t)}{\Hdot^{1+\lam t}}<\infty,
	\end{align*}
for all $T>0$. In particular, \eqref{eq:Euler:endpoint} is globally well-posed in $H^1\cap \Hdot^{-1}$ in the Hadamard sense.
\end{Thm}

\begin{Rmk}\label{rmk:osgood}
\cref{thm:deuler:global} is consistent with previous global well-posedness results found in \cite{Elgindi2014, ChaeConstantinWu2011, DabkowskiKiselevSilvestreVicol2014} in {non-borderline functional settings} for the so-called slightly supercritical Euler equations, i.e., $m(D)\equiv0$, but $p(D)=(\ln(I+\ln(I-\De)))^\gam$, where $\gam\in[0,1]$. Indeed, ill-posedness in the form of norm inflation or non-existence can occur in borderline topologies \cite{BourgainLi2015, ElgindiMasmoudi2020}. Thus, the improvement in \cref{thm:deuler:global} is that global well-posedness holds in the borderline topology $H^1\cap H^{-1}$ in the presence of logarithmic dissipation of order $1$.
\end{Rmk}

The first step is to establish a maximum principle. To do so, let us denote by
    \[
        {L}=\ln(I-\Delta).
    \]
We will consider an alternative representation of ${L}$ via the heat semigroup. This is accomplished through the following elementary identity:
    \begin{align}\label{eq:unity:ident}
        \ln(1+\lam)=\int_0^\infty (1-e^{-s\lam})e^{-s}\frac{ds}s
    \end{align}
Indeed, we see that
    \begin{align}
        \ln(1+\lam)=\int_0^1\frac{\lam}{1+\lam\tau}d\tau=\int_0^1\int_0^\infty \lam e^{-(1+\lam\tau)s}dsd\tau=\int_0^\infty \left(\int_0^1\lam e^{-\lam s\tau}d\tau\right) e^{-s}ds=\int_0^\infty (1-e^{-s\lam})e^{-s}\frac{ds}s.\notag
    \end{align}
Hence
    \begin{align}\label{eq:L1}
        {L}f(x)=\int_0^\infty(f(x)-e^{s\De}f(x))e^{-s}\frac{ds}s.
    \end{align}
{ Let $\mathscr{H}(t,\cdotp)$ denote the heat kernel corresponding to the heat semigroup $e^{t\De}$}. We then claim that the following inequality holds.
 
\begin{Lem}\label{lem:convex} Given $\Phi\in C^1(\RR)$ convex, we have
    \begin{align}\label{eq:convex}
        \Phi'(f){L}f-{L}\Phi(f)\geq0.
    \end{align}
\end{Lem}

\begin{proof}
We make use of \eqref{eq:L1}. Indeed, observe that
    \begin{align}
        \Phi&'(f(x))({L}f)(x)-({L}\Phi(f))(x)\notag\\
        =&\int_0^\infty\left[\Phi'(f(x))f(x)-\Phi'(f(x))e^{s\De}f(x)-\left(\Phi(f(x))-(e^{s\De}\Phi(f))(x))\right)\right]e^{-s}\frac{ds}{s}\notag\\
        =&\int_0^\infty\left[\Phi'(f(x))f(x)-\int_{\RR^d}\Phi'(f(x)){\mathscr{H}}(s,x-y)f(y)dy-\Phi(f(x))+\int_{\RR^d}{\mathscr{H}}(s,x-y)\Phi(f(y))dy\right]e^{-s}\frac{ds}{s}\notag\\
        =&\int_0^\infty\left[\int_{\RR^d}{\mathscr{H}}(s,x-y)\left(\Phi(f(y))-\Phi(f(x))-\Phi'(f(x))\left(f(y)-f(x)\right)\right)dy\right]e^{-s}\frac{ds}{s}.\notag
    \end{align}
By the convexity of $\Phi$, it follows that
    \[
        \Phi(f(y))-\Phi(f(x))-\Phi'(f(x))\left(f(y)-f(x)\right)\geq0.
    \]
We may now deduce \eqref{eq:convex}.
\end{proof}

Using a standard argument and applying \cref{lem:convex}, we establish a maximum principle for the case \eqref{eq:Euler:endpoint}, where $m(D)=\ln(I-\De)$. This result is stated in \cref{lem:max:principle} whose proof is provided in \cref{app:proof:maxp}.  In this setting, we interpret $\tht$ as the scalar vorticity of a two-dimensional incompressible fluid.

\begin{Lem}\label{lem:max:principle}
Given a sufficiently smooth solution of \eqref{eq:Euler:endpoint} on the time interval $[0,T]$ such that $\tht_0\in L^2\cap L^\infty$, there exists $C>0$ such that
    \begin{align}
      \sup_{0\leq t\leq T}\Sob{\tht(t)}{L^\infty}\leq C(\Sob{\tht_0}{L^2},\Sob{\tht_0}{L^\infty}).
    \end{align}
\end{Lem}

Finally, we prove global well-posedness of the initial value problem \eqref{eq:Euler:endpoint} posed in $H^1\cap\Hdot^{-1}$.

\begin{proof}[Proof sketch of \cref{thm:deuler:global}]
Given $\tht_0\in H^1$, \cref{cor:l:main:deuler} yields a local solution $\tht\in C([0,T_0];H^1\cap\Hdot^{-1})$, for some $T_0>0$, depending only on $\Sob{\tht_0}{H^1}$.  { By the linear-in-time gain in Sobolev regularity asserted by \eqref{cor:l:est:deuler:smoothing:alpha:1}} (applied with $\al=1$), it follows that $\tht(T_0/2)\in \Hdot^{1+\lam T_0/2}$, for some $\lam>0$. In particular $\tht(T_0/2)\in L^\infty$. By the \cref{lem:max:principle}, $C_\infty:=\sup_{T_0/2\leq t\leq T_0}\Sob{\tht(t)}{L^\infty}<\infty$, depends only on the value at $t=T_0/2$, and by the basic energy inequality $C_2:=\sup_{0\leq t\leq T_0}\Sob{\tht(t)}{L^2}\leq\Sob{\tht_0}{L^2}$.

We now consider the equation over the time interval $[T_0/2,T_0]$. Applying $\nabla$ on \eqref{eq:Euler:endpoint}, and then taking $L^2$ inner product with $\nabla \tht$, we obtain 
\begin{align*}
    \frac{d}{dt}\Sob{\nabla\tht}{L^2}^2+\Sob{{L}(D)^{1/2}\nabla \tht}{L^2}^2\le C\Sob{\nabla u}{L^\infty}\Sob{\nabla \tht}{L^2}^2.
\end{align*}
We recall a simple variation of a classical borderline Sobolev inequality that was established in \cite{Elgindi2014, ChaeConstantinWu2012}, which implies 
    \begin{align}\label{est:classical}
        \Sob{\nabla u}{L^\infty}\le C_2+C_\infty\left(1+\ln\left(1+\Sob{\tht}{H^{1+\de}}\right)\right)\left(\ln\left(1+\ln(1+\Sob{\tht}{H^{1+\de}}\right)\right)^\gam.
    \end{align}
for any $\de\in(0,1)$, over the time interval $[T_0/2,T_0]$. Choose $\de_0=\lam T_0/2$ and $s_0=1+\de_0$. Then
    \begin{align}\label{est:H1:open}
         \frac{d}{dt}\Sob{\nabla\tht}{L^2}^2\le \left[C_2+C_\infty\left(1+\ln\left(1+\Sob{\tht}{H^{s_0}}\right)\right)\left(\ln\left(1+\ln(1+\Sob{\tht}{H^{s_0}}\right)\right)^\gam\right]\Sob{\nabla\tht}{L^2}^2,
    \end{align}
holds for $t\in[T_0/2,T_0]$.

Applying $(I-\De)^{s_0/2}$ to \eqref{eq:Euler:endpoint}, and then taking $L^2$ inner product with $(I-\De)^{s_0/2} \tht$, and applying the Kato-Ponce commutator inequality (see \cite{KatoPonce1988}), we obtain 
\begin{align*}
    \frac{d}{dt}\Sob{\tht}{H^{s_0}}^2+\Sob{{L}(D)^{1/2} \tht}{H^{s_0}}^2\le C\Sob{\nabla u}{L^\infty}\Sob{\tht}{H^{s_0}}^2.
\end{align*}
By another application of \eqref{est:classical} with $\de=\de_0$, followed by the Gronwall inequality, we may deduce that
    \[
        \sup_{T_0/2\leq t\leq T_0}\Sob{\tht(t)}{H^{s_0}}\leq C_0(C_2,C_\infty, T_0),
    \]
for some constant $C_0$ depending only on $C_2, C_\infty, T_0$. This bound, in turn, allows one to close the $H^1$-estimate in \eqref{est:H1:open} over the time interval $[T_0/2,T_0]$. A standard continuation argument now applies to extend the solution globally.
\end{proof}

\section{Product and Commutator estimates}
To estimate the advective nonlinearity, we will make use of product and commutator estimates involving logarithmic and polynomial differential operators. The first result is \cref{thm:prod}, which establishes a product estimate localized to dyadic shells in frequency. These results are of a general nature, independent of the equation \eqref{eq:dgsqg} and the structure of its advecting velocity. Similar estimates were developed in \cite{ColombiniDelSantoFanelliMetivier2015} in the setting of logarithmic Besov spaces. By comparison, our estimates are performed entirely in the Sobolev setting, but we expand the ``regularity parameter" to encompass iterated logarithms as well; the proofs of such estimates are relegated to \cref{app:prod}. 

Next, we establish three commutator estimates: \cref{lem:commutator2}, \cref{lem:commutator3}, and \cref{lem:commutator4}. The estimates established in \cref{lem:commutator2} and \cref{lem:commutator3} involve the operator defining the constitutive law in \eqref{eq:dgsqg}, whereas \cref{lem:commutator4} studies commutators involving differential operators that will be used to capture the smoothing mechanism conferred by our dissipation; these commutators are directly inspired by the study of the Gevrey regularity and may be viewed as an expanded development of the classical Gevrey-norm approach introduced by Foias and Temam \cite{FoiasTemam1989}.

We recall from \eqref{eq:dgsqg} that the constitutive law is given by  $v=-\nabla^{\perp}{\Lam}^{\be-2}p(D)\tht$. In the regime of $\be>1$, this operator corresponds to an integral more singular than Riesz transform and constitutes the fundamental difficulty in this regime.  In \cref{lem:commutator2}, we establish a commutator estimate for the trilinear term in a non-localized form. We  obtain sharper estimates in \cref{lem:commutator3} under additional localizing assumptions. The proofs of \cref{lem:commutator2} and \cref{lem:commutator3} are based on an approach similar to the one used in \cite{JollyKumarMartinez2020b}, which crucially exploits an elementary convexity estimate (see \eqref{est:elem:convex}) and {a finer analysis by} Littlewood-Paley decomposition. {As such, the estimates obtained below and are of independent interest to the main results of this article.}

\subsection{Product estimates}\label{sect:prod}
The first result that we state is the main product estimate which generalizes \cref{lem:Sobolev}. As mentioned earlier, the proof will be supplied in  \cref{app:prod}, but we invoke it to prove the main commutator estimate in \cref{sect:comm}. We are now ready to state our main product estimate.

\begin{Thm}\label{thm:prod}
Let $d\ge 2$. Suppose that $s,\bar{s}\in\RR$ are given such that $s, \bar{s}\leq d/2$ and $s+\bar{s}>0$.
Let $\om, \om_{\ell},\tilde{\om}_\ell\in\mathscr{M}_W$, for $\ell=1,2,3$. Assume that $\Gam_\ell:[0,\infty)\goesto[0,\infty)$, where $\ell=1,2,3$, are functions such that for all $y\geq0$
    \begin{align}\label{cond:om:Gam:a}
        \begin{split}
        &\frac{\om(y)}{{\tilde{\om}_1(y)}}\left({\mathbbm{1}}_{(-\infty,d/2)}(s)\int_0^1\frac{r^{d-2s-1}}{\om_1^2(y r)}dr+{\mathbbm{1}}_{[d/2,\infty)}(s)\int_{0}^{y}\frac{r^{d-1}}{(1+r^2)^{d/2}\om_{1}^{2}(r)}dr\right)^{1/2}\le C_1\Gam_1(y),\\
        & \frac{\om(y)}{{\tilde{\om}_2(y)}}\left({\mathbbm{1}}_{(-\infty,d/2)}(\bar{s})\int_0^1\frac{r^{d-2\bar{s}-1}}{\om_2^2(y r)}dr+{\mathbbm{1}}_{[d/2,\infty)}(\bar{s})\int_{0}^{y}\frac{r^{d-1}}{(1+r^2)^{d/2}\om_{2}^{2}(r)}dr\right)^{1/2}\le C_2\Gam_2(y),
        \end{split}
    \end{align}
and
    \begin{align}\label{cond:om:Gam:b}
        \frac{\om(y)}{{\om_3(y)}{\tilde{\om}_3(y)}}\le C_3\Gam_{3}(y),
    \end{align}
for some $C_1,C_2,C_3>0$. Then there exists $C>0$ and $\{c_j\}\in\ell^2(\ZZ)$ satisfying $\Sob{\{c_j\}}{\ell^2}\leq1$ such that the following inequality holds
    \begin{align}\label{est:product:Thm}
       \Sob{\triangle_j(fg)}{L^2}
       &\leq
         Cc_j\om(2^j)^{-1}{2^{-(s+\bar{s}-d/2)j}}\left( {\Gam_1(2^j)}\pi^{s,\bar{s}}_{\om_1,\til{\om}_1}(f,g)+{\Gam_2(2^j)}\pi^{\bar{s},s}_{\om_2,\til{\om}_2}(g,f)+{\Gam_3(2^j)}\rho^{s,\bar{s}}_{\om_3,\til{\om}_3}(f,g)\right),
    \end{align}
where for $k=1,2$, we have
    \begin{align}\label{def:pi:rho}
        \pi^{r,t}_{\vr,\til{\vr}}(f,g):=\begin{cases}\Sob{f}{\Hdot^{r}_{\vr}}\Sob{g}{\Hdot^{t}_{\til{\vr}}},&r<d/2\\
        \Sob{f}{H^{r}_{\vr}}\Sob{g}{\Hdot^{t}_{\til{\vr}}},&r=d/2
        \end{cases},\qquad \rho^{r,t}_{\vr,\til{\vr}}(f,g):=\Sob{f}{\Hdot^{r}_{\vr}}\Sob{g}{\Hdot^{t}_{\til{\vr}}}.
    \end{align}
\end{Thm}

Upon summing \cref{thm:prod} in $j$, we obtain a product estimate in a modified Sobolev space.

\begin{Cor}\label{cor:prod}
Under the assumptions of \cref{thm:prod} with $\Gam=\max\{\Gam_1, \Gam_2, \Gam_3\}$, it follows that
  \begin{align}\label{est:product:Cor}
       \Sob{fg}{\Hdot^{s+\bar{s}-d/2}_{\om \Gam^{-1}}}\leq
        C\left( \pi^{s,\bar{s}}_{\om_1,\til{\om}_1}(f,g)+\pi^{\bar{s},s}_{\om_2,\til{\om}_2}(g,f)+\rho^{s,\bar{s}}_{\om_3,\til{\om}_3}(f,g)\right)
    \end{align}
\end{Cor}

\subsection{Commutator estimates for logarithmic-type multipliers}\label{sect:comm} 

Now we state and prove our main commutator estimates.  Note that we will make use of the usual commutator bracket notation:
    \[
        [A,B]=AB-BA.
    \]
The commutator estimates that we establish will directly involve multipliers associated to either the phase space, characterized by the class $\mathscr{M}_W$ defined in \eqref{def:MM}, or to its slightly larger counterpart $\tilde{\mathscr{M}}_W$ defined in \eqref{def:MMt} that eventually corresponds to multipliers associated to the constitutive law.

\begin{Lem}\label{lem:commutator2}
Let $s\in[0,1)$, $\de, \eps \in (0,1]$ be given such that $\eps+s\le 1$ and $\de<\eps$. Let $p\in\tilde{\mathscr{M}}_W$ be represented as $p=p_ap_b^{-1}$. Then there exists a constant $C>0$ such that
    \begin{align}\label{est:commutator2:permute}
       |\lb [{\Lam}^{-s}p(D)\bdy_\ell,g]f,h\rb|
       \leq C\Sob{g}{{H}^{2-s-\de}}\left(\Sob{p_a(D)f}{\Hdot^{\eps}}\Sob{h}{L^2}+\Sob{p_a(D)h}{\Hdot^{\eps}}\Sob{f}{L^2}\right).
    \end{align}
\end{Lem}

In order to prove \cref{lem:commutator2}, we will make use of an elementary convexity estimate that was proved in \cite{JollyKumarMartinez2020b}. We briefly recall the inequality here: let $\varphi,\vartheta\in\RR^2$ such that $|\vartheta|=1$. Then for all $s\in[0,1)$, there exists a constant $C>0$, independent of $\varphi,\vartheta$, such that
    \begin{align}\label{est:elem:convex}
    \int_0^1\frac{1}{|\varphi+\tau\vartheta|^s}d\tau\leq C.
    \end{align}
We are now ready to prove \cref{lem:commutator2}.

\begin{proof}[Proof of \cref{lem:commutator2}]
It will be convenient to define the following functional:
    \begin{align}\label{def:L:functional:1}
        \mathcal{L}(f,g,h):=\iint m(\xi,\eta)\hat{f}(\xi-\eta)\hat{g}(\eta)\overline{\hat{h}(\xi)}d\eta d\xi,
    \end{align}
where
    \begin{align*}
        m(\xi,\eta)=|\xi|^{-s}p(|\xi|)\xi_\ell-|\xi-\eta|^{-s}p(|\xi-\eta|)(\xi-\eta)_\ell.
    \end{align*}
By Plancherel's theorem, we see that
    \[
         \mathcal{L}(f,g,h)=\lb [{\Lam}^{-s}p(D)\bdy_\ell,g]f,h\rb.
    \]
Denote the parametrization of the line segment starting at $\xi-\eta$ and ending at $\xi$ by
    \begin{align}\label{def:S}
    \mathbf{S}(\tau,\xi,\eta):=\tau\xi+(1-\tau)(\xe),\quad \tau\in[0,1].
    \end{align}
For convenience, we will suppress the dependence of $\mathbf{S}$ on $\xi,\eta$. Moreover, we fix $\xi,\eta\in\mathbb{R}^2$, where $|\eta|\neq0$.

Observe that we have
    \begin{align}\label{eq:meanvalue1}
        |m(\xi,\eta)|&=\bigg|\int_{0}^{1}\frac{d}{{d\tau}}\left(\abs{\mathbf{S}(\tau)}^{-s}p(|\mathbf{S}(\tau)|)\mathbf{S}(\tau)_{\ell}\right)d\tau\bigg|\notag\\
        &=\bigg |{\int_{0}^{1}\left(-s\abs{\mathbf{S}(\tau)}^{-s-2}(\mathbf{S}(\tau)\cdot \eta)p(|\mathbf{S}(\tau)|)\mathbf{S}(\tau)_{\ell}+\abs{\mathbf{S}(\tau)}^{-s}p(|\mathbf{S}(\tau)|)\eta_{\ell}\right.}\notag\\
        &{\left.\qquad\qquad\qquad \qquad\qquad\qquad+\abs{\mathbf{S}(\tau)}^{-s-1} p'(|\mathbf{S}(\tau)|)(\mathbf{S}(\tau)\cdot \eta)\mathbf{S}(\tau)_{\ell}\right)d\tau}\bigg|\notag\\
        &\le |\eta|\int_{0}^{1}\abs{\mathbf{S}(\tau)}^{-s}\left(p(|\mathbf{S}(\tau)|)+ |p'(|\mathbf{S}(\tau)|)||\mathbf{S}(\tau)|\right)d\tau\notag\\
    &\le C\abs{\eta}\int_{0}^{1}\abs{\mathbf{S}(\tau)}^{-s}p(|\mathbf{S}(\tau)|)\,d\tau\notag\\
    &\le C\Ae\left(\int_{0}^{1}\abs{\mathbf{S}(\tau)}^{-s}d\tau\right)\sup_{\tau\in[0,1]}p(|\mathbf{S}(\tau)|),
    \end{align}
where we applied \ref{item:O2} in obtaining the second inequality since
\[\frac{|(p'_a(r)r)p_b(r)-(p'_b(r)r)p_a(r)|}{p^2_b(r)}\le C\frac{p_a(r)p_b(r)+p_b(r)p_a(r)}{p^2_b(r)}=2Cp(r).\]
Now let $\varphi=\frac{\xe}{\Ae}$ and $\vartheta=\frac{\eta}{\Ae}$. For fixed $\xi$ and $\eta$, we have from  \eqref{est:elem:convex} that
\begin{align}\label{eq:A-integral}
\int_{0}^{1}\abs{ \mathbf{S}(\tau)}^{-s}\,d\tau=\Ae^{-s}\int_{0}^{1}\frac{1}{\abs{\varphi+\tau \vartheta}^{s}}\,d\tau\leq C\Ae^{-s}.
\end{align}
Applying \eqref{eq:A-integral} in \eqref{eq:meanvalue1} and invoking \ref{item:O1}, yields
    \begin{align}\label{est:psl:final}
        |m(\xi,\eta)|\leq C|\eta|^{1-s}\sup_{\tau\in[0,1]}p(|\mathbf{S}(\tau)|)\le C|\eta|^{1-s}p_a(|\xe|+|\xi|).
    \end{align}
We recall that $\eps+s\leq1$, where $\eps>0$. Thus, upon invoking \eqref{def:gen:tri:prod}, we obtain
\begin{align*}
|\mathcal{L}(f,g,h)| \le& C\iint\left(|\eta|^{1-s-\eps}   |\hat{g}(\eta)|\right)\left(p_a (|\xe|)|\xe|^{\eps}|\hat{f}(\xi-\eta)|\right)|\hat{h}(\xi)|d\eta d\xi\notag\\
&+C\iint\left(|\eta|^{1-s-\eps}   |\hat{g}(\eta)|\right)|\hat{f}(\xe)|\left(p_a (|\xi|)|\xi|^{\eps}|\hat{h}(\xi)|\right)d\eta d\xi.
\end{align*}
Lastly, upon applying the Cauchy-Schwarz inequality, Young's convolution inequality, and Plancherel's theorem, we obtain \eqref{est:commutator2:permute}.
\end{proof}

When the functions $f$ and $h$ are spectrally localized away from the origin, one can afford additional flexibility in  \cref{lem:commutator2}.

\begin{Lem}\label{lem:commutator3}
Let $s\in[0,1)$ and $\eps \in [0,1]$ be such that $\eps+s\le 1$. Let $\om \in \mathscr{M}_W$ and $p \in \tilde{\mathscr{M}}_W$, where $p$ is represented as $p_ap_b^{-1}$. Let $\Gam:[0,\infty)\goesto[0,\infty)$ be a function for which there exists a $C>0$ such that 
    \begin{align}\label{cond:Gam3}
         \left(\int_{0}^{y}\frac{rdr}{(1+r^2)\om^2(r)}\right)^{\frac{1}{2}}\le C\Gam(y),
    \end{align}
holds for all $y\geq0$. Then there exists a constant $C>0$ such that if $\supp \hat{f},\ \supp \hat{h}\subset\Acal_j$, for some $j\in\mathbb{Z}$, where $\Acal_j$ is defined in \eqref{def:ann:ball}, then
    \begin{align}
       |\lb [{\Lam}^{-s}p(D)\bdy_\ell,g]f,h\rb|\leq C\left(p(2^j)\Gam(2^j)+p_a(2^j)\om^{-1}(2^j)\right)2^{\eps j}\Sob{g}{H^{2-s-\eps}_\om}\Sob{f}{L^2}\Sob{h}{L^2}. \notag
    \end{align}
\end{Lem}
\begin{proof}
 By the spectral support condition of $f,h$, we may additionally assume that $\supp \hat{g}\subset\Bcal_{j+2}$. Using this, we obtain
\[\mathcal{L}(f,g,h)=I+II,\]
where
    \begin{align}
        I=&\iint m(\xi,\eta)\hat{f}(\xi-\eta){{\mathbbm{1}}_{{\Bcal}_{j-3}}(\eta)}\hat{g}(\eta)\overline{\hat{h}(\xi)}d\eta d\xi,\label{eq:L:split:lf}\\
        II=&\iint m(\xi,\eta)\hat{f}(\xi-\eta){{\mathbbm{1}}_{{\Acal}_{j-3,j+2}}(\eta)}\hat{g}(\eta)\overline{\hat{h}(\xi)}d\eta d\xi\label{eq:L:split:hf}.
    \end{align}
Now we treat $I$ and $II$. Clearly,
    \[
        \abs{\mathbf{S}(\tau)}\le |\xi|+|\xe|\le 2^{j+2}.
    \]

For $ \eta \in {\Bcal}_{j-3}$, we have
\begin{align*}
\abs{\mathbf{S}(\tau)}& \ge \abs{\xe}-\tau\abs{\eta} \ge 2^{j-1}-2^{j-3} \ge 2^{j-2}.
\end{align*}
From this, \eqref{eq:meanvalue1}, and \eqref{est:omega}, we obtain 
\begin{align*}
    |m(\xi,\eta)|\le C|\eta|^{1-s-\eps}p(2^j)2^{\eps j}.
\end{align*}
Applying the Cauchy-Schwarz inequality and \eqref{cond:Gam3}, we have
    \begin{align*}
        \Sob{|\cdotp|^{1-s-\eps}{{\mathbbm{1}}_{{\Bcal}_{j-3}}} \hat{g}}{L^1}&\le\left(\int_{\Bcal_{j-3}}\frac{1}{(1+|\eta|^2)\om^{2}(|\eta|)}d\eta\right)^{1/2}\left(\int_{\Bcal_{j-3}} {(1+|\eta|^2)^{2-s-\eps}\om(|\eta|)^{2}}|\hat{g}(\eta)|^2d\eta\right)^{1/2}\notag\\ 
        &\leq C\Gam(2^j)\Sob{g}{H^{2-s-\eps}_\om}.
    \end{align*}
Using the above estimate in \eqref{eq:L:split:lf} and applying Young's convolution inequality and Plancherel's theorem gives us
    \begin{align}\label{est:I:lf:final}
        I\leq Cp(2^j)2^{\eps j}\Gam(2^j)\Sob{g}{H^{2-s-\eps}_\om}\Sob{f}{L^2}\Sob{h}{L^2}. 
    \end{align}

For $ \eta \in {\Acal}_{j-3,j+2}$, we see from \eqref{est:psl:final}, \ref{item:O1}, and \eqref{est:omega} that
    \begin{align}\label{est:psl:final:loc}
        |m(\xi,\eta)|\leq  C|\eta|^{1-s-\eps}p_a(2^j)2^{\eps j}.
    \end{align}
Applying the Cauchy-Schwarz inequality and \eqref{est:omega}, we have
    \begin{align*}
        \Sob{|\cdotp|^{1-s-\eps}{{\mathbbm{1}}_{{\Acal}_{j-3,j+2}}} \hat{g}}{L^1}&\le\left(\int_{\Acal_{j-3,j+2}}\frac{1}{|\eta|^2\om^{2}(|\eta|)}d\eta\right)^{1/2}\left(\int_{\Acal_{j-3,j+2}} {|\eta|^{2(2-s-\eps)}\om(|\eta|)^{2}}|\hat{g}(\eta)|^2d\eta\right)^{1/2}\notag\\ 
        &\leq C\om^{-1}(2^j)\left(\int_{2^{j-3}}^{2^{j+2}}\frac{1}{r}\,dr\right)^{\frac{1}{2}}\Sob{g}{\Hdot^{2-s-\eps}_\om}\le C\om^{-1}(2^j)\Sob{g}{\Hdot^{2-s-\eps}_\om}.
    \end{align*}
Using the above estimate in \eqref{eq:L:split:hf} and applying Young's convolution inequality and Plancherel's theorem gives us
    \begin{align}\label{est:II:hf:final}
        II\leq Cp_a(2^j)2^{\eps j}\om^{-1}(2^j)\Sob{g}{\Hdot^{2-s-\eps}_\om}\Sob{f}{L^2}\Sob{h}{L^2}. 
    \end{align}
Finally collecting the bounds in \eqref{est:I:lf:final} and \eqref{est:II:hf:final}, we obtain the desired estimate.
\end{proof}

Given $\lam\geq0$ and $\nu(D)\in\mathscr{M}_S$, we define the operator $E^\lam_\nu$ by 
    \begin{align}\label{def:Elamnu}
        (\mathcal{F}E^\lam_\nu\phi)(\xi)=e^{\lam\nu(|\xi|)}(\mathcal{F}\phi)(\xi)
    \end{align}
We then have the following commutator estimates.

\begin{Lem}\label{lem:commutator4}
Let $r, s, \bar{s}\in \RR$ be such that $s,\bar{s}\leq1$, $s+\bar{s}>0$. Let $\om,\om_\ell,\tilde{\om}_\ell\in\mathscr{M}_W$, for $\ell=1,2,3$ { and $\nu$ satisfies \ref{item:S1}, \ref{item:S2}}. Assume that $\Gam_\ell:[0,\infty)\goesto[0,\infty)$, where $\ell=1,2,3$, are functions satisfying \eqref{cond:om:Gam:a}, \eqref{cond:om:Gam:b}. Let $\pi^{r,t}_{\vr,\til{\vr}}$ and $\rho^{r,t}_{\vr,\til{\vr}}$ be defined as in \eqref{def:pi:rho}. Then there exist constants $c,C>0$, and $\{c_j\}\in\ell^2(\ZZ)$ satisfying $\Sob{\{c_j\}}{\ell^2}\leq1$ such that if $\supp \hat{h}\subset\Acal_j$, then 
    \begin{align}\label{eq:commutator4}
       &|\lb[\om(D){\Lam}^{r}E^\lam_\nu{\lpj}\bdy_\ell,g]f,h\rb|\leq
         C(1+\lam)e^{c\lam}c_j2^{(r-s-\bar{s}+1)j}\\
         &\times\left\{ \Gam_1(2^j)\pi^{s,\bar{s}}_{\om_1,\til{\om}_1}(E^\lam_\nu f,\Lam{E^\lam_\nu}g)+\Gam_2 (2^j)\pi^{\bar{s},s}_{\om_2,\til{\om}_2}(\Lam {E^\lam_\nu}g,{E^\lam_\nu}f)+\Gam_{3}(2^j)\rho^{s,\bar{s}}_{\om_3,\til{\om}_3}({E^\lam_\nu}f,\Lam{E^\lam_\nu}g)\right\}\Sob{h}{L^2},\notag
    \end{align}
for all $\lam\geq0$. Lastly, the same inequality holds if $\om(D){\Lam}^{r}E^\lam_\nu{\lpj}\bdy_\ell$ is replaced by $\om(D){\Lam}^{r+1}E^\lam_\nu{\lpj}$ in the left-hand side of \eqref{eq:commutator4}.
\end{Lem}

\begin{proof}[Proof of \cref{lem:commutator4}]
 In order to avoid redundancy in the argument, we will only provide the proof for \eqref{eq:commutator4}. With this in mind, we define
    \begin{align*}
        \mathcal{L}(f,g,h):=\iint\limits_{\xi \in \Acal_j} m(\xi,\eta)\hat{f}(\xi-\eta)\hat{g}(\eta)\overline{\hat{h}(\xi)}\, d\eta\,d\xi,
    \end{align*}
where
    \begin{align*}
        m(\xi,\eta):= e^{\lam \nu(|\xi|)}\abs{\xi}^{r}\xi_{\ell}\om(|\xi|)\phi_{j}(\xi)-e^{\lam \nu(|\xe|)}\abs{\xi-\eta}^{r}(\xi-\eta)_{\ell}\om(|\xe|)\phi_{j}(\xe).
    \end{align*}
By Plancherel's theorem we see that
    \begin{align}\label{def:L:comm:relation}
       \mathcal{L}(f,g,h)=\lb[\om(D){\Lam}^{s}E^\lam_\nu{\lpj}\bdy_\ell,g]f,h\rb.
    \end{align}
Hence, it will be equivalent to obtain the desired bounds for $\mathcal{L}(f,g,h)$.

Let $\mathbf{S}(\tau)$  be as in \eqref{def:S}. It follows that
\begin{align}\label{eq:meanvalue2}
        m(\xi,\eta)
        &=\int_{0}^{1}\Bigg{\{}\left(\lam \nu'(\mathbf{S}(\tau))\frac{\mathbf{S}(\tau)}{|\mathbf{S}(\tau)|}\cdot \eta+r\frac{\mathbf{S}(\tau)}{|\mathbf{S}(\tau)|^{2}}\cdot \eta+\frac{\om'(\mathbf{S}(\tau))}{\om (|\mathbf{S}(\tau)|)}\frac{\mathbf{S}(\tau)}{|\mathbf{S}(\tau)|}\cdot \eta\right)\phi_j (\mathbf{S}(\tau))\mathbf{S}(\tau)_{\ell}\notag\\
        & \qquad \qquad\qquad+\left(\phi_{j}(\mathbf{S}(\tau))\eta_{\ell}+(\nabla \phi_{0})(2^{-j}\mathbf{S}(\tau))\cdot(2^{-j} \eta)\mathbf{S}(\tau)_{\ell}\right)\Bigg{\}} e^{\lam \nu(|\mathbf{S}(\tau)|)}|\mathbf{S}(\tau)|^{r}\om (|\mathbf{S}(\tau)|)d\tau.
    \end{align}
Suppose that $\xi\in\Acal_j$. Since $\supp \phi_{j} \subset \mathcal{A}_{j}$ and $\supp \nabla \phi_0\subset\mathcal{A}_0$, it follows from \eqref{est:omega} that
    \begin{align}\notag
        |m(\xi,\eta)|&\leq C(1+\lam)2^{jr}|\eta|\om(2^j)\int_{0}^{1}\Bigg{\{}\left(|\mathbf{S}(\tau)| \nu'(\mathbf{S}(\tau))+1+\frac{|\mathbf{S}(\tau)|\om'(\mathbf{S}(\tau))}{\om (|\mathbf{S}(\tau)|)}\right)\phi_j (\mathbf{S}(\tau))+1\Bigg{\}} e^{\lam \nu(|\mathbf{S}(\tau)|)}d\tau.
    \end{align}
We then apply \ref{item:S1}, \ref{item:S2}, { and \ref{item:O2} to deduce}
    \begin{align}\notag
        |m(\xi,\eta)|&\leq C(1+\lam)\abs{\eta}2^{r j}\om(2^j)e^{\lam\nu(|\xi-\eta|+|\eta|)}
    \end{align}

Now we estimate $\nu(|\xe|+|\eta|)$. We consider two cases:
\subsubsection*{Case: $|\xe|<|\eta|$} Integrating in \ref{item:S2} yields
\begin{align*}
   \nu(|\xe|+|\eta|)\le \nu(|\eta|)+C\ln\left(1+\frac{|\xe|}{|\eta|}\right)\le \nu(|\eta|)+C.
\end{align*}
\subsubsection*{Case: $|\eta|\le|\xe|$} Similarly, \ref{item:S2} implies
\begin{align*}
   \nu(|\xe|+|\eta|)\le \nu(|\xe|)+C\ln\left(1+\frac{|\eta|}{|\xe|}\right)\le \nu(|\xe|)+C.
\end{align*}
Upon returning to \eqref{eq:meanvalue2} and invoking these estimates, we arrive at
\begin{align*}
    |m(\xi,\eta)|\le C(1+\lam)\abs{\eta}2^{r j}\om(2^j)e^{\lam C}e^{\lam \nu(|\eta|)}e^{\lam \nu(|\xe|)}.
\end{align*}

Finally, we let
    \[
        \mathcal{F}F=|\mathcal{F}{E^\lam_\nu}f|,\quad \mathcal{F}G=|\mathcal{F}{\Lam{E^\lam_\nu}g}|.
    \]
We apply \eqref{eq:meanvalue2} in \eqref{def:L:comm:relation}, followed by the Cauchy-Schwarz inequality and Young's convolution inequality to obtain
    \begin{align}\label{in:final:lem2}
        |\mathcal{L}(f,g,h)|\leq C(1+\lam)e^{\lam C}2^{rj}\om(2^j)\Sob{\tlpj(FG)}{L^2}\Sob{h}{L^2},
    \end{align}
where $\tlpj$ denotes the extended Littlewood-Paley blocks as defined in \eqref{def:para:factors}. Since $s,\bar{s}\in\RR$ is assumed to satisfy $s, \bar{s}\leq 1$ and $s+\bar{s}>0$, and since the $\Gam_\ell$ satisfy \eqref{cond:om:Gam:a}, \eqref{cond:om:Gam:b}, we may apply \cref{thm:prod}. The proof is complete upon application of Plancherel's theorem, followed by the characterization of Sobolev norms in terms of Besov norms (see \eqref{eq:Sob:Bes}).
\end{proof}

\section{Analysis of the Protean System}\label{sect:apriori}
In this section, we will obtain apriori estimates for the solution of \eqref{def:mod:flux} in $L^2$, $H^{\s}_{\om}$, and ${E}_{\nu,\s,\om}^\lam$, for particular choices of $\s\in[-1,3]$ that depend on $\be\in[0,2]$ and whether $\lam=0$ or not. We will restrict to $\s$ in the following ranges:
\begin{align}\label{sigma:range}
    \s\in
    \begin{cases}
    (-1,2], &\text{if} \quad \be \in [0,1],\\
    [1,1+\be], &\text{if} \quad \be \in (1,2],\\
    \{-1\},&\text{if}\quad \be=0, \lam=0.
    \end{cases}
\end{align}
Development of the subsequent apriori estimates establishes a global existence theory for the protean system \eqref{eq:mod:claw}. This is formally stated in \cref{thm:modclaw:wellposed} after the apriori estimates have been established; the rigorous details for global existence and uniqueness are supplied in \cref{app:well:posed}.

Observe that when $\be\in[0,1]$, \eqref{eq:mod:claw} is a linear transport equation, while in the more singular regime, $\be\in(1,2]$, \eqref{eq:mod:claw} is a conservation law with a flux that modifies the linear transport equation. The modification is ultimately required to accommodate suitable stability-type estimates for \eqref{eq:dgsqg}, but due to \eqref{eq:cancel}, we find it expedient to obtain all estimates in the generality of \eqref{eq:mod:claw} and simply reduce them to the case of \eqref{eq:dgsqg}, as needed. The formal apriori estimates are developed in  \cref{sect:Hs:apriori}.  The stability-type estimates are then established in \cref{sect:stab:apriori}. 

{ Recall that $m_1(D):=I+m(D)\in\mathscr{M}_W$ }(see \eqref{def:MD}) and, thus, satisfies \ref{item:O1}, \ref{item:O2}, \ref{item:O3}. For convenience, we also recall that  $\om\in\mathscr{M}_W$ (see \eqref{def:MM}), $p\in\mathscr{M}_C$ (see \eqref{def:MC}), and $\nu(D)\in\mathscr{M}_S(m)$ (see \eqref{def:M:m}).

\begin{Rmk}\label{rmk:convention:constants}Below, we will adopt the convention of summation over repeated indices. We also denote by $C$ a generic positive constant, which may depend on various regularity parameters in addition to the size of a given time interval $[0,T]$. For clarity, we may indicate the dependence of the $C$ on parameters through a subscript. In general, however, the value of $C$ may change line-to-line. 
\end{Rmk}

We begin by establishing the following lemma demonstrating the equivalence of the norms $\Sob{\phi}{\Hdot^\s_\om\cap L^2_{\om}}$, $\Sob{\phi}{L^2}+\Sob{\phi}{\Hdot^\s_\om}$, and $\Sob{\phi}{H^\s_\om}$ when $\s>0$. 

\begin{Lem}\label{lem:norm:equiv}
Let $\om(D)\in\mathscr{M}_W$ as defined in \cref{sect:op:classes}. For any $\s>0$, there exist positive constants $c_{\s,\om}, C_{\s,\om}$ such that
     \begin{align}\label{eq:norm:equiv:a}
        c_{\s,\om}^{-1}\Sob{\phi}{\Hdot^\s_\om\cap L^2_{\om}}\leq\Sob{\phi}{\Hdot_\om^\s}+ \Sob{\phi}{L^2}\leq c_{\s,\om}\Sob{\phi}{\Hdot^\s_\om\cap L^2_{\om}},
    \end{align}
and
    \begin{align}\label{eq:norm:equiv:b}
        C_{\s,\om}^{-1}\Sob{\phi}{H^\s_\om}\leq \Sob{\phi}{\Hdot_\om^\s}+\Sob{\phi}{L^2}\leq C_{\s,\om}\Sob{\phi}{H^\s_\om}.
    \end{align}
\end{Lem}

\begin{proof} 
Since $\s>0$, we may invoke \eqref{eq:omega:eps}, in addition to \ref{item:O1}, to obtain
    \begin{align}\notag
        \Sob{\phi}{L^2_{\om}}^2=\int_{\RR^2} |\hat{\phi}(\xi)|^2\om(|\xi|)d\xi=\int_{|\xi|\leq1}+\int_{|\xi|>1}\leq \frac{\om_a(1)}{\om_a(0)}\Sob{\phi}{L^2}^2+c_\s^2\Sob{\phi}{\Hdot_\om^\s}^2.
    \end{align}
Hence
    \begin{align}
        \Sob{\phi}{\Hdot^\s_\om\cap L^2_{\om}}^2&\leq\frac{\om_a(1)}{\om_a(0)}\Sob{\phi}{L^2}^2+(1+c_\s^2)\Sob{\phi}{\Hdot_\om^\s}^2\notag\\
        &\leq \left(\frac{\om_a(1)}{\om_a(0)}+1+c_\s^2\right)\Sob{\phi}{H^\s_\om}^2\leq C_\s^2\left(\frac{\om_a(1)}{\om_a(0)}+1+c_\s^2\right)\Sob{\phi}{\Hdot^\s_\om\cap L^2_{\om}}^2,\notag
    \end{align}
for some constant $C_\s>0$. This implies \eqref{eq:norm:equiv:a}, \eqref{eq:norm:equiv:b}.
 \end{proof}

\subsection{A priori estimates in $L^2$}\label{sect:L2:apriori}

Upon taking the inner product in $L^2$ of \eqref{eq:mod:claw}  with $\tht$, we obtain
    	\begin{align}\label{eq:balance:L2} 
	   \frac{1}{2}\frac{d}{dt}\Sob{\tht}{L^2}^2+\Sob{m(D)^{\frac{1}2}\tht}{L^2}=-\lb\Div F_q(\tht),{\tht}\rb+\lb G ,\tht \rb=I^0+II^0.
    \end{align}
We estimate $II^0$ with the Cauchy-Schwarz inequality and Young's inequality, to obtain
    \begin{align}\label{est:II:L2}
        |II^0|\leq& \Sob{m_1(D)^{-\frac{1}2}G}{L^2}\Sob{m_1(D)^{\frac{1}2}\tht}{L^2}\notag \\\leq& C\Sob{m_1(D)^{-\frac{1}2}G}{L^2}^2+C\Sob{\tht}{L^2}^2+\frac{1}8\Sob{m(D)^{\frac{1}2}\tht}{L^2}^2.
    \end{align}
We now consider the cases $\be \in [0,1]$ and $\be \in (1,2]$ separately to treat $I^0$.

\subsubsection*{Case $\be\in[0,1]$}
By \eqref{def:mod:flux} and \eqref{eq:v:cancel}, we have $I^0=0$. Given \eqref{est:II:L2} for $II^0$, \eqref{eq:balance:L2} then becomes
    \begin{align}\label{est:L2:a}
        \frac{d}{dt}\Sob{\tht}{L^2}^2+\frac{7}4\Sob{m(D)^{\frac{1}2}\tht}{L^2}^2\leq C\Sob{m_1(D)^{-\frac{1}2}G}{L^2}^2+C\Sob{\tht}{L^2}^2.
    \end{align}
An application of Gronwall's inequality, then yields
    \begin{align}\label{est:L2:a:gronwall}
        \Sob{\tht(t)}{L^2}^2\leq C\exp(CT)\left(\Sob{\tht_0}{L^2}^2+\int_0^T\Sob{m_1(D)^{-\frac{1}2}G(s)}{L^2}^2ds\right),
    \end{align}
for all $0\leq t\leq T$. Then, upon integrating \eqref{est:L2:a} over $[0,T]$ and applying \eqref{est:L2:a:gronwall}, we obtain
    \begin{align}\label{est:L2log:a:gronwall}
        \int_0^T\Sob{m(D)^{\frac{1}2}\tht(s)}{L^2}^2ds\leq C\exp(CT)\left(\Sob{\tht_0}{L^2}^2+\int_0^T\Sob{m_1(D)^{-\frac{1}2}G(s)}{L^2}^2ds\right).
    \end{align}
Hence
    \begin{align}\label{est:L2:final}
         \sup_{0\leq t\leq T}\left(\Sob{\tht(t)}{L^2}^2+\int_0^t\Sob{m(D)^{\frac{1}2}\tht(s)}{L^2}^2ds\right)\leq C\exp(CT)\left(\Sob{\tht_0}{L^2}^2+\int_0^T\Sob{m_1(D)^{-\frac{1}2}G(s)}{L^2}^2ds\right).
    \end{align}
\subsubsection*{Case $\be \in (1,2]$}
Since $a(D)\bdy_\ell$ is a skew self-adjoint operator, in light of  \eqref{eq:v:cancel}, we see that
    \begin{align}\label{eq:skew:adjoint}
   I^0=\lb a(D){\nabla}\cdotp(({\nabla}^\perp q)\tht),\tht\rb&=-\lb\nabla^\perp q\cdotp\nabla (a(D)\tht),\tht\rb\notag\\
   &=\frac{1}{2}\lb [a(D),\nabla^\perp q\cdotp\nabla]\tht,\tht\rb=\frac{1}{2}\lb [a(D)\bdy_\ell,\bdy^\perp_\ell q]\tht,\tht\rb,
    \end{align}
where $\bdy^\perp_\ell q=(\nabla^\perp q)^\ell$ and we adopt the convention of summation over repeated indices. Upon applying \cref{lem:commutator2} with $s=2-\beta$ and $,\eps',\eps\in(0,1]$ sufficiently small so that $\eps+s\leq1$, and $(1+\gam)\eps'\leq\eps$, we obtain
\begin{align}\notag
    |I^0|
    \leq C\Sob{\nabla^{\perp}q}{{H}^{\be-\de}}\Sob{p_a(D)\tht}{\Hdot^{\eps'}}\Sob{\tht}{L^2},
\end{align}
for any $0<\de<\eps$. We assume that
    \begin{align}\label{cond:pa}
        \sup_{y>0}\frac{p_a(y)}{\om(y)m^{\gam}_1(y)}<\infty.
    \end{align}
Thus
    \begin{align}\notag
        |I^0|\leq C\Sob{q}{H_\om^{1+\be}}\Sob{m_1(D)^{\gam}\tht}{\Hdot_\om^{\eps'}}\Sob{\tht}{L^2}.
    \end{align}
We recall that $m_1(D)$ satisfies \eqref{eq:omega:eps:bdd}, so that
    \[
    \Sob{m_1(D)^{\gam}\tht}{\Hdot_\om^{\eps'}}\leq \Sob{\tht}{H_\om^{(1+\gam)\eps'}}\leq \Sob{\tht}{H_\om^{\eps}}.
    \]
Finally, after an application of the Cauchy-Schwarz inequality, we arrive at estimate
    \begin{align}\label{est:I:L2}
        |I^0|\leq C\Sob{q}{H_\om^{1+\be}}\Sob{\tht}{H_\om^{\eps}}\Sob{\tht}{L^2}.
    \end{align}
Returning to \eqref{eq:balance:L2} and applying \eqref{est:II:L2} and \eqref{est:I:L2} then yields
    \begin{align}\label{est:L2:b}
        \frac{d}{dt}\Sob{\tht}{L^2}^2+\frac{7}4\Sob{m(D)^{\frac{1}2}\tht}{L^2}^2\le C\Sob{q}{H^{1+\be}_\om}\Sob{\tht}{H^\eps_\om}\Sob{\tht}{L^2}+C\Sob{m_1(D)^{-\frac{1}2}G}{L^2}^2+C\Sob{\tht}{L^2}^2.
    \end{align}
In particular, when $\be\in(1,2]$ we will require control in $H^\eps_\om$, for some $\eps>0$, in order to close estimates.

\subsection{A priori estimates in $\dot{E}^\lam_{\nu,\s,\om}$}\label{sect:Hs:apriori} With the exception of the special case $\be=0$, $\s=-1$, it will be convenient to develop the apriori estimates in the stronger space $\dot{E}^{\lam}_{\nu,\s,\om}$ and then specialize to the case of $\Hdot^\s_{\om}$ later by simply setting $\lam=0$. We emphasize that the commutator estimates developed in \cref{sect:comm} can accommodate such a procedure due to the form of dependency of the constants on $\lam$ (see \cref{lem:commutator4}). In developing the estimates in $\dot{E}^{\lam}_{\nu,\s,\om}$, we will make use of the following shorthand: given $r\in\RR$, $\lam \ge 0$, and $j\in\ZZ$, we let
    \begin{align}\label{def:short}
       \til{f}:=E^{\lam(t)}_\nu f,\quad \til{\Lam}^{r}_{\om}:=\om(D){\Lam}^{r}E^{\lam(t)}_\nu,\quad \til{\Lam}_{\om,j}^{r}=:\om(D){\Lam}^{r}E^{\lam(t)}_\nu{\lpj},\quad \lam(t)=\lam_1t,
    \end{align}
for a fixed $\lam_1>0$, where $E^\lam_\nu$ is defined as in \eqref{def:Elamnu}. We will often abuse notation and say $\lam=\lam(t)$. Throughout, we suppose that $0\leq t\leq T$.

Similar to \cref{lem:norm:equiv}, we establish the equivalence of the norms $\Sob{\til{\phi}}{\Hdot^\s}+\Sob{\phi}{L^2}$, $\Sob{\til{\phi}}{\Hdot^\s\cap L^2}$, and  $\Sob{\til{\phi}}{H^\s}$ are equivalent as norms when $\s\geq0$.

\begin{Lem}\label{lem:norm:equiv:til}
Given $m(D)\in\mathscr{M}_D$, Suppose $\nu(D)\in\mathscr{M}_S(m)$, where $\mathscr{M}_S(m)$ is defined as in \eqref{def:MD} from \cref{sect:op:classes}. Given $\lam>0$, let $E^\lam_\nu$ be defined as in \eqref{def:Elamnu}. For any $\s\geq0$, there exists a positive constant $C_{\lam,m}$ such that
    \begin{align}\label{eq:norm:equiv:til:a}
        C_{\lam,m}^{-1}\Sob{E^\lam_\nu\phi}{\Hdot^\s\cap L^2}\leq &\Sob{E^\lam_\nu\phi}{\Hdot^\s}+\Sob{\phi}{L^2}\leq C_{\lam,m}\Sob{E^\lam_\nu\phi}{\Hdot^\s\cap L^2}.
    \end{align}
and
    \begin{align}\label{eq:norm:equiv:til:b}
        C_{\lam,m}^{-1}\Sob{E^\lam_\nu\phi}{H^\s}\leq &\Sob{E^\lam_\nu\phi}{\Hdot^\s}+\Sob{\phi}{L^2}\leq C_{\lam,m}\Sob{E^\lam_\nu\phi}{H^\s}.
    \end{align}
\end{Lem}
\begin{proof}
Observe that from \eqref{def:MD}, we have
    \begin{align}
        \Sob{\til{\phi}}{L^2}^2=\int_{\RR^2}e^{2C\lam\nu(|\xi|)}|\hat{\phi}(\xi)|^2d\xi\leq \int_{|\xi|\leq 1}e^{2\lam(1+m(|\xi|))}|\hat{\phi}(\xi)|^2d\xi+\int_{|\xi|>1}|\xi|^{2\s}|\til{\phi}(\xi)|^2d\xi,\notag
    \end{align}
provided that $\s\geq0$. Since $I+m(D)\in{\mathscr{M}}_W$, we may assume that $I+m(D)=m_a(D)m_b(D)^{-1}$. In particular, by \ref{item:O1}, it follows that $1+m(|\xi|)\leq m_a(1)m_b(0)^{-1}$. Hence
    \begin{align}
        \Sob{\til{\phi}}{\Hdot^\s}^2+\Sob{\til{\phi}}{L^2}^2&\leq \exp\left(2C\lam\frac{m_a(1)}{m_b(0)}\right)\Sob{\phi}{L^2}^2+2\Sob{\til{\phi}}{\Hdot^\s}^2\notag\\
        &\leq \exp\left(2C\lam\frac{m_a(1)}{m_b(0)}\right)\Sob{\til{\phi}}{L^2}^2+2\Sob{\til{\phi}}{\Hdot^\s}^2\leq C_{\lam,m}^2\Sob{\til{\phi}}{H^\s}^2,\notag
    \end{align}
which implies \eqref{eq:norm:equiv:til:a}, \eqref{eq:norm:equiv:til:b}.
\end{proof}

From \cref{lem:norm:equiv} and \cref{lem:norm:equiv:til}, we immediately deduce the following equivalence.

\begin{Cor}\label{cor:norm:equiv}
Suppose $\om(D)\in\mathscr{M}_W$, $m(D)\in\mathscr{M}_D$, $\nu(D)\in\mathscr{M}_S(m)$, and that $E^\lam_\nu$ is defined as in \eqref{def:Elamnu} with $\lam>0$. Then for any $\s\geq0$, we have the following chain of equivalent norms:
    \begin{align}\label{eq:norm:equiv:summary}
            \Sob{E^\lam_\nu\phi}{\Hdot^\s_\om\cap L^2_{\om}}\sim\Sob{E^\lam_\nu\phi}{H^\s_\om}+\Sob{E^\lam_\nu\phi}{L^2}\sim\Sob{E^\lam_\nu\phi}{H^\s_\om}+\Sob{\phi}{L^2}\sim\Sob{E^\lam_\nu\phi}{H^\s_\om},
    \end{align}
where the suppressed constants depend on $\s,\lam,\om,m$.
\end{Cor}

Now, upon applying $ \til{\Lam}_{\om,j}^{\s}$ to \eqref{eq:mod:claw}, one obtains
\begin{align}\label{eq:balance:Hsigma:prep}
    \partial_t (\til{\Lam}_{\om,j}^{\s} \tht)+\til{\Lam}_{\om,j}^{\s}(\Div F_q(\tht))=\til{\Lam}^{\s}_{\om,j}\tht-m_1(D)\til{\Lam}_{\om,j}^{\s}\tht+\lam_1 \nu(D)\til{\Lam}_{\om,j}^{\s}\tht+\til{\Lam}_{\om,j}^{\s}G.
\end{align}
Then taking the $L^2$--inner product of \eqref{eq:balance:Hsigma:prep} with $ \til{\Lam}_{\om,j}^{\s}\tht$, we obtain
\begin{align}\label{eq:balance:Hsigma}
	   \frac{1}{2}\frac{d}{dt}\Sob{ \til{\Lam}_{\om,j}^{\s}\tht}{L^2}^2+&\Sob{m_1(D)^{\frac{1}2} \til{\Lam}_{\om,j}^{\s}\tht}{L^2}^{2}\notag\\
	   &=\lam_1\Sob{\nu(D)^{\frac{1}2} \til{\Lam}_{\om,j}^{\s}\tht}{L^2}^{2}+\Sob{\til{\Lam}^\s_{\om,j}\tht}{L^2}^2-\lb  \til{\Lam}_{\om,j}^{\s}(\Div F_q(\tht)), \til{\Lam}_{\om,j}^{\s}{\tht}\rb+\lb \til{\Lam}_{\om,j}^{\s}G , \til{\Lam}_{\om,j}^{\s}\tht\rb\notag\\
	   &=\lam_1\Sob{\nu(D)^{\frac{1}2} \til{\Lam}_{\om,j}^{\s}\tht}{L^2}^{2}+\Sob{\til{\Lam}^\s_{\om,j}\tht}{L^2}^2+I^\s+II^\s.
	\end{align}
Invoking the fact that $\nu\in \mathscr{M}_R(m)$ (see \eqref{def:M:m}), we have for $\lam_1$ sufficiently small
\begin{align}\label{est:gevrey:residual}
    \lam_1\Sob{\nu(D)^{\frac{1}2} \til{\Lam}_{\om,j}^{\s}\tht}{L^2}^{2}\le \frac{1}{16}\Sob{m_1(D)^{\frac{1}2} \til{\Lam}_{\om,j}^{\s}\tht}{L^2}^{2}.
\end{align}
We estimate $II^\s$ with the Cauchy-Schwarz inequality and Young's inequalities to obtain
\begin{align}\label{est:II:Hsigma}
    |II^\s|&\le \Sob{m_1(D)^{-\frac{1}2}\til{\Lam}^{\s}_{\om,j} G}{L^2}\Sob{m_1(D)^{\frac{1}2}\til{\Lam}^{\s}_{\om,j} \tht}{L^2}\notag\\&\leq C\Sob{m_1(D)^{-\frac{1}2}\til{\Lam}^{\s}_{\om,j} G}{L^2}^2+\frac{1}{16}\Sob{m_1(D)^{\frac{1}2}\til{\Lam}^{\s}_{\om,j}\tht}{L^2}^2.
\end{align}
Upon applying \eqref{est:gevrey:residual} and \eqref{est:II:Hsigma} in \eqref{eq:balance:Hsigma}, we obtain
\begin{align}\label{est:balance:Hsigma:final}
    \frac{d}{dt}\Sob{\til{\Lam}_{\om,j}^{\s}\tht}{L^2}^2+\frac{7}4\Sob{m_1(D)^{\frac{1}2} \til{\Lam}_{\om,j}^{\s}\tht}{L^2}^{2}
    &\leq 2\Sob{ \til{\Lam}_{\om,j}^{\s}\tht}{L^2}^2+C\Sob{m_1(D)^{-\frac{1}2}\til{\Lam}^{\s}_{\om,j} G}{L^2}^2+2{I^{\s}}.
\end{align}
We treat $I^{\s}$ by considering the cases $\be\in[0,1]$, $\be\in(1,2]$, and $\be=0$, recalling that $\s$ is restricted by \eqref{sigma:range}.

\subsubsection*{Case 1: $\be\in[0,1]$} In this case, \eqref{sigma:range} implies $\s\in(-1,2]$.

\subsubsection*{Subcase 1a: $\s\in(-1,1]$} First observe that since $v$ is divergence-free, we have in \eqref{est:balance:Hsigma:final} that
    \begin{align}\label{def:I:decomp:a}
        I^{\s}=\lb[\til{\Lam}^{\s}_{\om, j}\bdy_{\ell},v^{\ell}]\tht,\til{\Lam}^{\s}_{\om,j}\tht\rb.
    \end{align}
Applying \cref{lem:commutator4} with $r=\s, (s,\bar{s})=(\s,1)$, $(\om_1,\tilde{\om}_1)=(\om,m_1^{\frac{\gam}{2}}p^{-1}\om), (\om_2,\tilde{\om}_2)=(p^{-1}\om,m_1^{\frac{\gam}{2}}\om), (\om_3,\tilde{\om}_3)=(\om, m_1^{\frac{\gam}{2}}p^{-1}\om)$, and $\Gam_1=\Gam_2=\Gam_3=m^{\frac{\gam}2}_1$, we obtain 
\begin{align}\label{est:I:hfreq:a}
    |I^\s|
    &\le C_\lam c_j m_1(2^j)^{\frac{\gam}{2}}\left(\nrm{\til{\tht}}_{\dot{H}^{\s}_\om \cap L^2_\om}\nrm{m_1(D)^{\frac{\gam}{2}}p(D)^{-1}\Lam\til{ v}}_{\Hdot^{1}_{\om}}+ \nrm{p(D)^{-1}\Lam\til{ v}}_{H^1_\om}\nrm{m_1(D)^{\frac{\gam}{2}}\til{\tht}}_{\Hdot^{\s}_\om}\right.\notag\\&\hspace{20 em}\left.+\nrm{\til{\tht}}_{\Hdot^{\s}_\om}\nrm{m_1(D)^{\frac{\gam}{2}}p(D)^{-1}\Lam \til{v}}_{\Hdot^{1}_\om}\right)\nrm{\til{\Lam}^{\s}_{\om, j}\tht}_{L^2}\notag\\
    &\le C_\lam c_{j}\left( \nrm{\til{\tht}}_{\dot{H}^{\s}_\om \cap L^2_\om}\nrm{m_1(D)^{\frac{\gam}{2}}\til{q}}_{\Hdot^{1+\be}_\om}+\nrm{\til{q}}_{H^{1+\be}_\om}\nrm{m_1(D)^{\frac{\gam}{2}}\til{\tht}}_{\Hdot^{\s}_\om}\right)\nrm{m_1(D)^{\frac{\gam}{2}}\til{\Lam}^{\s}_{\om, j}\tht}_{L^2}.
\end{align}
 Then \eqref{cond:om:Gam:a}, \eqref{cond:om:Gam:b} become
    \begin{align}\label{cond:Gam:I0}
         \begin{split}
     &\sup_{y>0}\frac{p(y)}{{m^{\gam}_1(y)}}\left({\mathbbm{1}}_{(-\infty,1)}(\sigma)\int_0^1\frac{r^{1-2\s}}{\om^2(y r)}dr+{\mathbbm{1}}_{[1,\infty)}(\s)\int_{0}^{y}\frac{r}{(1+r^2)\om^{2}(r)}dr\right)^{\frac{1}{2}}<\infty,\\
     &\sup_{y>0}\frac{1}{{m^{\gam}_1(y)}}\left(\int_{0}^{y}\frac{rp^2(r)}{(1+r^2)\om^{2}(r)}dr\right)^{\frac{1}{2}},\qquad \sup_{y>0}\frac{p(y)}{\om(y)m^{\gam}_1(y)}<\infty.
        \end{split}
    \end{align}
    
\subsubsection*{Subcase 1b: $\s\in (1,2]$}
Since $\nabla \cdot v=0$, it follows that 
    \begin{align}\label{def:I:decomp:b}
        I^\s=\lb[\til{\Lam}^{\s}_{\om, j},v^{\ell}]\bdy_{\ell}\tht,\til{\Lam}^{\s}_{\om,j}\tht\rb.
    \end{align}
We apply \cref{lem:commutator4} with $r=\s-1$, $(s,\bar{s})=(\s-1,1)$,  $(\om_1,\tilde{\om}_1)=(\om,m_1^{\frac{\gam}{2}}p^{-1}\om)$, $(\om_2,\tilde{\om}_2)=(p^{-1}\om,m_1^{\frac{\gam}{2}}\om)$, $(\om_3,\tilde{\om}_3)=(\om, m_1^{\frac{\gam}{2}}p^{-1}\om)$, and $\Gam_1=\Gam_2=\Gam_3=m^{\frac{\gam}2}_1$, to obtain
\begin{align}\label{est:I:hfreq:b}
    |I^\s|
    \le C_\lam c_{j}\left( \nrm{\til{\tht}}_{H^{\s}_\om}\nrm{m_1(D)^{\frac{\gam}{2}}\til{q}}_{\Hdot^{1+\be}_\om}+\nrm{\til{q}}_{H^{1+\be}_\om}\nrm{m_1(D)^{\frac{\gamma}{2}}\til{\tht}}_{\Hdot^{\s}_\om}\right)\nrm{m_1(D)^{\frac{\gam}{2}}\til{\Lam}^{\s}_{\om, j}\tht}_{L^2},
\end{align}
provided that the following holds:
    \begin{align}\label{cond:Gam:I0:b}
         \begin{split}
     &\sup_{y>0}\frac{p(y)}{{m^{\gam}_1(y)}}\left({\mathbbm{1}}_{(-\infty,1)}(\sigma-1)\int_0^1\frac{r^{1-2(\s-1)}}{\om^2(y r)}dr+{\mathbbm{1}}_{[1,\infty)}(\s-1)\int_{0}^{y}\frac{r}{(1+r^2)\om^{2}(r)}dr\right)^{\frac{1}{2}}<\infty,\\
     &\sup_{y>0}\frac{1}{{m^{\gam}_1(y)}}\left(\int_{0}^{y}\frac{rp^2(r)}{(1+r^2)\om^{2}(r)}dr\right)^{\frac{1}{2}},\qquad \sup_{y>0}\frac{p(y)}{\om(y)m^{\gam}_1(y)}<\infty.
        \end{split}
    \end{align}

\subsubsection*{Concluding Estimates for Case 1: $\be \in [0,1]$} 

First observe that the conditions \eqref{cond:Gam:I0} and \eqref{cond:Gam:I0:b} can both be reduced to the following single set of conditions:
    \begin{align}\label{cond:Gam:I0:final}
         \begin{split}
     \sup_{y>0}\left\{\frac{1}{m_1^\gam(y)}\left(\int_{0}^{y}\frac{r(p^2(y)+p^2(r))}{(1+r^2)\om^{2}(r)}dr\right)^{\frac{1}{2}},\quad \frac{p_a(y)\om_b(y)}{m^{\gam}_1(y)}\right\}<\infty.
        \end{split}
    \end{align}

Upon returning to \eqref{est:balance:Hsigma:final} and combining it with \eqref{est:I:hfreq:a}, \eqref{est:I:hfreq:b}, then summing over $j$ and invoking \eqref{eq:Sob:Bes}, we deduce that
    \begin{align}\label{est:Hs:a}
        \frac{d}{dt}\Sob{\til{\tht}}{\Hdot^\s_\om}^2+\frac{7}4\Sob{m_1(D)^{\frac{1}2}\til{\tht}}{\Hdot^\s_\om}^{2}&\leq C_\lam \left(\nrm{\til{\tht}}_{\dot{H}^{\s}_\om \cap L^2_\om}\nrm{m_1(D)^{\frac{\gam}{2}}\til{q}}_{\Hdot^{1+\be}_\om}+\nrm{\til{q}}_{H^{1+\be}_\om}\nrm{m_1(D)^{\frac{\gam}{2}}\til{\tht}}_{\Hdot^{\s}_\om}\right)\nrm{m_1(D)^{\frac{\gam}{2}}\til{\tht}}_{\Hdot^\s_\om}\notag\\&\quad+C\left(\Sob{m_1(D)^{-\frac{1}{2}}\til{G}}{\Hdot^\s_\om}^2+\Sob{\til{\tht}}{\Hdot^\s_\om}^2\right),
    \end{align}
holds for all $\s\in(-1,2]$, provided that \eqref{cond:Gam:I0:final} holds.

In particular, by jointly applying \eqref{est:Hs:a} for $\s\in(-1,2]$ and $\s=0$, we deduce that
    \begin{align}\label{est:Hs:b:final}
            \frac{d}{dt}\Sob{\til{\tht}}{\Hdot^\s_\om\cap L^2_{\om}}^2&+\frac{7}4\Sob{m_1(D)^{\frac{1}2}\til{\tht}}{\Hdot^\s_\om\cap L^2_{\om}}^{2}\notag\\
            &\leq C_\lam \left(\nrm{\til{\tht}}_{\dot{H}^{\s}_\om \cap L^2_\om}\nrm{m_1(D)^{\frac{\gam}{2}}\til{q}}_{\Hdot^{1+\be}_\om}+\nrm{\til{q}}_{H^{1+\be}_\om}\nrm{m_1(D)^{\frac{\gam}{2}}\til{\tht}}_{\Hdot^{\s}_\om\cap L^2_{\om}}\right)\nrm{m_1(D)^{\frac{\gam}{2}}\til{\tht}}_{\Hdot^\s_\om\cap L^2_{\om}}\notag\\&\quad+C\left(\Sob{m_1(D)^{-\frac{1}{2}}\til{G}}{\Hdot^\s_\om\cap L^2_{\om}}^2+\Sob{\til{\tht}}{\Hdot^\s_\om\cap L^2_{\om}}^2\right).
    \end{align}

\begin{Rmk}\label{rmk:equivalence}

Note that upon summing in $j$ in obtaining \eqref{est:Hs:a}, we in fact obtain the inequality
 \begin{align}\label{est:Hs:a:rmk}
        \frac{d}{dt}\Sob{\Lam^{\s}_{\om}\til{\tht}}{\Bdot^0_{2,2}}^2+&\frac{7}4\Sob{\Lam^{\s}_{\om}m_1(D)^{\frac{1}2}\til{\tht}}{\Bdot^0_{2,2}}^{2}\notag\\
        &\leq C_\lam \left(\nrm{\til{\tht}}_{\dot{H}^{\s}_\om \cap L^2_\om}\nrm{m_1(D)^{\frac{\gam}{2}}\til{q}}_{\Hdot^{1+\be}_\om}+\nrm{\til{q}}_{H^{1+\be}_\om}\nrm{m_1(D)^{\frac{\gam}{2}}\til{\tht}}_{\Hdot^{\s}_\om}\right)\nrm{\Lam^{\s}_{\om}m_1(D)^{\frac{\gam}{2}}\til{\tht}}_{\Bdot^0_{2,2}}\notag\\&\quad+C\left(\Sob{\Lam^{\s}_{\om}m_1(D)^{-\frac{1}{2}}\til{G}}{\Bdot^0_{2,2}}^2+\Sob{\Lam^{\s}_{\om}\til{\tht}}{\Bdot^0_{2,2}}^2\right).
    \end{align}
Owing to \eqref{eq:Sob:Bes} and \cref{lem:Bernstein}, observe that we may bound all Sobolev-norm based quantities in terms of their equivalent Besov-norm based quantities. Later on, after an application of Young's inequality and Gronwall's inequality, we may convert all Besov-norm based quantities back in terms of their Sobolev-based counterparts. In particular, in the final analysis, all quantities may be interpreted in their Sobolev-based form.

Henceforth, we will abuse notation and express all quantities related to the apriori estimates in terms of Sobolev norms.
\end{Rmk}

\subsubsection*{Case 2: $\be\in(1,2]$} 
By \eqref{sigma:range}, we restrict to the regime $\s\in[1,1+\be]$. We will treat the cases $\s \in [1,2)$ and $\s \in [2,1+\be]$ separately.

\subsubsection*{Subcase 2a: $\s\in[1,2)$}
From \eqref{def:mod:flux} and the facts that $v$ is divergence-free and $a(D)\bdy_\ell:=\Lam^{\be-2}p(D)\bdy_\ell$ is skew self-adjoint, we see that $I^\s$ in \eqref{est:balance:Hsigma:final} can be decomposed as
\begin{align}
    I^\s
    &=I_{1}^\s+I_{2}^\s+I_{3}^\s+I_{4}^\s,\label{def:I:decomp:c}
\end{align}
where
\begin{align*}
    I_{1}^\s&= -\lb \til{\Lam}^{\s}_{\om,j}(\nabla^{\perp}a(D)q\cdot\nabla \tht),\til{\Lam}^{\s}_{\om,j}{\tht}\rb+\lb \nabla^{\perp}a(D)q\cdot\nabla\til{\Lam}^{\s}_{\om,j} \tht,\til{\Lam}^{\s}_{\om,j}{\tht}\rb\notag \\&=-\lb [\til{\Lam}^{\s}_{\om,j},\bdy^{\perp}_{\ell}a(D)q]\bdy_{\ell}\tht,\til{\Lam}^{\s}_{\om,j}\tht \rb=-\lb [\til{\Lam}^{\s}_{\om,j},v^\ell]\bdy_\ell\tht,\til{\Lam}^{\s}_{\om,j}\tht\rb,\\
    I_{2}^\s&=-\lb \nabla^{\perp}a(D)q\cdot\nabla\til{\Lam}^{\s}_{\om,j} \tht,\til{\Lam}^{\s}_{\om,j}{\tht}\rb=0,\\
    I_{3}^\s&=\lb \til{\Lam}^{\s}_{\om,j}a(D)(\nabla^{\perp}q\cdot\nabla \tht),\til{\Lam}^{\s}_{\om,j}{\tht}\rb-\lb \nabla^{\perp}q\cdot\nabla{a(D)}^{1/2}\til{\Lam}^{\s}_{\om,j} \tht,{a(D)}^{\frac{1}2}\til{\Lam}^{\s}_{\om,j}{\tht}\rb\notag\\&=\lb [\til{\Lam}^{\s}_{\om,j}{a(D)}^{\frac{1}2},\bdy^{\perp}_{\ell}q]\bdy_{\ell}\tht,\til{\Lam}^{\s}_{\om,j}{a(D)}^{\frac{1}2}\tht \rb=\lb [\til{\Lam}^{\s+\frac{\be}2-1}_{\om p^{1/2},j},\bdy^{\perp}_{\ell}q]\bdy_{\ell}\tht,\til{\Lam}^{\s}_{\om,j}{a(D)}^{\frac{1}2}\tht \rb,\\
    I_{4}^\s&=-\lb \nabla^{\perp}q\cdot\nabla{a(D)}^{\frac{1}2}\til{\Lam}^{\s}_{\om,j} \tht,\til{\Lam}^{\s}_{\om,j}{a(D)}^{1/2}{\tht}\rb=0.
\end{align*}

We treat $I^\s_1$ as we did with $I^\s$ from \eqref{def:I:decomp:b} in the case $\be\in[0,1]$. In particular, we apply \cref{lem:commutator4} with $r=\s-1$, $(s, \bar{s})=(\s-1,1)$, $(\om_1,\til{\om}_1)=(\om,m_1^{\frac{\gam}{2}}p^{-1}\om)$, $(\om_2, \tilde{\om}_2)=(p^{-1}\om,m_1^{\frac{\gam}{2}}\om)$, $(\om_3, \tilde{\om}_3)=(\om,m_1^{\frac{\gam}{2}}p^{-1}\om)$, and $\Gam_1=\Gam_2=\Gam_3=m^{\frac{\gam}2}_1$, to obtain
\begin{align}\label{est:I1}
    |I_{1}^\s|\le& C_\lam c_jm_1(2^j)^{\frac{\gam}{2}} \left\{\nrm{\til{\tht}}_{\Hdot^{\s}_\om}\nrm{m_1(D)^{\frac{\gam}{2}}p(D)^{-1}\Lam \til{v}}_{\Hdot^{1}_\om}+ \nrm{p(D)^{-1}\Lam\til{ v}}_{H^1_\om}\nrm{m_1(D)^{\frac{\gam}{2}}\til{\tht}}_{\Hdot^{\s}_\om}\right.\notag\\
    &\hspace{15 em}\left.+\nrm{\til{\tht}}_{\Hdot^{\s}_\om}\nrm{m_1(D)^{\frac{\gam}{2}}p(D)^{-1}\Lam \til{v}}_{\Hdot^{1}_\om}\right\}\nrm{\til{\Lam}^{\s}_{\om,j}\tht}_{L^2}\notag\\
    \le& C_\lam c_{j}\left( \nrm{\til{\tht}}_{\Hdot^{\s}_\om}\nrm{m_1(D)^{\frac{\gam}{2}}\til{q}}_{\Hdot^{1+\be}_\om}+\nrm{\til{q}}_{H^{1+\be}_\om}\nrm{m_1(D)^{\frac{\gam}{2}}\til{\tht}}_{\Hdot^{\s}_\om}\right)\nrm{m_1(D)^{\frac{\gam}{2}}\til{\Lam}^{\s}_{\om,j}\tht}_{L^2},
    \end{align}
provided that 
    \begin{align}\label{cond:Gam:I1s}
         \begin{split}
     &\sup_{y>0}\frac{p(y)}{{m^{\gam}_1(y)}}\left(\int_0^1\frac{r^{1-2(\s-1)}}{\om^2(y r)}dr\right)^{\frac{1}{2}},\quad\sup_{y>0}\frac{1}{{m^{\gam}_1(y)}}\left(\int_{0}^{y}\frac{rp^2(r)}{(1+r^2)\om^{2}(r)}dr\right)^{\frac{1}{2}},\quad \sup_{y>0}\frac{p(y)}{\om(y)m^{\gam}_1(y)}<\infty.
        \end{split}
    \end{align}

 For $I^\s_3$, we apply \cref{lem:commutator4} with $r+1=\s+\be/2-1$, 
 ($s,\bar{s})=(\s-1,\be-1$), $(\om_1,\til{\om}_1)= (\om_2,\tilde{\om}_2)=(\om_3,\tilde{\om}_3)=(\om, m_1^{\frac{\gam}{2}}\om)$, and $\Gam_1=\Gam_2=\Gam_3=m^{\frac{\gam}2}_1p^{-\frac{1}{2}}$, to obtain
\begin{align}\label{est:I3}
    |I_{3}^\s|\le& C_\lam c_jm_1(2^j)^{\frac{\gam}{2}}p(2^j)^{-\frac{1}{2}}2^{(1-\frac{\be}{2})j} \left\{\nrm{\til{\tht}}_{\Hdot^{\s}_\om}\nrm{m_1(D)^{\frac{\gam}{2}}\nabla^{\perp}\Lam \til{q}}_{\Hdot^{\be-1}_\om}+ \nrm{\nabla^{\perp}\Lam \til{q}}_{H^{\be-1}_\om}\nrm{m_1(D)^{\frac{\gam}{2}}\til{\tht}}_{\Hdot^{\s}_\om}\right.\notag\\
    &\hspace{15 em}\left.+\nrm{\til{\tht}}_{\Hdot^{\s}_\om}\nrm{m_1(D)^{\frac{\gam}{2}}\nabla^{\perp}\Lam \til{q}}_{\Hdot^{\be-1}_\om}\right\}\nrm{\til{\Lam}^{\s}_{\om,j}a(D)^{\frac{1}2}\tht}_{L^2}\notag\\
     \le& C_\lam c_{j}\left( \nrm{\til{\tht}}_{\Hdot^{\s}_\om}\nrm{m_1(D)^{\frac{\gam}{2}}\til{q}}_{\Hdot^{1+\be}_\om}+\nrm{\til{q}}_{H^{1+\be}_\om}\nrm{m_1(D)^{\frac{\gam}{2}}\til{\tht}}_{\Hdot^{\s}_\om}\right)\nrm{m_1(D)^{\frac{\gam}{2}}\til{\Lam}^{\s}_{\om,j}\tht}_{L^2}.
\end{align}
Then \eqref{cond:om:Gam:a}, \eqref{cond:om:Gam:b} become
    \begin{align}\label{cond:Gam:I3}
         \begin{split}
     &\sup_{y>0}\frac{p(y)}{{m^{\gam}_1(y)}}\left(\int_0^1\frac{r^{1-2(\s-1)}}{\om^2(y r)}dr\right)^{\frac{1}{2}},\qquad \sup_{y>0}\frac{p(y)}{\om(y)m^{\gam}_1(y)}<\infty,\\
     &\sup_{y>0}\frac{p(y)}{{m^{\gam}_1(y)}}\left({\mathbbm{1}}_{(1,2)}(\be)\int_0^1\frac{r^{1-2(\be-1)}}{\om^2(y r)}dr+{\mathbbm{1}}_{\{2\}}(\be)\int_{0}^{y}\frac{r}{(1+r^2)\om^{2}(r)}dr\right)^{\frac{1}{2}}<\infty.
        \end{split}
    \end{align}

\subsubsection*{Subcase 2b: $\s\in[2,1+\be]$} From \eqref{def:mod:flux} and the facts that $\nabla^\perp q$ is divergence-free and $a(D)\bdy_\ell$ is skew self-adjoint, we may instead re-write $I^\s$ as
\begin{align}
    I^\s&=-\lb \til{\Lam}^{\s}_{\om,j}(\nabla^{\perp}a(D)q\cdot\nabla \tht),\til{\Lam}^{\s}_{\om,j}{\tht}\rb+\lb \til{\Lam}^{\s}_{\om,j}a(D)(\nabla^{\perp}q\cdot\nabla \tht),\til{\Lam}^{\s}_{\om,j}{\tht}\rb\notag\\
    &=J_{1}^\s+J_{2}^\s+J_{3}^\s+J_{4}^\s+J_{5}^\s,\label{def:I:decomp:d}
\end{align}
where
\begin{align}
	J_{1}^\s=&-\left\{\lb (\nabla^{\perp}a(D)\til{\Lam}^{\s}_{\om,j}{q} \cdot \nabla) \tht,\til{\Lam}^{\s}_{\om,j} \tht \rb -\lb \nabla^{\perp}a(D)\cdot(\til{\Lam}^{\s}_{\om,j}{q}\nabla \tht),\til{\Lam}^{\s}_{\om,j} \tht \rb\right\}\notag\\
	=&-\lb [\bdy_{\ell}^{\perp}a(D),\bdy_{\ell} \tht]\til{\Lam}^{\s}_{\om,j}q,\til{\Lam}^{\s}_{\om,j}\tht\rb\notag\\
	J_{2}^\s=&-\lb \nabla^{\perp}a(D){q}\cdot {\nabla} \til{\Lam}^{\s}_{\om,j} \tht,\til{\Lam}^{\s}_{\om,j} \tht \rb=0\notag\\
	J_{3}^\s=&-\left \{ \lb \til{\Lam}^{\s}_{\om,j}(\nabla^{\perp}a(D){q} \cdot \nabla \theta),\til{\Lam}^{\s}_{\om,j} \tht  \rb
	- \lb (\nabla^{\perp}a(D)\til{\Lam}^{\s}_{\om,j}{{q}} \cdot \nabla) \tht,\til{\Lam}^{\s}_{\om,j} \tht \rb  -\lb \nabla^{\perp}a(D){q}\cdot {\nabla} \til{\Lam}^{\s}_{\om,j} \tht,\til{\Lam}^{\s}_{\om,j} \tht \rb \right\}\notag\\
	J_{4}^\s=&-\lb (\nabla^{\perp}{q}\cdot {\nabla} a(D)^{\frac{1}2}\til{\Lam}^{\s}_{\om,j} \tht),a(D)^{\frac{1}2}\til{\Lam}^{\s}_{\om,j} \tht \rb=0\notag\\
	J_{5}^\s=&\ \lb \til{\Lam}^{\s}_{\om,j} a(D)(\nabla^{\perp}q\cdot\nabla \tht),\til{\Lam}^{\s}_{\om,j}{\tht}\rb
	- \lb \nabla^{\perp}a(D)\cdot(\til{\Lam}^{\s}_{\om,j}{q}\nabla \tht),\til{\Lam}^{\s}_{\om,j} \tht \rb \notag\\
 &\quad-\lb (\nabla^{\perp}{q}\cdot {\nabla} a(D)^{\frac{1}2}\til{\Lam}^{\s}_{\om,j} \tht),a(D)^{\frac{1}2}\til{\Lam}^{\s}_{\om,j} \tht \rb \notag
\end{align}
We observe as in \cite{HuKukavicaZiane2015} that we may write $J_{3}^\s$ as a double commutator. Indeed, for any $\widetilde{\s}\geq2$, we have
	\begin{align}\label{split:Lam:rewrite}
	{\Lam}^{\widetilde{\s}}f={\Lam}^{\widetilde{\s}-2}(-\De)f=-({\Lam}^{\widetilde{\s}-2}\bdy_{l})\bdy_{l}f.
	\end{align}
Since $\s\geq2$, we may apply \eqref{split:Lam:rewrite}, so that by the product rule and \eqref{eq:v:cancel}, we have
	\begin{align*}
	    J_{3}^\s=&-\lb\til{\Lam}^{\s-2}_{\om,j}\bdy_{l}(\nabla^{\perp}a(D)\bdy_{l}{q} \cdot \nabla \theta),\til{\Lam}^{\s}_{\om,j} \tht ,\rb+\lb (\nabla^{\perp}a(D)\til{\Lam}^{\s-2}_{\om,j}\bdy_{l}\bdy_{l}{{q}} \cdot \nabla) \tht,\til{\Lam}^{\s}_{\om,j} \tht \rb\\
	    &-\lb\til{\Lam}^{\s-2}_{\om,j}\bdy_{l}(\nabla^{\perp}a(D){q} \cdot \nabla \bdy_{l}\theta),\til{\Lam}^{\s}_{\om,j} \tht ,\rb+\lb (\nabla^{\perp}a(D){q}\cdot {\nabla} \til{\Lam}^{\s-2}_{\om,j}\bdy_{l}\bdy_{l} \tht),\til{\Lam}^{\s}_{\om,j} \tht \rb\\
	    =&-\lb[\til{\Lam}_{\om,j}^{\s-2}\bdy_{l},\bdy_{\ell}\tht]\bdy_{\ell}^{\perp}\bdy_{l}a(D)q,\til{\Lam}^{\s}_{\om,j} \tht\rb-\lb [\til{\Lam}_{\om,j}^{\s-2}\bdy_{l},\bdy_{\ell}^{\perp}a(D)q]\bdy_{\ell}\bdy_{l}\tht ,\til{\Lam}_{\om,j}^{\s}\tht\rb=J^{\s}_{3,a}+J^{\s}_{3,b}.
	\end{align*}
Similarly, we can express $J_{5}^\s$ as
\begin{align*}
    J_{5}^\s&=\lb[\til{\Lam}_{\om,j}^{\s-2}\bdy_{l},\bdy_{\ell}\tht]\bdy_{\ell}^{\perp}\bdy_{l} q,a(D)\til{\Lam}^{\s}_{\om,j} \tht\rb+\lb [\til{\Lam}_{\om p^{1/2},j}^{\s+\frac{\be}2-3}\bdy_{l},\bdy_{\ell}^{\perp} q]\bdy_{\ell}\bdy_{l}\tht ,a(D)^{1/2}\til{\Lam}_{\om,j}^{\s}\tht\rb\notag\\
    &=J^{\s}_{5,a}+J^{\s}_{5,b}.
\end{align*}
We will now estimate terms $J_{1}^\s, J_{3,a}^\s, J_{3,b}^\s, J_{5,a}^\s, J_{5,b}^\s$.

Applying \cref{lem:commutator3} with $s=2-\be$, $\eps=\be+1-\s$, $\Gamma=m_1^{\gam}p^{-1}$, and Bernstein's inequality, we obtain
\begin{align}\label{est:J1}
    |J_{1}^\s|&\le Cm_1(2^j)^{\gam}2^{(\be+1-\s)j}\Sob{\nabla \tht}{{H}^{\s-1}_\om}\Sob{\til{\Lam}^{\s}_{\om,j}q}{L^2}\Sob{\til{\Lam}^{\s}_{\om,j}\tht}{L^2}\notag\\
    &\le Cc_{j}\Sob{\tht}{H^{\s}_\om}\Sob{m_1(D)^{\frac{\gam}{2}}\til{q}}{\Hdot^{1+\be}_\om}\Sob{m_1(D)^{\frac{\gam}{2}}\til{\Lam}^{\s}_{\om,j}\tht}{L^2},
\end{align}
where
    \[
        c_{j}=\frac{\Sob{m_1(D)^{\frac{\gam}{2}}\til{\Lam}^{\s}_{\om,j}q}{\Hdot^{\be+1-\s}}}{\Sob{m_1(D)^{\frac{\gam}{2}}\til{q}}{\Hdot^{1+\be}_\om}}\in \ell^{2}(\mathbb{Z}),
    \]
provided that
    \begin{align}\label{cond:Gam:J1}
    \sup_{y>0}\left\{\frac{p(y)}{m^{\gam}_1(y)} \left(\int_{0}^{y}\frac{r}{(1+r^2)\om^2(r)}dr\right)^{\frac{1}2},\quad \frac{p_a(y)}{\om(y)m^{\gam}_1(y)}\right\}<\infty,
    \end{align}

For $J^\s_{3,a}$, we apply \cref{lem:commutator4} with $r=\s-2$, $(s,\bar{s})=(1,\s-2) $, $(\om_1,\til{\om}_1)=(p^{-1}\om, m_1^{\frac{\gam}{2}}\om)$,  $(\om_2, \til{\om}_2)=(\om, m_1^{\frac{\gam}{2}}p^{-1}\om)$, $(\om_3, \til{\om}_3)=(p^{-1}\om, m_1^{\frac{\gam}{2}}\om)$, and $\Gam_1=\Gam_2=\Gam_3=m^{\frac{\gam}2}_1$, we obtain
\begin{align}\label{est:J3a}
    |J^{\s}_{3,a}|\le &C_\lam c_{j}m_1(2^j)^{\frac{\gam}{2}}\sum_{\ell,l}\left\{\nrm{\bdy_{\ell}^{\perp}\bdy_{l}p(D)^{-1}a(D)\til{q}}_{H^{1}_\om}\nrm{m_1(D)^{\frac{\gam}{2}}\Lam \bdy_{\ell}\til{\tht}}_{\Hdot^{\s-2}_\om}\right.\notag\\&\hspace{7 em}\left.+\nrm{\Lam \bdy_{\ell}\til{\tht}}_{H^{\s-2}_\om}\nrm{\bdy_{\ell}^{\perp}\bdy_{l}m_1(D)^{\frac{\gam}{2}}p(D)^{-1}a(D)\til{q}}_{\Hdot^{1}_\om}\right. \notag \\
    &\hspace{7 em}\left. +\nrm{\bdy_{\ell}^{\perp}\bdy_{l}p(D)^{-1}a(D)\til{q}}_{\Hdot^{1}_\om}\nrm{m_1(D)^{\frac{\gam}{2}}\Lam \bdy_{\ell}\til{\tht}}_{\Hdot^{\s-2}_\om} \right\}\nrm{\til{\Lam}^{\s}_{\om,j}\tht}_{L^2}\notag\\
 \le& C_\lam c_{j}\left( \nrm{\til{\tht}}_{H^{\s}_\om}\nrm{m_1(D)^{\frac{\gam}{2}}\til{q}}_{\Hdot^{1+\be}_\om}+\nrm{\til{q}}_{H^{1+\be}_\om}\nrm{m_1(D)^{\frac{\gam}{2}}\til{\tht}}_{\Hdot^{\s}_\om}\right)\nrm{m_1(D)^{\frac{\gam}{2}}\til{\Lam}^{\s}_{\om,j}\tht}_{L^2},
\end{align}
provided that 
    \begin{align}\label{cond:Gam:Js3a}
         \begin{split}
     &\sup_{y>0}\frac{p(y)}{{m^{\gam}_1(y)}}\left({\mathbbm{1}}_{(-\infty,1)}(\sigma-2)\int_0^1\frac{r^{1-2(\s-2)}}{\om^2(y r)}dr+{\mathbbm{1}}_{[1,\infty)}(\s-2)\int_{0}^{y}\frac{r}{(1+r^2)\om^{2}(r)}dr\right)^{\frac{1}{2}}<\infty,\\
     &\sup_{y>0}\frac{1}{{m^{\gam}_1(y)}}\left(\int_{0}^{y}\frac{rp^2(r)}{(1+r^2)\om^{2}(r)}dr\right)^{\frac{1}{2}},\qquad \sup_{y>0}\frac{p(y)}{\om(y)m^{\gam}_1(y)}<\infty.
        \end{split}
    \end{align}

For $J^{\s}_{3,b}$, we apply \cref{lem:commutator4} with $r=\s-2$, $(s, \bar{s})=(\s-2,1)$,  $(\om_1,\til{\om}_1)=(\om,m_1^{\frac{\gam}{2}}p^{-1}\om)$, $(\om_2, \tilde{\om}_2)=(p^{-1}\om,m_1^{\frac{\gam}{2}}\om)$, $(\om_3, \tilde{\om}_3)=(\om,m_1^{\frac{\gam}{2}}p^{-1}\om)$, and $\Gam_1=\Gam_2=\Gam_3=m^{\frac{\gam}2}_1$, to obtain
\begin{align}\label{est:J3b}
    |J^{\s}_{3,b}| \le& C_\lam c_{j}m_1(2^j)^{\frac{\gam}{2}}\sum_{\ell,l}\left\{\nrm{ \bdy_{\ell}\bdy_{l}\til{\tht}}_{H^{\s-2}_\om}\nrm{\bdy_{\ell}^{\perp}m_1(D)^{\frac{\gam}{2}}p(D)^{-1}\Lam a(D)\til{q}}_{\Hdot^{1}_\om}\right.\notag\\&\hspace{7 em}\left.+\nrm{\bdy_{\ell}^{\perp}p(D)^{-1}\Lam a(D)\til{q}}_{H^{1}_\om}\nrm{m_1(D)^{\frac{\gam}{2}} \bdy_{\ell}\bdy_{l}\til{\tht}}_{\Hdot^{\s-2}_\om}\right. \notag\\
    &\hspace{7 em} \left. +\nrm{ \bdy_{\ell}\bdy_{l}\til{\tht}}_{\Hdot^{\s-2}_\om}\nrm{\bdy_{\ell}^{\perp}m_1(D)^{\frac{\gam}{2}}p(D)^{-1}\Lam a(D)\til{q}}_{\Hdot^{1}_\om}\right\}\nrm{\til{\Lam}^{\s}_{\om,j}\tht}_{L^2}\notag\\
    \le& C_\lam c_{j}\left( \nrm{\til{\tht}}_{H^{\s}_\om}\nrm{m_1(D)^{\frac{\gam}{2}}\til{q}}_{\Hdot^{1+\be}_\om}+\nrm{\til{q}}_{H^{1+\be}_\om}\nrm{m_1(D)^{\frac{\gam}{2}}\til{\tht}}_{\Hdot^{\s}_\om}\right)\nrm{m_1(D)^{\frac{\gam}{2}}\til{\Lam}^{\s}_{\om,j}\tht}_{L^2},
\end{align}
provided that \eqref{cond:Gam:Js3a} holds.

For $J^{\s}_{5,a}$, we apply \cref{lem:commutator4} with 
$r=\s-2$, $(s,\bar{s})=(\be-1,\s-2)$, $(\om_1,\til{\om}_1)= (\om_2,\tilde{\om}_2)=(\om_3,\tilde{\om}_3)=(\om, m_1^{\frac{\gam}{2}}\om)$, and $\Gam_1=\Gam_2=\Gam_3=m^{\frac{\gam}2}_1p^{-1}$, to obtain
\begin{align}\label{est:J5a}
    |J^{\s}_{5,a}|
    \le &C_\lam c_{j}m_1(2^j)^{\frac{\gam}{2}}p(2^j)^{-1}2^{(2-\be)j}\sum_{\ell,l}\left\{\nrm{\bdy_{\ell}^{\perp}\bdy_{l}\til{q}}_{H^{\be-1}_\om}\nrm{\bdy_{\ell}m_1(D)^{\frac{\gam}{2}}\Lam\til{\tht}}_{\Hdot^{\s-2}_\om}\right. \notag\\
    &\hspace{13 em}\left.+\nrm{\bdy_{\ell}\Lam \til{\tht}}_{H^{\s-2}_\om}\nrm{m_1(D)^{\frac{\gam}{2}}\bdy_{\ell}^{\perp}\bdy_{l}\til{q}}_{\Hdot^{\be-1}_\om}\right. \notag\\
    &\hspace{13 em}\left. +\nrm{\bdy_{\ell}^{\perp}\bdy_{l}\til{q}}_{\Hdot^{\be-1}}\nrm{\bdy_{\ell}m_1(D)^{\frac{\gam}{2}}\Lam \til{\tht}}_{\Hdot^{\s-2}_\om}\right\}\nrm{\til{\Lam}^{\s}_{\om,j}a(D)\tht}_{L^2}\notag\\
     \le& C_\lam c_{j}\left( \nrm{\til{\tht}}_{H^{\s}_\om}\nrm{m_1(D)^{\frac{\gam}{2}}\til{q}}_{\Hdot^{1+\be}_\om}+\nrm{\til{q}}_{H^{1+\be}_\om}\nrm{m_1(D)^{\frac{\gam}{2}}\til{\tht}}_{\Hdot^{\s}_\om}\right)\nrm{m_1(D)^{\frac{\gam}{2}}\til{\Lam}^{\s}_{\om,j}\tht}_{L^2},
\end{align}
provided that 
    \begin{align}\label{cond:Gam:Js5a}
         \begin{split}
        &\sup_{y>0}\frac{p(y)}{{m^{\gam}_1(y)}}\left({\mathbbm{1}}_{(1,2)}(\be)\int_0^1\frac{r^{1-2(\be-1)}}{\om^2(y r)}dr+{\mathbbm{1}}_{\{2\}}(\be)\int_{0}^{y}\frac{r}{(1+r^2)\om^{2}(r)}dr\right)^{\frac{1}{2}}<\infty,\\
        &\sup_{y>0}\frac{p(y)}{{m^{\gam}_1(y)}}\left({\mathbbm{1}}_{(-\infty,1)}(\sigma-2)\int_0^1\frac{r^{1-2(\s-2)}}{\om^2(y r)}dr+{\mathbbm{1}}_{[1,\infty)}(\s-2)\int_{0}^{y}\frac{r}{(1+r^2)\om^{2}(r)}dr\right)^{\frac{1}{2}}<\infty,\\
        &\sup_{y>0}\frac{p(y)}{\om(y)m^{\gam}_1(y)}<\infty.
        \end{split}
    \end{align}

For $J^{\s}_{5,b}$, we apply \cref{lem:commutator4} with $r=\s+\be/2-3$, $(s,\bar{s})=(\s-2,\be-1)$, $(\om_1,\til{\om}_1)$ $=$ $(\om_2,\tilde{\om}_2)$ $=$ $(\om_3,\tilde{\om}_3)$ $=$ $(\om, m_1^{\frac{\gam}{2}}\om)$, and $\Gam_1=\Gam_2=\Gam_3=m^{\frac{\gam}2}_1p^{-\frac{1}{2}}$, to obtain
\begin{align}\label{est:J5b}
    |J^{\s}_{5,b}|\le& C_\lam c_j  m_1(2^j)^{\frac{\gam}{2}}p(2^j)^{-\frac{1}{2}}2^{(1-\frac{\be}{2})j}\sum_{\ell,l}\left\{ \nrm{\bdy_{\ell}\bdy_{l}\til{\tht}}_{H^{\s-2}_\om}\nrm{m_1(D)^{\frac{\gam}{2}}\bdy_{\ell}^{\perp}\Lam \til{q}}_{\Hdot^{\be-1}_\om}\right.\notag \\
    &\hspace{14 em}\left.+\nrm{\bdy_{\ell}^{\perp}\Lam \til{q}}_{H^{\be-1}_\om}\nrm{m_1(D)^{\frac{\gam}{2}}\bdy_{\ell}\bdy_{l}\til{\tht}}_{\Hdot^{\s-2}_\om}\right.\notag \\
    &\hspace{14 em}\left.+\nrm{\bdy_{\ell}\bdy_{l}\til{\tht}}_{\Hdot^{\s-2}_\om}\nrm{m_1(D)^{\frac{\gam}{2}}\bdy_{\ell}^{\perp}\Lam \til{q}}_{\Hdot^{\be-1}_\om}\right\}\nrm{a(D)^{1/2}\til{\Lam}^{\s}_{\om,j}\tht}_{L^2}\notag\\
        \le& C_\lam c_{j}\left( \nrm{\til{\tht}}_{H^{\s}_\om}\nrm{m_1(D)^{\frac{\gam}{2}}\til{q
        }}_{\Hdot^{1+\be}_\om}+\nrm{\til{q}}_{H^{1+\be}_\om}\nrm{m_1(D)^{\frac{\gam}{2}}\til{\tht}}_{\Hdot^{\s}_\om}\right)\nrm{m_1(D)^{\frac{\gam}{2}}\til{\Lam}^{\s}_{\om,j}\tht}_{L^2},
\end{align}
provided that 
     \begin{align}\label{cond:Gam:Js5b}
         \begin{split}
     &\sup_{y>0}\frac{p(y)}{{m^{\gam}_1(y)}}\left({\mathbbm{1}}_{(-\infty,1)}(\sigma-2)\int_0^1\frac{r^{1-2(\s-2)}}{\om^2(y r)}dr+{\mathbbm{1}}_{[1,\infty)}(\s-2)\int_{0}^{y}\frac{r}{(1+r^2)\om^{2}(r)}dr\right)^{\frac{1}{2}}<\infty,\\
     &\sup_{y>0}\frac{p(y)}{{m^{\gam}_1(y)}}\left({\mathbbm{1}}_{(-\infty,1)}(\be-1)\int_0^1\frac{r^{1-2(\be-1)}}{\om^2(y r)}dr+{\mathbbm{1}}_{[1,\infty)}(\be-1)\int_{0}^{y}\frac{r}{(1+r^2)\om^{2}(r)}dr\right)^{\frac{1}{2}}<\infty,\\
     & \sup_{y>0}\frac{p(y)}{\om(y)m^{\gam}_1(y)}<\infty.
        \end{split}
    \end{align}

\subsubsection*{Concluding Estimates for Case 2: $\be \in (1,2]$} As in Case 1, we observe that the conditions stated in \eqref{cond:Gam:I1s}, \eqref{cond:Gam:I3}, \eqref{cond:Gam:J1}, \eqref{cond:Gam:Js3a}, \eqref{cond:Gam:Js5a}, \eqref{cond:Gam:Js5b} can be reduced to
    \begin{align}\label{cond:Gam:J:final}
    \begin{split}
     \sup_{y>0}\left\{\frac{1}{m_1^\gam(y)}\left(\int_{0}^{y}\frac{r(p^2(y)+p^2(r))}{(1+r^2)\om^{2}(r)}dr\right)^{\frac{1}{2}},\quad \frac{p_a(y)\om_b(y)}{m^{\gam}_1(y)}\right\}<\infty.
        \end{split}
    \end{align}
Upon returning to \eqref{est:balance:Hsigma:final}, we may now apply either \eqref{est:I1}, \eqref{est:I3} or \eqref{est:J1}, \eqref{est:J3a}, \eqref{est:J3b}, \eqref{est:J5a}, \eqref{est:J5b}, then sum in $j$, while invoking \eqref{eq:Sob:Bes}, \eqref{eq:Sob:equiv}, to obtain
    \begin{align}\label{est:Hs:b1}
        \frac{d}{dt}\Sob{\til{\tht}}{\Hdot^{\s}_\om}^2+\frac{7}4\Sob{m_1(D)^{\frac{1}2}\til{\tht}}{\Hdot^\s_\om}^{2}&\leq C_\lam \left(\nrm{\til{\tht}}_{H^{\s}_\om}\nrm{m_1(D)^{\frac{\gam}{2}}\til{q
        }}_{\Hdot^{1+\be}_\om}+\nrm{\til{q}}_{H^{1+\be}_\om}\nrm{m_1(D)^{\frac{\gam}{2}}\til{\tht}}_{\Hdot^{\s}_\om }\right)\nrm{m_1(D)^{\frac{\gam}{2}}\til{\tht}}_{\Hdot^\s_\om}\notag\\
        &\quad+C\left(\Sob{m_1(D)^{-\frac{1}{2}}\til{G}}{\Hdot^\s_\om}^2+\Sob{\til{\tht}}{\Hdot^\s_\om}^2\right),
    \end{align}
for all $\s\in[1,1+\be]$. We lastly combine \eqref{est:Hs:b1} with \eqref{est:L2:b} and \cref{cor:norm:equiv} to deduce
    \begin{align}\label{est:Hs:b1:final}
        \frac{d}{dt}&\left(\Sob{\til{\tht}}{\Hdot^{\s}_\om}^2+\Sob{\tht}{L^2}^2\right)+\frac{7}4\left(\Sob{m_1(D)^{\frac{1}2}\til{\tht}}{\Hdot^\s_\om}^{2}+\Sob{m_1(D)^{\frac{1}2}\tht}{L^2}^2\right)\notag\\
        &\leq C \nrm{\til{\tht}}_{H^{\s}_\om}\nrm{m_1(D)^{\frac{\gam}{2}}\til{q
        }}_{\Hdot^{1+\be}_\om}\nrm{m_1(D)^{\frac{\gam}{2}}\til{\tht}}_{\Hdot^\s_\om}+C\nrm{\til{q}}_{H^{1+\be}_\om}\left(\nrm{m_1(D)^{\frac{\gam}{2}}\til{\tht}}_{\Hdot^{\s}_\om}^2+\Sob{m_1(D)^{\frac{\gam}2}{\tht}}{L^2}^2\right)\notag\\
        &\quad+C\left(\Sob{m_1(D)^{-\frac{1}{2}}\til{G}}{H^\s_\om}^2+\Sob{\til{\tht}}{H^\s_\om}^2\right),
    \end{align}
where we used the facts that $\Sob{\tht}{L^2}\leq \Sob{m_1(D)^{\frac{\gam}2}\tht}{L^2}$ and $\Sob{\tht}{H^\eps_\om}\leq\Sob{m_1(D)^{\frac{\gam}2}\til{\tht}}{H^\s_\om}$, for $\eps$ sufficiently small.

\subsection*{Summary of estimates in $\dot{E}^{\lam}_{\nu,\s,\om}$:}  

Given $\be\in[0,2]$, suppose $\s$ satisfies \eqref{sigma:range} (omitting the case $\be=0$, $\lam=0$, $\s=-1$, for now, since it is assumed that $\lam_1>0$). Observe that \eqref{cond:Gam:J:final} implies \eqref{cond:Gam:I0:final}, so that we ultimately reduce these conditions to
\begin{align}\label{cond:summary}
         \begin{split}
     \sup_{y>0}\left\{\frac{1}{m_1^\gam(y)}\left(\int_{0}^{y}\frac{r(p^2(y)+p^2(r))}{(1+r^2)\om^{2}(r)}dr\right)^{\frac{1}{2}},\quad \frac{p_a(y)\om_b(y)}{m^{\gam}_1(y)}\right\}<\infty.
        \end{split}
    \end{align}
Then for $\be \in [0,1]$, we have

     \begin{align}\label{est:Hs:summary:1}
        \frac{d}{dt}\Sob{\til{\tht}}{\Hdot^{\s}_\om\cap L^2_{\om}}^2&+\frac{7}4\Sob{m_1(D)^{\frac{1}2}\til{\tht}}{\Hdot^\s_\om\cap L^2_{\om}}^{2}\notag\\
        &\quad\leq C_\lam \left(\nrm{\til{\tht}}_{\Hdot^{\s}_\om\cap L^2_{\om}}\nrm{m_1(D)^{\frac{\gam}{2}}\til{q
        }}_{\Hdot^{1+\be}_\om}+\nrm{\til{q}}_{H^{1+\be}_\om}\nrm{m_1(D)^{\frac{\gam}{2}}\til{\tht}}_{\Hdot^{\s}_\om \cap L^2_{\om}}\right)\nrm{m_1(D)^{\frac{\gam}{2}}\til{\tht}}_{\Hdot^\s_\om\cap L^2_{\om}}\notag\\
        &\qquad+C\left(\Sob{m_1(D)^{-\frac{1}{2}}\til{G}}{\Hdot^\s_\om\cap L^2_{\om}}^2+\Sob{\til{\tht}}{\Hdot^\s_\om\cap L^2_{\om}}^2\right),
    \end{align}
and for $\be\in(1,2]$, we have
\begin{align}\label{est:Hs:summary:2}
        \frac{d}{dt}&\left(\Sob{\til{\tht}}{\Hdot^{\s}_\om}^2+\Sob{\tht}{L^2}^2\right)+\frac{7}4\left(\Sob{m_1(D)^{\frac{1}2}\til{\tht}}{\Hdot^\s_\om}^{2}+\Sob{m_1(D)^{\frac{1}2}\tht}{L^2}^2\right)\notag\\
        &\leq C \nrm{\til{\tht}}_{H^{\s}_\om}\nrm{m_1(D)^{\frac{\gam}{2}}\til{q
        }}_{\Hdot^{1+\be}_\om}\nrm{m_1(D)^{\frac{\gam}{2}}\til{\tht}}_{\Hdot^\s_\om}+C\nrm{\til{q}}_{H^{1+\be}_\om}\left(\nrm{m_1(D)^{\frac{\gam}{2}}\til{\tht}}_{\Hdot^{\s}_\om}^2+\Sob{m_1(D)^{\frac{\gam}2}{\tht}}{L^2}^2\right)\notag\\
        &\quad+C\left(\Sob{m_1(D)^{-\frac{1}{2}}\til{G}}{H^\s_\om}^2+\Sob{\til{\tht}}{H^\s_\om}^2\right),
    \end{align}
 
 When $\be\in[0,1]$, observe that by interpolation \eqref{est:interpolation:weighted} and Young's inequality we obtain
    \begin{align}
    C&\nrm{\til{\tht}}_{\Hdot^{\s}_{\om}\cap L^2_{\om}}\nrm{m_1(D)^{\frac{\gamma}{2}}\til{q}}_{\Hdot^{1+\be}_\om}\nrm{m_1(D)^{\frac{\gamma}{2}}\til{\tht}}_{\Hdot^{\s}_{\om}\cap L^2_{\om}}\label{est:Hs:a3}\\
            &\leq  \nrm{m_1(D)^{\frac{\gamma}{2}}\til{q}}_{\Hdot^{1+\be}_\om}\nrm{\til{\tht}}_{\Hdot^{\s}_{\om}\cap L^2_{\om}}^{2-{\gam}}\nrm{m_1(D)^{\frac{1}{2}}\til{\tht}}_{\Hdot^{\s}_{\om}\cap L^2_{\om}}^{{\gam}}\notag\\
            &\leq \frac{1}8\Sob{m(D)^{\frac{1}2}\til{\tht}}{\Hdot^{\s}_{\om}\cap L^2_{\om}}^2+C\nrm{m_1(D)^{\frac{\gam}2}\til{q}}_{\Hdot^{1+\be}_{\om}}^{\frac{2}{2-\gam}}\nrm{\til{\tht}}_{\Hdot^{\s}_{\om}\cap L^2_{\om}}^2.\label{est:Hs:a2}
    \end{align}
and
    \begin{align}
       C\nrm{\til{q}}_{H^{1+\be}_\om} \nrm{m_1(D)^{\frac{\gam}{2}}\til{\tht}}_{\Hdot^{\s}_{\om}\cap L^2_{\om}}^2&\leq C\nrm{\til{q}}_{H^{1+\be}_\om}\Sob{m_1(D)^{\frac{1}2}\til{\tht}}{\Hdot^{\s}_{\om}\cap L^2_{\om}}^{{2\gam}}\Sob{\til{\tht}}{\Hdot^{\s}_{\om}\cap L^2_{\om}}^{2-2{\gam}}\label{est:Hs:a4}\\
       &\leq \frac{1}8\Sob{m_1(D)^{\frac{1}2}\til{\tht}}{\Hdot^{\s}_{\om}\cap L^2_{\om}}^2+C\Sob{\til{q}}{H^{1+\be}_\om}^{\frac{1}{1-\gam}}\Sob{\til{\tht}}{\Hdot^{\s}_{\om}\cap L^2_{\om}}^2,\label{est:Hs:a1}
    \end{align}
Applying \eqref{est:Hs:a1} and \eqref{est:Hs:a2} in \eqref{est:Hs:summary:1}, we arrive at
\begin{align}\label{est:Hs:final}
      \frac{d}{dt}\Sob{\tilde{\tht}}{\Hdot^{\s}_{\om}\cap L^2_{\om}}^2&+\frac{3}2\Sob{m_1(D)^{\frac{1}2}\til{\tht}}{\Hdot^{\s}_{\om}\cap L^2_{\om}}^{2}\notag\\
        &\quad\leq C \left(1+\Sob{\til{q}}{H^{1+\be}_\om}^{\frac{1}{1-\gam}}+\Sob{m_1(D)^{\frac{\gam}2}\til{q}}{\Hdot^{1+\be}_\om}^{\frac{2}{2-\gam}}\right)\Sob{\tilde{\tht}}{\Hdot^{\s}_{\om}\cap L^2_{\om}}^2+C\Sob{m_1(D)^{-\frac{1}{2}}\til{G}}{\Hdot^{\s}_{\om}\cap L^2_{\om}}^2.
    \end{align}
An application of Gronwall's inequality then yields
    \begin{align}\label{est:Gev:final}
     \sup_{0\leq t \leq T}\Bigg{(}\Sob{\tilde{\tht}(t)}{\Hdot^{\s}_{\om}\cap L^2_{\om}}^2&+\int_0^t \Sob{m_1(D)^{\frac{1}2}\tilde{\tht}(s)}{\Hdot^{\s}_{\om}\cap L^2_{\om}}^2ds\Bigg{)}\notag\\
     &\quad\le \tilde{C}_T(\be)\left(\Sob{\tht_0}{\Hdot^{\s}_{\om}\cap L^2_{\om}}^2+\int_0^T\Sob{m_1(D)^{-\frac{1}2}\tilde{G}(s)}{\Hdot^{\s}_{\om}\cap L^2_{\om}}^2ds\right),
    \end{align}
where
    \begin{align}\label{def:Gev:constant}
        \tilde{C}_T(\be)=\exp\left(C\int_0^T\left(1+\Sob{\tilde{q}(t)}{H^{1+\be}_\om}^{\frac{1}{1-\gam}}+\Sob{m_1(D)^{\frac{\gam}2}\tilde{q}(t)}{\Hdot^{1+\be}_\om}^{\frac{2}{2-\gam}}\right)dt\right),
    \end{align}
for some constant $C$ depending on $\be$.

When $\be\in(1,2]$, we may estimate the right-hand side of \eqref{est:Hs:summary:2} in a similar fashion to \eqref{est:Hs:summary:1}, except that we additionally invoke \cref{cor:norm:equiv}, in order to also deduce \eqref{est:Gev:final}.

\subsection{A priori estimates in $\Hdot^{\s}_{\om}$}
In order to obtain estimates in Sobolev spaces, we suppress the smoothing multiplier in \eqref{est:Hs:final} by formally setting $\lam\equiv0$; this evaluation is justified due to the way in which the constants depend on $\lam$ in the above estimates. Lastly, we recall that we have yet to treat the case $\be=0$, $\s=-1$, $\lam\equiv0$; this will also be done in this section.

Upon setting $\lam\equiv0$ in \eqref{est:Gev:final} we obtain
    \begin{align}\label{est:Hs:gronwall}
     \sup_{0\leq t \leq T}\Bigg{(}\Sob{{\tht}(t)}{\Hdot^{\s}_{\om}\cap L^2_{\om}}^2&+\int_0^t \Sob{m_1(D)^{\frac{1}2}\tht(s)}{\Hdot^{\s}_{\om}\cap L^2_{\om}}^2ds\Bigg{)}\notag\\
     &\quad\le C_T(\be)\left(\Sob{\tht_0}{\Hdot^{\s}_{\om}\cap L^2_{\om}}^2+\int_0^T\Sob{m_1(D)^{-\frac{1}2}G(s)}{\Hdot^{\s}_{\om}\cap L^2_{\om}}^2ds\right),
    \end{align}
where
\begin{align}\label{def:Hs:constant}
    C_T(\be)=\exp\left(C\int_0^T\left(1+\Sob{q(t)}{H^{1+\be}_\om}^{\frac{1}{1-\gam}}+\Sob{m_1(D)^{\frac{\gam}2}q(t)}{\Hdot^{1+\be}_\om}^{\frac{2}{2-\gam}}\right)dt\right),
\end{align}
for some constant $C$ depending on $\be$.

\subsubsection*{Case $\be=0$, $\s=-1$, $\lam\equiv0$} Since $\lam\equiv0$, we may drop the tilde notation from \eqref{eq:balance:Hsigma}. Furthermore, observe that from \eqref{eq:balance:Hsigma} we have
    \[  
        I^{-1}=\lb m_1(D)^{-\frac{\gam}2}{\Lam}^{-1}_{\om,j}\nabla\cdotp(v \tht),m_1(D)^{\frac{\gam}2}{\Lam}^{-1}_{\om,j}\tht\rb.
    \]
Then by Bernstein's inequalities and \eqref{est:omega}, we have
    \begin{align}\notag
        |I^{-1}|&\leq C\om(2^j)m_1(2^j)^{-\frac{\gam}2}\Sob{\lpj(v\tht)}{L^2}\Sob{m_1(D)^{\frac{\gam}2}\Lam_{\om,j}^{-1}\tht}{L^2}
    \end{align}
 
Applying \cref{thm:prod} with $(s,\bar{s})=(1,0)$, $(\om_1,\til{\om}_1)=(p^{-1}\om,m_1^{\frac{\gam}{2}}\om)$, $(\om_2, \til{\om}_2)=(\om, p^{-1}m_1^{\frac{\gam}{2}})$ , 
$(\om_3, \tilde{\om}_3)=(p^{-1}\om, m_1^{\frac{\gam}{2}}\om)$, and $\Gam_1=\Gam_2=\Gam_3=m_1^{\frac{\gam}2}$ we obtain 
\begin{align}\label{est:I:Hminus1}
    |I^{-1}|&\le Cc_j\left(\Sob{p(D)^{-1}v}{H^1_{\om}}\Sob{m_1(D)^{\frac{\gam}{2}}\tht}{L^2_{\om}}+\Sob{\tht}{L^2_{\om}}\Sob{p(D)^{-1}m_1(D)^{\frac{\gam}{2}}v}{\Hdot^1_\om}\right)\Sob{m_1(D)^{\frac{\gam}2}{\Lam}^{-1}_{\om, j} \tht}{L^2},
\end{align}
provided that 
\begin{align}\label{cond:Gam:I:minus1}
         \begin{split}
     \sup_{y>0}\frac{1}{{m^{\gam}_1(y)}}\left(\int_{0}^{y}\frac{rp^2(r)}{(1+r^2)\om^{2}(r)}dr\right)^{\frac{1}{2}},\quad \sup_{y>0}\frac{p(y)}{{m^{\gam}_1(y)}}\left(\int_0^1\frac{r}{\om^2(y r)}dr\right)^{\frac{1}{2}},\quad  \sup_{y>0}\frac{p(y)}{\om(y)m^{\gam}_1(y)}<\infty,
        \end{split}
    \end{align}
holds. As before, observe that \eqref{cond:Gam:I0:final} implies \eqref{cond:Gam:I:minus1}.

Recall that $v$ is given by \eqref{def:v}, so that
    \[
    \Sob{p(D)^{-1}v}{H^1_{\om}}\leq C\Sob{q}{\Hdot^{-1}_{\om}\cap L^2_\om},\quad \Sob{p(D)^{-1}m_1(D)^{\frac{\gam}{2}}v}{\Hdot^1_{\om}}\leq C\Sob{m_1(D)^{\frac{\gam}{2}}q}{L^2_{\om}}.
    \]
Upon returning to \eqref{est:Hs:a} with $\lam_1=0$ and $\s=0$, applying \eqref{est:I:Hminus1}, summing in $j$, and invoking \eqref{eq:Sob:Bes}, \eqref{eq:Sob:equiv}, we obtain
    \begin{align}\label{est:Hminus1}
    \frac{d}{dt}\Sob{\tht}{\Hdot^{-1}_{\om}\cap L^2_{\om}}^2+\frac{7}4\Sob{m(D)^{\frac{1}2}\tht}{\Hdot^{-1}_{\om}\cap L^2_{\om}}^{2}&\leq C\left(\Sob{q}{\Hdot^{-1}_{\om}\cap \dot{H}^1_\om}\Sob{m_1(D)^{\frac{\gam}{2}}\tht}{L^2_{\om}}+\Sob{\tht}{L^2_{\om}}\Sob{m_1(D)^{\frac{\gam}{2}}q}{H^1_{\om}}\right)\Sob{m_1(D)^{\frac{\gam}2}\tht}{\Hdot^{-1}_{\om}\cap L^2_\om}\notag\\
    &\quad +C\left(\Sob{m_1(D)^{-\frac{1}{2}}G}{\Hdot^{-1}_\om \cap L^2_\om}^2+\Sob{\tht}{\Hdot^{-1}_\om \cap L^2_\om}^2\right).
    \end{align}

Using Plancherel's theorem, \eqref{est:interpolation:weighted}, and Young's inequality we see that
    \begin{align}
    C\Sob{q}{\Hdot^{-1}_{\om}\cap \dot{H}^1_\om}\Sob{m_1(D)^{\frac{\gam}{2}}\tht}{L^2_{\om}}\Sob{m_1(D)^{\frac{\gam}2}\tht}{\Hdot^{-1}_{\om}\cap L^2_\om}&\leq C\Sob{q}{\Hdot^{-1}_{\om}\cap \dot{H}^1_\om}\Sob{m_1(D)^{\frac{1}2}\tht}{\Hdot^{-1}_\om\cap L^2_{\om}}^{2\gam}\Sob{\tht}{\Hdot^{-1}_\om\cap L^2_{\om}}^{2(1-\gam)}\notag\\
    &\leq \frac{1}8\Sob{m_1(D)^{\frac{1}2}\tht}{\Hdot^{-1}_\om\cap L^2_{\om}}^2+C\Sob{q}{\Hdot^{-1}_{\om}\cap \dot{H}^1_\om}^{\frac{1}{1-\gam}}\Sob{\tht}{\Hdot^{-1}_\om\cap L^2_{\om}}^2\label{est:Hminus1:nonlinear:a}\\
    C\Sob{\tht}{L^2_{\om}}\Sob{m_1(D)^{\frac{\gam}{2}}q}{H^1_{\om}}\Sob{m_1(D)^{\frac{\gam}2}\tht}{\Hdot^{-1}_{\om}\cap L^2_\om}&\leq C\Sob{m_1(D)^{\frac{\gam}{2}}q}{H^1_{\om}}\Sob{m_1(D)^{\frac{1}2}\tht}{\Hdot^{-1}_\om\cap L^2_{\om}}^{{\gam}}\Sob{\tht}{\Hdot^{-1}_\om\cap L^2_{\om}}^{{2-\gam}}\notag\\
    &\leq \frac{1}8\Sob{m_1(D)^{\frac{1}2}\tht}{\Hdot^{-1}_\om\cap L^2_{\om}}^2+C\Sob{m_1(D)^{\frac{\gam}{2}}q}{H^1_{\om}}^{\frac{2}{2-\gam}}\Sob{\tht}{\Hdot^{-1}_\om\cap L^2_{\om}}^2\label{est:Hminus1:nonlinear:b}
    \end{align}
    
Now we return to \eqref{est:Hminus1}, then apply \eqref{est:Hminus1:nonlinear:a}, \eqref{est:Hminus1:nonlinear:b} to deduce
\begin{align}\label{est:Hminus1:final}
  \frac{d}{dt}\Sob{\tht}{\Hdot^{-1}_{\om}\cap L^2_{\om}}^2&+\frac{3}2\Sob{m(D)^{\frac{1}2}\tht}{\Hdot^{-1}_{\om}\cap L^2_{\om}}^{2}\notag\\
  &\quad\leq C\left(1+\Sob{q}{ \Hdot^{-1}_\om\cap \Hdot^1_\om}^{\frac{1}{1-\gam}}+\Sob{m_1(D)^{\frac{\gam}2}q}{H^1_\om}^{\frac{2}{2-\gam}}\right)\Sob{\tht}{\Hdot^{-1}_{\om}\cap L^2_{\om}}^2+C\Sob{m_1(D)^{-\frac{1}{2}}G}{\Hdot^{-1}_\om\cap L^2_{\om}}^2.
\end{align}
An application of Gronwall's inequality now yields
    \begin{align}\label{est:Hminus1:gronwall}
        \sup_{0\leq t\leq T}\Bigg{(}\Sob{\tht(t)}{\Hdot^{-1}_\om\cap L^2_{\om}}^2&+\int_0^t\Sob{m(D)^{\frac{1}{2}}\tht(s)}{\Hdot^{-1}_\om\cap L^2_{\om}}^{2}ds\Bigg{)}\notag\\
        &\quad\leq {C}_T(0^-)\left(\Sob{\tht_0}{\Hdot^{-1}_\om\cap L^2_{\om}}^2+\int_0^T\Sob{m_1(D)^{-\frac{1}2}G(s)}{\Hdot^{-1}_\om\cap L^2_{\om}}^2ds\right),
    \end{align}
where
\begin{align}\label{def:Hminus1:constant}
    {C}_T(0^-)=\exp\left(C\int_0^T\left(1+\Sob{q(s)}{ \Hdot^{-1}_\om\cap \Hdot^1_\om}^{\frac{1}{1-\gam}}+\Sob{m_1(D)^{\frac{\gam}2}q(s)}{H^1_\om}^{\frac{2}{2-\gam}}\right)ds\right),
\end{align}
for some constant $C$.

\subsection{Global existence and uniqueness for the protean system}\label{subsect:apriori:summary} 

From the apriori estimates developed previously and a standard artificial viscosity approximation, we obtain the following theorem for the well-posedness of \eqref{eq:mod:claw}. Let us denote by
\begin{align}\label{def:p0}
    p_0=\frac{2}{2-\gam}.
\end{align}

\begin{Thm}\label{thm:modclaw:wellposed}
Let $\be\in [0,2]$ and  $\s\geq-1$ satisfy \eqref{sigma:range}. 
Let $\om \in \mathscr{M}_W$, $p\in\mathscr{M}_C$, $m\in\mathscr{M}_D$ be given such that \eqref{cond:summary} holds for some $\gam\in (0,1)$. Given $T>0$, suppose that
    \begin{align}\label{cond:q:G}
        q\in L^{\infty}(0,T;H^{1+\be}_\om),\quad m_1(D)^{\frac{\gam}{2}}q\in L^{p_0}(0,T; \Hdot^{1+\be}_\om),\quad m_1(D)^{-\frac{\gam}2}G \in L^{2}(0,T;\Hdot^{\s}_\om\cap L^2_{\om}),
    \end{align}
Then for any $\tht_0\in \Hdot^{\s}_\om\cap L^2_{\om}$, there exists a unique solution, $\tht$, of \eqref{eq:mod:claw} satisfying \eqref{est:Hs:gronwall} and 
    \[
      \theta\in C([0,T];\Hdot^{\s}_\om\cap L^2_{\om}),\quad m_1(D)^{\frac{1}2}\tht\in L^2(0,T;\Hdot^{\s}_{\om}\cap L^2_{\om}).
    \]
Moreover, given any $\nu(D)\in\mathscr{M}_S(m)$, the unique solution satisfies \eqref{est:Gev:final} provided that
    \[
\int_0^T\left(\Sob{E^{\lam_1t}_\nu q(t)}{H^{1+\be}_\om}^{\frac{1}{1-\gam}}+\Sob{m_1(D)^{\frac{\gam}2}E^{\lam_1t}_\nu{q}(t)}{\Hdot^{1+\be}_\om}^{\frac{2}{2-\gam}}\right)dt<\infty,
    \]
where $E^\lam_\nu$ is defined by \eqref{def:Elamnu}.
\end{Thm}

The proof of \cref{thm:modclaw:wellposed} is provided in \cref{app:well:posed}.

\subsection{Stability of the protean system}\label{sect:stab:apriori} In this section, we will establish continuity properties of \eqref{eq:mod:claw} with respect to its datum. This will rely on having access to suitable stability-type estimates. Ultimately, the results of this section are part of the development of the uniqueness of solutions and continuity of the solution map of \eqref{eq:dgsqg}. 

The standing assumption of this section will be that $p\in\mathscr{M}_C$, $\om \in \mathscr{M}_W$, $m\in\mathscr{M}_D$, $\nu\in \mathscr{M}_R(m)$ are given and satisfy \eqref{cond:summary}. The main result of this section is then stated in the following theorem.

\begin{Thm}\label{thm:stab}
Let $\be \in [0,2]$ and $T>0$ be fixed. Suppose that sequences $\{\tht^n\}_n, \{q^n\}_n, \{G^n\}_n$ and functions $\tht^\infty, q^\infty, G^\infty$ are given such that
    \begin{align}\label{cond:Q0}
        &Q_\be:=\notag\\
        &\sup_{n\in\NN\cup\{\infty\}}
            \begin{cases}
                \Sob{\tht^n_0}{L^2_{\om}\cap\Hdot^{-1}_\om}+\Sob{q^{n}}{L^\infty _T(H^1_\om\cap\Hdot^{-1}_\om)}+\Sob{m_1(D)^{\frac{\gam}{2}}q^{n}}{L^{p_0}_{T}\Hdot^{1}_{\om}}+\Sob{m_1(D)^{-\frac{1}{2}}G^n}{L^2_T(L^2_{\om}\cap\Hdot^{-1}_\om)},&\be=0\\
                \Sob{\tht^n_0}{H^{\be}_\om}+\Sob{q^{n}}{L^\infty _T H^{1+\be}_\om}+\Sob{m_1(D)^{\frac{\gam}{2}}q^{n}}{L^{p_0}_{T}\Hdot^{1+\be}_{\om}}+\Sob{m_1(D)^{-\frac{1}{2}}G^n}{L^2_TH^{\be}_\om},&\be\in(0,1],\\
                \Sob{\tht^n_0}{H^{\be}_\om}+\Sob{q^{n}}{L^\infty_T H^{1+\be}_\om}+\Sob{m_1(D)^{\frac{\gam}{2}}q^{n}}{L^{p_0}_{T}\Hdot^{1+\be}_{\om}}+\Sob{m_1(D)^{-\frac{1}{2}}G^n}{L^2_TH^{\be}_\om},&\be\in(1,2].
            \end{cases}
    \end{align}
is finite. Moreover, suppose that
    \begin{align}\label{cond:Z}
        &Z_\be^{n}:=\notag\\
           & \begin{cases}
               \Sob{\tht^n_0-\tht^\infty_0}{L^2_{\om}\cap\Hdot^{-1}_\om}^2+\Sob{q^n-q^\infty}{L^\infty_T (L^2_{\om}\cap\Hdot^{-1}_\om)}^2+\Sob{m_1(D)^{\frac{\gam}{2}}(q^{n}-q^\infty)}{L^{2}_{T}L^2_{\om}}^2\\
               \hspace{22 em}+\Sob{m_1(D)^{-\frac{1}{2}}(G^n-G^\infty)}{L^2_T(L^2_{\om}\cap\Hdot^{-1}_\om)}^2,&\be=0\\
                \Sob{\tht_0^n-\tht_0^\infty}{L^2_{\om}}^2+\Sob{q^n-q^\infty}{L^\infty_T L^2_{\om}}^2+\Sob{m_1(D)^{\frac{\gam}{2}}(q^{n}-q^\infty)}{L^{2}_{T}L^2_{\om}}^2+\Sob{m_1(D)^{-\frac{1}{2}}(G^n-G^\infty)}{L^2_TL^2_{\om}}^2,&\be\in(0,1),\\
                \Sob{\tht^n_0-\tht^\infty_0}{H^\be_\om}^2+\Sob{q^n-q^\infty}{L^\infty_TH^{\be}_\om}^2+\Sob{m_1(D)^{\frac{\gam}{2}}(q^{n}-q^\infty)}{L^{2}_{T}H^{\be}_{\om}}^2+\Sob{m_1(D)^{-\frac{1}{2}}(G^n-G^\infty)}{L^2_TH^\be_\om}^2,&\be\in[1,2]
            \end{cases}
    \end{align}
converges to $0$ as $n\goesto\infty$. For all $n\in\NN\cup\{\infty\}$, let $\tht^n$ denote the unique solution guaranteed by \cref{thm:modclaw:wellposed} of the initial value problem
    \begin{align}\label{eq:modclaw:n}
        \partial_{t}\theta^{n}+ m(D)\tht^{n}+\Div F_{q^{n}}(\tht^{n})=G^n,\quad\tht^{n}(0,x)=\tht^n_0(x),
    \end{align}
Then
    \begin{align}\label{eq:stab:a}
        \lim_{n\goesto\infty}\left(\Sob{\tht^n-\tht^\infty}{L^\infty_T(L^2_{\om}\cap\Hdot^{-1}_\om)}^2+\Sob{m(D)^{\frac{1}{2}}(\tht^n-\tht^\infty)}{L^2_T(L^2_{\om}\cap\Hdot^{-1}_\om)}^2\right)=0,
    \end{align}
when $\be=0$, and
    \begin{align}\label{eq:stab:b}
        \lim_{n\goesto\infty}\left(\Sob{\tht^n-\tht^\infty}{L^\infty_TH^{\be}_\om}^2+\Sob{m(D)^{\frac{1}{2}}(\tht^n-\tht^\infty)}{L^2_TH^{\be}_\om}^2\right)=0,
    \end{align}   
when $\be\in(0,2]$.
\end{Thm}

The proof of \cref{thm:stab} will rely on the following stability-type estimates. 

\begin{Prop}\label{prop:stab}
Let $\be \in [0,2]$ and $T>0$.  Suppose that $\{\tht^n\}_n, \{q^n\}_n, \{G^n\}_n$ and $\tht^\infty, q^\infty, G^\infty$ be such that
    \begin{align}\label{cond:Qbeta}
        &Q_\be^*:=\\
        &\sup_{n\in\NN\cup\{\infty\}}
            \begin{cases}
                \Sob{\tht^n_0}{H^1_\om\cap\Hdot^{-1}_\om}+\Sob{q^{n}}{L^\infty _T(H^1_\om\cap\Hdot^{-1}_\om)}+\Sob{m_1(D)^{\frac{\gam}{2}}q^{n}}{L^{p_0}_{T}\Hdot^{1}_{\om}}+\Sob{m_1(D)^{-\frac{1}{2}}G^n}{L^2_T(H^1_\om\cap\Hdot^{-1}_\om)},&\be=0\\
                \Sob{\tht^n_0}{H^{1+\be}_\om}+\Sob{q^{n}}{L^\infty _T H^{1+\be}_\om}+\Sob{m_1(D)^{\frac{\gam}{2}}q^{n}}{L^{p_0}_{T}\Hdot^{1+\be}_{\om}}+\Sob{m_1(D)^{-\frac{1}{2}}G^n}{L^2_T\Hdot^{1+\be}_\om},&\be\in(0,1],\\
                \Sob{\tht^n_0}{H^{1+\be}_\om}+\Sob{q^{n}}{L^\infty_T H^{1+\be}_\om}+\Sob{m_1(D)^{\frac{\gam}{2}}q^{n}}{L^{p_0}_{T}\Hdot^{1+\be}_{\om}}+\Sob{m_1(D)^{-\frac{1}{2}}G^n}{L^2_T\Hdot^{1+\be}_\om},&\be\in(1,2].
            \end{cases}\notag
    \end{align}
is finite. For all $n\in\NN\cup\{\infty\}$, let $\tht^n$ denote the  unique solution of \eqref{eq:modclaw:n}, corresponding to data ($\tht^n_0,q^n,G^n)$, guaranteed by \cref{thm:modclaw:wellposed}.

{
\noindent When $\be=0$:
    \begin{align}\label{eq:stab:estimate:a}
        &\Sob{\tht^n-\tht^\infty}{L^\infty_T(L^2_{\om}\cap\Hdot^{-1}_\om)}^2+\Sob{m(D)^{\frac{1}{2}}(\tht^n-\tht^\infty)}{L^2_T(L^2_{\om}\cap\Hdot^{-1}_\om)}^2\leq CZ^n_0
    \end{align}
When $\be\in(0,1)$:
    \begin{align}\label{eq:stab:estimate:b1}
            &\Sob{\tht^n-\tht^\infty}{L^\infty_TL^2_{\om}}^2+\Sob{m(D)^{\frac{1}{2}}(\tht^n-\tht^\infty)}{L^2_TL^2_{\om}}^2\leq CZ^n_\be\\
             &\Sob{\tht^n-\tht^\infty}{L^\infty_TH^\be_\om}^2+\Sob{m(D)^{\frac{1}{2}}(\tht^n-\tht^\infty)}{L^2_TH^\be_\om}^2\leq C(Z^n_\be)^{\frac{1}{1+\be}}.\label{eq:stab:estimate:b2}
    \end{align}
When $\be\in[1,2]$:
    \begin{align}\label{eq:stab:estimate:c}
        &\Sob{\tht^n-\tht^\infty}{L^\infty_TH^{\be}_\om}^2+\Sob{m(D)^{\frac{1}{2}}(\tht^n-\tht^\infty)}{L^2_TH^{\be}_\om}^2\leq  C Z^n_\be,
\end{align}
for all $n\in\NN\cup\{\infty\}$, for some constant $C>0$ depending on $Q_\be^*$.
}
\end{Prop}

Note that with these stability-type estimates in hand, we may argue by density to establish a continuity property for the system in its datum with respect to the \textit{weaker} topology of $H^\be_\om$, which is a crucial difference between \cref{thm:stab} and \cref{prop:stab}. Indeed, let us first prove \cref{thm:stab} assuming \cref{prop:stab}. We will then provide the proof of \cref{prop:stab} after.

\begin{proof}[Proof of \cref{thm:stab} (assuming \cref{prop:stab})] 
Suppose that $\tht_0^n\in H^{\be}_\om$ and $m_1(D)^{-\frac{1}{2}}G^n\in L^2(0,T;H^{\be}_\om)$, for all $n \in \mathbb{N}\cup\{\infty\}$. For each $k\in \mathbb{N}$, denote by $\tht^{n}_{k}$ the unique solution to
    \begin{align} \label{eq:modclaw:k}
        \partial_{t}\theta^{n}_{k}+ m(D)\tht^n_k+\Div F_{q^{n}}(\tht^{n}_{k})=S_{k}G^n ,\quad\tht^{n}_{k}(0,x)=S_{k}\tht_0^n(x),
    \end{align}
where $S_{k}$ denotes the Littlewood-Paley projection onto frequencies $|\xi|\leq 2^{k}$ defined in \cref{sect:lp:decomposition}. Let
    \[
        \Tht^n_k=\tht^n-\tht^n_k,\quad \text{for all}\ n\in\NN\cup\{\infty\}.
    \]
From \eqref{eq:modclaw:n} and \eqref{eq:modclaw:k}, it follows that
    \[
    \partial_{t}\Tht^n_k+m(D)\Theta^n_k+ \Div F_{q^{n}}(\Tht^n_k)=(I-S_{k})G^n,\quad \text{for all}\ n\in\NN\cup\{\infty\}.
    \]
Then by \cref{thm:modclaw:wellposed}, we obtain
    \begin{align}\notag
        \sup_{n\in\NN\cup\{\infty\}}&\left(\Sob{\Tht^n_k}{L^{\infty}_{T}(L^2_\om\cap\Hdot^{-1}_\om)}+\Sob{m(D)^{\frac{1}{2}}\Tht^n_k}{L^{2}_{T}(L^2_\om\cap\Hdot^{-1}_\om)}\right)\notag\\
        &\le C_{T}\left(\Sob{(I-S_{k})\tht_{0}^n}{L^2_\om\cap\Hdot^{-1}_\om}+\Sob{m_1(D)^{-\frac{1}{2}}(I-S_{k})G^n}{L^2_{T}(L^2_\om\cap\Hdot^{-1}_\om)}\right).\notag
    \end{align}
when $\be=0$, and
    \begin{align}\notag
        \sup_{n\in\NN\cup\{\infty\}}\left(\Sob{\Tht^n_k}{L^{\infty}_{T}H_\om^{\beta}}+\Sob{m(D)^{\frac{1}{2}}\Tht^n_k}{L^{2}_{T}H_\om^{\beta}}\right) \le C\left(\Sob{(I-S_{k})\tht_{0}^n}{H_\om^{\beta}}+\Sob{m_1(D)^{-\frac{1}{2}}(I-S_{k})G^n}{L^2_{T}H_\om^{\be}}\right),
    \end{align}
when $\be\in(0,2]$. Let $\delta>0$. Because of \eqref{cond:Q0}, we may choose $k_0>0$, independently of $n$, such that
    \begin{align}\label{est:modclaw:intermediate:high:a}
        \sup_{n\in\NN\cup\{\infty\}}\left(\Sob{\Tht^n_{k_0}}{L^{\infty}_{T}(L^2_\om\cap\Hdot^{-1}_\om)}+\Sob{m(D)^{\frac{1}{2}}\Tht^n_{k_0}}{L^{2}_{T}(L^2_\om\cap\Hdot^{-1}_\om)}\right)\le \delta/3.
    \end{align}
when $\be=0$, and
    \begin{align}\label{est:modclaw:intermediate:high:b}
        \sup_{n\in\NN\cup\{\infty\}}\left(\Sob{\Tht^n_{k_0}}{L^{\infty}_{T}H_\om^{\beta}}+\Sob{m(D)^{\frac{1}{2}}\Tht^n_{k_0}}{L^{2}_{T}H_\om^{\beta}}\right)\le \delta/3.
    \end{align}
when $\be\in(0,2]$.

Now observe that $S_{k_0}\tht_{0}^n\in H_\om^{1+\be}$ and $ m_1(D)^{-\frac{1}{2}}S_{k_0}G^n \in L^{2}(0,T;H_\om^{1+\be})$. We may thus apply \cref{thm:modclaw:wellposed} to obtain a sequence $\{\tht_{k_0}^n\}$ of solutions to \eqref{eq:modclaw:n} corresponding to data $S_{k_0}\tht_0^n$ and $S_{k_0}G^n$. Since $Q_\be<\infty$, we may now invoke \cref{prop:stab} in conjunction with \eqref{cond:Z} to find an integer $N>0$ such that
    \begin{align}\label{eq:approx:converge:a}
        \sup_{n\geq N}\left(\Sob{\tht^n_{k_0}-\tht^\infty_{k_0}}{L^{\infty}_{T}(L^2_\om\cap\Hdot^{-1}_\om)}+\Sob{m(D)^{\frac{1}{2}}(\tht^n_{k_0}-\tht^\infty_{k_0})}{L^{2}_{T}(L^2_\om\cap\Hdot^{-1}_\om)}
        \right)\le \delta/3,
    \end{align}
when $\be=0$, and 
    \begin{align}\label{eq:approx:converge:b}
        \sup_{n\geq N}\left(\Sob{\tht^n_{k_0}-\tht^\infty_{k_0}}{L^{\infty}_{T}H_\om^{\beta}}+\Sob{m(D)^{\frac{1}{2}}(\tht^n_{k_0}-\tht^\infty_{k_0})}{L^{2}_{T}H_\om^{\beta}}
        \right)\le \delta/3,
    \end{align}
when $\be\in(0,2]$. Finally, we see that
    \[
        \Tht^n:=\tht^n-\tht^\infty=\Tht^n_{k_0}+(\tht^n_{k_0}-\tht_{k_0}^\infty)-\Tht^\infty_{k_0}.
    \]
Therefore, by the triangle inequality, combined with \eqref{est:modclaw:intermediate:high:a}, \eqref{eq:approx:converge:a}, we obtain
    \begin{align}\notag
         &\Sob{\Theta^{n}}{L^{\infty}_{T}(L^2_\om\cap\Hdot^{-1}_\om)\cap L^2_T(L^2_{\om m^{1/2}}\cap\Hdot^{-1}_{\om m^{1/2}})}
            \notag\\&\leq  \Sob{\Tht^n_{k_0}}{L^\infty_T(L^2_\om\cap\Hdot^{-1}_\om)\cap L^2_T(L^2_{\om m^{1/2}}\cap\Hdot^{-1}_{\om m^{1/2}})}+\Sob{\tht^n_{k_0}-\tht^\infty_{k_0}}{L^{\infty}_{T}(L^2_\om\cap\Hdot^{-1}_\om)\cap L^2_T(L^2_{\om m^{1/2}}\cap\Hdot^{-1}_{\om m^{1/2}})}\notag\\&+\Sob{\Tht^\infty_{k_0}}{L^{\infty}_{T}(L^2_\om\cap\Hdot^{-1}_\om)\cap L^2_T(L^2_{\om m^{1/2}}\cap\Hdot^{-1}_{\om m^{1/2}})}\leq\de,\notag
    \end{align}
for all $n\geq N$, when $\be=0$. Similarly, when $\be\in(0,2]$, we apply \eqref{est:modclaw:intermediate:high:b}, \eqref{eq:approx:converge:b} to deduce
    \begin{align}
         \sup_{n\geq N}\Sob{\Theta^{n}}{L^{\infty}_{T}H_\om^{\beta}\cap L^2_TH^\be_{\om m^{1/2}}}&\leq\de,\notag
    \end{align}
Since $\delta$ was arbitrary, this establishes the claims \eqref{eq:stab:a} and \eqref{eq:stab:b}. 
\end{proof}

Let us now bring our attention to proving \cref{prop:stab}. In order to do so, we require certain bounds for the divergence of the flux in \eqref{eq:mod:claw}. This is stated in the following lemma.

\begin{Lem}\label{lem:stab:G}
 Let $\beta \in [0,2]$ and $F_q(\tht)$ be defined as in \eqref{def:mod:flux}. Let $\gam \in (0,1)$. Then we have the following estimates depending on $\be$: \\When $\be=0$:
            \begin{align}
                \Sob{m_1(D)^{-\frac{\gam}2}\Div F_{q}(\tht)}{\Hdot^{-1}_\om}&\le C\nrm{q}_{\Hdot^{-1}_\om\cap L^2_{\om}}\Sob{m_1(D)^{\frac{\gam}{2}} \tht}{L^2_{\om}}+C\nrm{\tht}_{L^2_{\om}}\nrm{m_1(D)^{\frac{\gam}2}q}_{L^2_{\om}},\label{est:div:a1}\\
                \Sob{m_1(D)^{-\frac{\gam}2}\Div F_{q}(\tht)}{L^2_{\om}}&\le C\nrm{q}_{\Hdot^
   {-1}_\om\cap L^2_{\om}}\nrm{m_1(D)^{\frac{\gam}{2}}\tht}_{\Hdot^{1}_\om}+C\nrm{\tht}_{\Hdot^{1}_\om}\nrm{m_1(D)^{\frac{\gam}{2}}q}_{L^2_{\om}}. \label{est:div:a2}
            \end{align}
        When $\be\in(0,1)$:
            \begin{align}
                 \Sob{m_1(D)^{-\frac{\gam}2}\Div F_{q}(\tht)}{\Hdot^\be_\om}&\le C\nrm{q}_{\Hdot^{-1}_\om\cap \Hdot^\be_\om}\nrm{m_1(D)^{\frac{\gam}2}\tht}_{\Hdot^{1+\be}_\om}+C\nrm{\tht}_{\Hdot^{1+\be}_\om}\nrm{m_1(D)^{\frac{\gam}2}q}_{\Hdot^{\be}_\om},\label{est:div:b1}\\
               \Sob{m_1(D)^{-\frac{\gam}2}\Div F_q(\tht)}{L^2_{\om}}&\leq C\nrm{q}_{L^2_{\om}}\nrm{m_1(D)^{\frac{\gam}2}\tht}_{\Hdot^{1+\be}_\om}+C\nrm{\tht}_{\Hdot^{1+\be}_\om}\nrm{m_1(D)^{\frac{\gam}2}q}_{L^2_{\om}}\label{est:div:b2}.
            \end{align}
        When $\be\in[1,2]$:
            \begin{align}
            \Sob{m_1(D)^{-\frac{\gam}2}\Div F_{q}(\tht)}{H^\be_\om}&\le C\nrm{q}_{ H^\be_\om}\nrm{m_1(D)^{\frac{\gam}2}\tht}_{H^{1+\be}_\om}+C\nrm{\tht}_{H^{1+\be}_\om}\nrm{m_1(D)^{\frac{\gam}2}q}_{H^{\be}_\om}\label{est:div:c}.
            \end{align}
\end{Lem} 

Let us assume \cref{lem:stab:G} and prove \cref{prop:stab}. We will then conclude the section by proving \cref{lem:stab:G}.

\begin{proof}[Proof of \cref{prop:stab} (assuming \cref{lem:stab:G})]
Suppose that $\tht_{0}^n \in H^{1+\be}_\om$ and $m_1(D)^{-\frac{1}{2}}G^n\in L^2(0,T;H^{1+\be}_\om)$. By \cref{thm:modclaw:wellposed}, for each $n\in\NN\cup\{\infty\}$, we may let $\tht^n$ denote the unique solution of \eqref{eq:modclaw:n} that belongs to $C([0,T];H^{1+\be}_\om)$. Moreover, by \eqref{est:Hminus1:gronwall} and \eqref{cond:Qbeta} when $\be=0$, it follows that
    \begin{align}\label{est:tht:Hminus1:bdd}
        \sup_{0<n\leq\infty}\Sob{\tht^n}{L^\infty_T(H^{1}_\om\cap\Hdot^{-1}_\om)},\quad \sup_{0<n\leq\infty}\Sob{m(D)^{\frac{1}{2}}\tht^n}{L^2_T(H^{1}_\om\cap\Hdot^{-1}_\om)}<\infty.
    \end{align}
Similarly, \eqref{est:Hs:gronwall} and \eqref{cond:Q0} imply 
    \begin{align}\label{est:tht:bdd}
        \sup_{0<n\leq\infty}\Sob{\tht^n}{L^\infty_T H^{1+\be}_\om},\quad \sup_{0<n\leq\infty}\Sob{m(D)^{\frac{1}{2}}\tht^n}{L^2_T H^{1+\be}_\om}<\infty,
    \end{align}
when $\be\in(0,2]$. For each $\NN\cup\{\infty\}$, let     
    \[
    \Tht^n:=\tht^{n}-\tht^{\infty},\quad z^n:=q^n-q^\infty,\quad W^n:=G^n-G^\infty.
    \]
Observe that $\Tht^\infty=z^\infty=H^\infty=0$. 

Now observe that for each $n\in\NN\cup\{\infty\}$, the pair $(\Tht^n,q^n)$ satisfies
    \begin{align}\label{eq:modclaw:intermediate}
        \partial_{t}\Tht^n+\Div F_{q^{n}}(\Tht^n)=-m(D)\Tht^n-\Div F_{z^n}(\tht^{\infty})+W^n,\quad \Tht^n(0,x)=\Tht^n_0(x).
    \end{align} 
Note that \eqref{eq:modclaw:intermediate} possesses the same structure as \eqref{eq:mod:claw} with $\tht\mapsto\Tht^n$, $q\mapsto q^n$, and $G\mapsto -\Div F_{z^n}(\tht^\infty)+W^n$. Thus, when $\be=0$, it follows from \eqref{est:Hminus1:gronwall} that
    \begin{align}\label{eq:convergence:a}
         &\Sob{\Tht^n}{L^\infty_T(L^2_{\om}\cap\Hdot^{-1}_{\om})}^2+\Sob{m(D)^{\frac{1}{2}}\Tht^n}{L^2_T(L^2_{\om}\cap\Hdot^{-1}_{\om})}^{2}\\
            &\leq C^{(n)}_T(0^-)\left(\Sob{\Tht^n_0}{L^2_{\om}\cap\Hdot^{-1}_\om}^2+\Sob{m_1(D)^{-\frac{1}{2}}\Div F_{z^n}(\tht^\infty)}{L^2_T(L^2_{\om}\cap\Hdot^{-1}_\om)}^2+\Sob{m_1(D)^{-\frac{1}{2}}W^n}{L^2_T(L^2_{\om}\cap\Hdot^{-1}_\om)}^2\right),\notag
    \end{align}
where $C_T^{(n)}(0^-)$ is defined in \eqref{def:Hminus1:constant} corresponding to $q\mapsto q^n$. On the other hand, for any $\be \in[0,2]$ and $\s\ge 0$ satisfying \eqref{sigma:range}, it follows from \eqref{est:Hs:gronwall} that
    \begin{align}\label{eq:convergence:b}
        \Sob{\Tht^n}{L^\infty_TH^{\s}_\om}^2+&\Sob{m(D)^{\frac{1}{2}}\Tht^n}{L^2_TH^{\s}_\om}^2\notag\\
            &\leq C_T^{(n)}(\be)\left(\Sob{\Tht_0^n}{H^\s_\om}^2+\Sob{m_1(D)^{-\frac{1}{2}}\Div F_{z^n}(\tht^\infty)}{L^2_TH^\s_\om}^2+\Sob{m_1(D)^{-\frac{1}{2}}W^n}{L^2_TH^\s_\om}^2\right),
    \end{align}
where $C_T^{(n)}(\be)$ is defined in \eqref{def:Hs:constant} corresponding to $q\mapsto q^n$. By the uniform bounds in \eqref{cond:Q0}, we have 
    \begin{align}\label{est:constant}
        \sup_{n\in\NN\cup\{\infty\}}C_T^{(n)}(\be)<\infty.
    \end{align}
We are thus left to estimate $m_1(D)^{-\frac{1}2}\Div F_{z^n}(\tht^\infty)$.

When $\be=0$, \eqref{est:div:a1}, \eqref{est:div:a2} in \cref{lem:stab:G}, and \eqref{est:tht:Hminus1:bdd} imply
    \begin{align}\label{est:div:convergence:a}
            \Sob{m_1(D)^{-\frac{\gam}2}\Div F_{z^n}(\tht^\infty)}{L^2_T (L^2_{\om}\cap \Hdot^{-1}_\om)}&\le C\nrm{z^n}_{L^\infty_T(\Hdot^
   {-1}_\om\cap L^2_{\om})}\nrm{m_1(D)^{\frac{\gam}{2}}\tht^\infty}_{L^2 _T H^{1}_\om}+C\nrm{\tht^\infty}_{L^\infty_T H^{1}_\om}\nrm{m_1(D)^{\frac{\gam}{2}}z^n}_{L^2_T L^2_{\om}}.
    \end{align}
When $\be\in(0,1)$, \eqref{est:div:b2} in \cref{lem:stab:G}, and \eqref{est:tht:bdd} imply
    \begin{align}\label{est:div:convergence:b}
            \Sob{m_1(D)^{-\frac{\gam}2}\Div F_{z^n}(\tht^\infty)}{L^2_T L^2_{\om}}&\leq C\nrm{z^n}_{L^\infty_T L^2_{\om}}\nrm{m_1(D)^{\frac{\gam}2}\tht^\infty}_{L^2_T \Hdot^{1+\be}_\om}+C\nrm{\tht^\infty}_{L^\infty_T H^{1+\be}_\om}\nrm{m_1(D)^{\frac{\gam}2}z^n}_{L^2_T L^2_{\om}}.
    \end{align}
Lastly, when $\be\in[1,2]$, \eqref{est:div:c} in \cref{lem:stab:G}, and \eqref{est:tht:bdd} imply
    \begin{align}\label{est:div:convergence:c}
        \Sob{m_1(D)^{-\frac{\gam}2}\Div F_{z^n}(\tht^\infty)}{L^2_T H^\be_\om}&\le C\nrm{z^n}_{L^\infty_T H^\be_\om}\nrm{m_1(D)^{\frac{\gam}2}\tht^\infty}_{L^2 H^{1+\be}_\om}+C\nrm{\tht^\infty}_{L^\infty_T H^{1+\be}_\om}\nrm{m_1(D)^{\frac{\gam}2}z^n}_{L^2_T H^{\be}_\om}.
    \end{align}

Now we apply \eqref{est:div:convergence:a} in \eqref{eq:convergence:a} when $\be=0$ to obtain
    \begin{align}\label{est:Tht:converge:a}
      &\Sob{\Tht^n}{L^\infty_T(L^2_{\om}\cap\Hdot^{-1}_\om)}^2+\Sob{m(D)^{\frac{1}{2}}\Tht^n}{L^2_T(L^2_{\om}\cap\Hdot^{-1}_\om)}^2\notag \\
      &\leq C\left(\Sob{\Tht^n_0}{L^2_{\om}\cap\Hdot^{-1}_\om}^2+\nrm{z^n}_{L^\infty_T(\Hdot^
   {-1}_\om\cap L^2_{\om})}^2\nrm{m_1(D)^{\frac{\gam}{2}}\tht^\infty}_{L^2 _T H^{1}_\om}^2\right. \notag \\&\left.\hspace{3 em}+\nrm{\tht^\infty}_{L^\infty_T H^{1}_\om}^2\nrm{m_1(D)^{\frac{\gam}{2}}z^n}_{L^2_T L^2_{\om}}^2+\Sob{m_1(D)^{-\frac{1}{2}}W^n}{L^2_T(L^2_{\om}\cap\Hdot^{-1}_\om)}^2\right).
    \end{align}
Similarly, when $\be\in(0,1)$, 
we apply \eqref{est:div:convergence:b} in \eqref{eq:convergence:b} with $\s=0$ to deduce
    \begin{align}\label{est:Tht:converge:b2}
          &\Sob{\Tht^n}{L^\infty_T L^2_{\om}}^2+\Sob{m(D)^{\frac{1}{2}}\Tht^n}{L^2_T L^2_{\om}}^2\notag \\
      &\leq C\left(\Sob{\Tht^n_0}{L^2_{\om}}^2+\nrm{z^n}_{L^\infty_T L^2_{\om}}^2\nrm{m_1(D)^{\frac{\gam}2}\tht^\infty}_{L^2_T \Hdot^{1+\be}_\om}^2\right. \notag \\&\left.\hspace{3 em}+\nrm{\tht^\infty}_{L^\infty_T H^{1+\be}_\om}^2\nrm{m_1(D)^{\frac{\gam}2}z^n}_{L^2_T L^2_{\om}}^2+\Sob{m_1(D)^{-\frac{1}{2}}W^n}{L^2_T L^2_{\om} }^2\right).
    \end{align}
Moreover, we may upgrade this bound using interpolation \eqref{est:interpolation} and by applying the uniform bounds \eqref{est:tht:bdd} to obtain
     \begin{align}\label{est:Tht:converge:b1}
        \Sob{\Tht^n}{L^\infty_TH^\be_\om\cap L^2_TH^\be_{\om m^{1/2}}}
            &\leq C\Sob{\Tht^n}{L^\infty_T H^{1+\be}_\om\cap L^2_T H^{1+\be}_{\om m^{1/2}}}^{\frac{\be}{1+\be}}\Sob{\Tht^n}{L^\infty_T L^2_{\om}\cap L^2_TL^2_{\om m^{1/2}}}^{\frac{1}{1+\be}}\notag\\
            &\leq C\left(\sup_{n\in\NN\cup\{\infty\}}\Sob{\tht^n}{L^\infty_T H^{1+\be}_\om\cap L^2_T H^{1+\be}_{\om m^{1/2}}}\right)^{\frac{\be}{1+\be}}\Sob{\Tht^n}{L^\infty_TL^2_{\om}\cap L^2_TL^2_{\om m^{1/2}}}^{\frac{1}{1+\be}}\notag\\
            &\leq C\left(\Sob{\Tht^n_0}{L^2_{\om}}^2+\nrm{z^n}_{L^\infty_T L^2_{\om}}^2\nrm{m_1(D)^{\frac{\gam}2}\tht^\infty}_{L^2_T \Hdot^{1+\be}_\om}^2\right. \notag \\&\left.\hspace{3 em}+\nrm{\tht^\infty}_{L^\infty_T H^{1+\be}_\om}^2\nrm{m_1(D)^{\frac{\gam}2}z^n}_{L^2_T L^2_{\om}}^2+\Sob{m_1(D)^{-\frac{1}{2}}W^n}{L^2_T L^2_{\om} }^2\right)^{\frac{1}{1+\be}}.
    \end{align}    
Lastly, when $\be\in[1,2]$, we apply \eqref{est:div:convergence:c} in \eqref{eq:convergence:b} with $\s=\be$, to obtain
    \begin{align}\label{est:Tht:converge:c2}
        &\Sob{\Tht^n}{L^\infty_TH^{\be}_\om}^2+\Sob{m(D)^{\frac{1}{2}}\Tht^n}{L^2_TH^{\be}_\om}^2\notag \\
      &\leq C\left(\Sob{\Tht^n_0}{H^\be_\om}^2+\nrm{z^n}_{L^\infty_T H^\be_\om}^2\nrm{m_1(D)^{\frac{\gam}2}\tht^\infty}_{L^2_T H^{1+\be}_\om}^2\right. \notag \\&\left.\hspace{3 em}+\nrm{\tht^\infty}_{L^\infty_T H^{1+\be}_\om}^2\nrm{m_1(D)^{\frac{\gam}2}z^n}_{L^2_T H^\be_\om}^2+\Sob{m_1(D)^{-\frac{1}{2}}W^n}{L^2_T H^\be_\om }^2\right).
    \end{align}
Gathering the estimates \eqref{est:Tht:converge:a}, \eqref{est:Tht:converge:b1}, \eqref{est:Tht:converge:c2} completes the proof.
\end{proof}

Finally, let us prove \cref{lem:stab:G}

\begin{proof}[Proof of \cref{lem:stab:G}]
We split the proof into the cases $\be=0$, $\be\in(0,1)$, $\be=1$, $\be\in(1,2]$. Throughout, recall that $v$ is given by \eqref{def:v}. We will repeatedly apply \cref{cor:prod} and \cref{lem:commutator4}, which are applicable due to \eqref{cond:summary}.

\subsubsection*{Case $\be =0$} Then $\Div F_q(\tht)=\nabla\cdotp (v\tht)=v\cdotp\nabla\tht$ and we have
    \begin{align}\label{eq:stab:setup}
    \Sob{m_1(D)^{-\frac{\gam}{2}}\Div F_q(\tht)}{\Hdot^{-1}_{\om }}\leq \Sob{v\tht}{L^2_{\om m_1^{-\gam/2}}},\quad \Sob{m_1(D)^{-\frac{\gam}{2}}\Div F_q(\tht)}{L^2_{\om}}\leq \Sob{v\cdotp\nabla\tht}{L^2_{\om m_1^{-\gam/2}}}.
    \end{align}
We apply \cref{cor:prod} to the first expression in \eqref{eq:stab:setup} with $(s,\bar{s})=(1,0)$,  $(\om_1,\til{\om}_1)=(p^{-1}\om,m_1^{\frac{\gam}{2}}\om)$, $(\om_2,\til{\om}_2)=(\om, p^{-1}m_1^{\frac{\gam}{2}}\om)$, $(\om_3, \til{\om}_3)=(p^{-1}\om,m_1^{\frac{\gam}{2}}\om)$, and $\Gamma=m_1^{\frac{\gam}{2}}$, we obtain
\begin{align}\label{est:stab:0:a}
   \Sob{m_1^{-\frac{\gam}{2}}{\Lam}^{-1}\nabla\cdotp(v  \tht)}{L^2_{\om}} &\le C\Sob{p(D)^{-1}v}{H^1_\om}\Sob{m_1(D)^{\frac{\gam}{2}} \tht}{L^2_{\om}}+C\Sob{ \tht}{L^2_{\om}}\Sob{m_1(D)^{\frac{\gam}{2}}p(D)^{-1}v}{\Hdot^1_\om}\notag\\
   &\le C\nrm{q}_{\Hdot^
   {-1}_\om\cap L^2_{\om}}\Sob{m_1(D)^{\frac{\gam}{2}} \tht}{L^2_{\om}}+C\nrm{\tht}_{L^2_{\om}}\nrm{m_1(D)^{\frac{\gam}2}q}_{L^2_{\om}}.
\end{align}
Similarly, we apply \cref{cor:prod} to the second expression in \eqref{eq:stab:setup} with $(s,\bar{s})=(1,0)$, $(\om_1,\til{\om}_1)=(p^{-1}\om, m_1^{\frac{\gam}{2}}\om,)$, $(\om_2,\til{\om}_2)=(\om,p^{-1}m_1^{\frac{\gam}{2}}\om)$, $(\om_3,\til{\om}_3)=(p^{-1}\om, m_1^{\frac{\gam}{2}}\om)$, and $\Gamma=m_1^{\frac{\gam}{2}}$, obtain
\begin{align}\label{est:stab:0:b}
   \Sob{m_1(D)^{-\frac{\gam}{2}}\Div F_{q}(\tht)}{L^2_{\om}} &\le C\Sob{p(D)^{-1}v}{H^1_\om}\Sob{m_1(D)^{\frac{\gam}{2}}\nabla \tht}{L^2_{\om}}+C\Sob{\nabla \tht}{L^2_{\om}}\Sob{m_1(D)^{\frac{\gam}{2}}p(D)^{-1}v}{\Hdot^1_\om}\notag\\
   &\le C\nrm{q}_{\Hdot^
   {-1}_\om\cap L^2_{\om}}\nrm{m_1(D)^{\frac{\gam}{2}}\tht}_{\Hdot^{1}_\om}+C\nrm{\tht}_{\Hdot^{1}_\om}\nrm{m_1(D)^{\frac{\gam}{2}}q}_{L^2_{\om}}.
\end{align}

\subsubsection*{Case $\be \in(0,1)$} 
We apply \cref{cor:prod} with $(s,\bar{s})=(1, \be)$,  $(\om_1,\til{\om}_1)=(p^{-1}\om, m_1^{\frac{\gam}{2}}\om,)$, $(\om_2,\til{\om}_2)=(\om,p^{-1}m_1^{\frac{\gam}{2}}\om)$, $(\om_3,\til{\om}_3)=(p^{-1}\om, m_1^{\frac{\gam}{2}}\om)$, and $\Gamma=m_1^{\frac{\gam}{2}}$, obtain
\begin{align}\label{est:stab:1:a}
   \Sob{m_1(D)^{-\frac{\gam}{2}}\Div F_{q}(\tht)}{\Hdot^{\be}_\om} &\le C\Sob{p(D)^{-1}v}{H^1_\om}\Sob{m_1(D)^{\frac{\gam}{2}}\nabla \tht}{\Hdot^\be_\om}+C\Sob{\nabla \tht}{\Hdot^{\be}_{\om}}\Sob{m_1(D)^{\frac{\gam}{2}}p(D)^{-1}v}{\Hdot^1_\om}\notag\\
   &\le C\nrm{q}_{\Hdot^{-1}_\om\cap \Hdot^\be_\om}\nrm{m_1(D)^{\frac{\gam}2}\tht}_{\Hdot^{1+\be}_\om}+C\nrm{\tht}_{\Hdot^{1+\be}_\om}\nrm{m_1(D)^{\frac{\gam}2}q}_{\Hdot^{\be}_\om}.
\end{align}
 On the other hand, applying  \cref{cor:prod} with $(s,\bar{s})=(1-\be, \be)$,  $(\om_1,\til{\om}_1)=(p^{-1}\om, m_1^{\frac{\gam}{2}}\om,)$, $(\om_2,\til{\om}_2)=(\om,p^{-1}m_1^{\frac{\gam}{2}}\om)$, $(\om_3,\til{\om}_3)=(p^{-1}\om, m_1^{\frac{\gam}{2}}\om)$, and $\Gamma=m_1^{\frac{\gam}{2}}$, we obtain
\begin{align}\label{est:stab:1:b}
   \Sob{m_1(D)^{-\frac{\gam}2}\Div F_{q}(\tht)}{L^2_{\om}} &\le C\Sob{p(D)^{-1}v}{\Hdot^{1-\be}_\om}\Sob{m_1(D)^{\frac{\gam}{2}}\nabla \tht}{\Hdot^\be_\om}+C\Sob{\nabla \tht}{\Hdot^\be_\om}\Sob{m_1(D)^{\frac{\gam}{2}}p(D)^{-1}v}{\Hdot^{1-\be}_\om}\notag\\
   &\le C\nrm{q}_{L^2_{\om}}\nrm{m_1(D)^{\frac{\gam}2}\tht}_{\Hdot^{1+\be}_\om}+C\nrm{\tht}_{\Hdot^{1+\be}_\om}\nrm{m_1(D)^{\frac{\gam}2}q}_{L^2_{\om}}.
\end{align}
\subsubsection*{Case $\be=1$} We apply \cref{cor:prod}  with $(s,\bar{s})=(0,1)$,  $(\om_1,\til{\om}_1)=(p^{-1}\om, m_1^{\frac{\gam}{2}}\om,)$, $(\om_2,\til{\om}_2)=(\om,p^{-1}m_1^{\frac{\gam}{2}}\om)$, $(\om_3,\til{\om}_3)=(p^{-1}\om, m_1^{\frac{\gam}{2}}\om)$, and $\Gamma=m_1^{\frac{\gam}{2}}$, obtain
\begin{align}\label{est:stab:1:endpoint:a}
   \Sob{m_1(D)^{-\frac{\gam}2}\Div F_{q}(\tht)}{L^{2}_\om} &\le C\Sob{p(D)^{-1}v}{L^{2}_\om}\Sob{m_1(D)^{\frac{\gam}{2}}\nabla \tht}{\Hdot^1_\om}+C\Sob{\nabla \tht}{H^1_\om}\Sob{m_1(D)^{\frac{\gam}{2}}p(D)^{-1}v}{L^2_{\om}}\notag\\
   &\le C\nrm{q}_{L^{2}_\om}\nrm{m_1(D)^{\frac{\gam}2}\tht}_{\Hdot^{2}_\om}+C\nrm{\tht}_{H^{2}_\om}\nrm{m_1(D)^{\frac{\gam}2}q}_{L^2_{\om}}.
\end{align}
Similarly, we apply \cref{cor:prod} with $(s,\bar{s})=(1,1)$, $(\om_1, \til{\om}_1)=(p^{-1}\om, m_1^{\frac{\gam}{2}}\om)$, $(\om_2,\til{\om}_2)=(\om, p^{-1}m_1^{\frac{\gam}{2}}\om)$, $(\om_3, \til{\om}_3)=(p^{-1}\om, m_1^{\frac{\gam}{2}}\om)$, and $\Gamma=m_1^{\frac{\gam}{2}}$, we obtain
\begin{align}\label{est:stab:1:endpoint:b}
   \Sob{m_1(D)^{-\frac{\gam}2}\Div F_{q}(\tht)}{\Hdot^1_\om} &\le C\Sob{p(D)^{-1}v}{H^{1}_\om}\Sob{m_1(D)^{\frac{\gam}{2}}\nabla \tht}{\Hdot^1_\om}+C\Sob{\nabla \tht}{H^1_\om}\Sob{m_1(D)^{\frac{\gam}{2}}p(D)^{-1}v}{\Hdot^{1}_\om}\notag\\
   &\le C\nrm{q}_{H^1_\om}\nrm{m_1(D)^{\frac{\gam}2}\tht}_{\Hdot^{2}_\om}+C\nrm{\tht}_{H^{2}_\om}\nrm{m_1(D)^{\frac{\gam}2}q}_{\Hdot^{1}_\om}.
\end{align}

\subsubsection*{Case $\be \in (1,2]$} In this range, we dualize. Let $H=\Div F_{q}(\tht)$. Firstly, we see from \cref{lem:commutator2} applied with $s=2-\beta$ and $\eps\in(0,1]$ satisfying $\eps/2+s<1$, that
    \begin{align}\label{est:H:L2}
        \Sob{H}{L^2_{\om}}^{2}&=\lb \nabla \cdot((\nabla^{\perp}a(D)q)\tht),\om^2 H\rb+\lb a(D)\nabla \cdot((\nabla^{\perp}\tht)q),\om^2 H\rb\notag\\
            &=-\lb[\bdy^{\perp}_{\ell}a(D),\bdy_{\ell}\tht]q ,\om^2 H\rb\notag\\
            &\le C\Sob{\nabla\tht}{H^{\be-\de}}\left(\Sob{p_a(D)q}{\Hdot^{\eps/2}}\Sob{H}{L^2_{\om^2}}+\Sob{p_a(D)H}{\Hdot^{\eps/2}_{\om^2}}\Sob{q}{L^2}\right)\notag\\
            &\leq C\Sob{\tht}{H^{1+\be}_\om}\left(\Sob{q}{H^{\eps}}\Sob{H}{H^\eps}+\Sob{H}{{H}^\eps}\Sob{q}{L^2}\right)\notag\\
            &\le C\Sob{\tht}{H^{1+\be}_\om}\Sob{q}{H^{\eps}}\Sob{H}{H^{\eps}}.
    \end{align}
where we applied the embedding $H^{1+\be}_\om \subset H^{1+\be-\de}$ and \eqref{eq:omega:eps:bdd} to $p_a(D), \om(D)$, in obtaining the final two inequalities.

Now, to estimate in $\Hdot^\be_\om$, we localize. Let $j\in \mathbb{Z}$. Then
    \begin{align*}
        \Sob{{\Lam}^{\be}_{\om,j}m_1(D)^{-\frac{\gam}{2}}H}{L^2}^2
            =&\lb {\Lam}^{\be}_{\om,j}(\nabla^{\perp}a(D)q\cdot\nabla \tht),{\Lam}^{\be}_{\om,j}m_1(D)^{-\gam}{H}\rb-\lb {\Lam}^{\be}_{\om,j}a(D)(\nabla^{\perp}q\cdot\nabla \tht),{\Lam}^{\be}_{\om,j}m_1(D)^{-\gam}{H}\rb\\
            =& -\lb[\bdy^{\perp}_{\ell}a(D),\bdy_{\ell}\tht]{\Lam}^{\beta}_{\om,j}q ,{\Lam}^{\beta}_{\om,j}m_1(D)^{-\gam}H\rb+\lb[{\Lam}^{\beta}_{\om,j}, \bdy_{\ell}\tht]\bdy^{\perp}_{\ell}a(D)q ,{\Lam}^{\beta}_{\om,j}m_1(D)^{-\gam}H\rb\\
                &-\lb[{\Lam}^{\be}_{\om,j},\bdy_{\ell}\tht]\bdy^{\perp}_{\ell}q ,a(D){\Lam}^{\beta}_{\om,j} m_1(D)^{-\gam}H\rb\notag\\
            =&\ K_1+K_2+K_3.
    \end{align*}
Applying \cref{lem:commutator3} with $s=2-\be$, $\eps=0$, $\Gamma=m_1^{\gam}p^{-1}$, and Bernstein's inequality, we obtain
\begin{align*}
    |K_1|&\le Cm_1 (2^j)^{\gam}\Sob{\bdy_{\ell} \tht}{{H}^{\be}_\om}\Sob{{\Lam}^{\be}_{\om, j}q}{L^2}\Sob{{\Lam}^{\beta}_{\om,j}m_1(D)^{-\gam}H}{L^2}\notag\\
    &\le Cc_j\Sob{\tht}{H^{1+\be}_\om}\Sob{m_1(D)^{\frac{\gam}{2}}q}{\Hdot^{\be}_\om}\Sob{{\Lam}^{\beta}_{\om,j}m_1(D)^{-\frac{\gam}{2}}H}{L^2}.
\end{align*}
where we applied property \eqref{est:omega} for $m_1(D)$.

For $K_2$, we apply \cref{lem:commutator4} with $r=\be-1$, $(s, \bar{s})=(1,\be-1)$, $(\om_1,\til{\om}_1)=(p^{-1}\om,m_1^{\frac{\gam}{2}}\om)$, $(\om_2,\til{\om}_2)=(\om, p^{-1}m_1^{\frac{\gam}{2}}\om)$, $(\om_1,\til{\om}_1)=(p^{-1}\om,m_1^{\frac{\gam}{2}}\om)$, and $\Gam_1=\Gam_2=\Gam_3=m^{\gam/2}_1$, to obtain
\begin{align*}
    |K_2|\le Cc_{j}& m_1(2^j)^{\frac{\gam}{2}}\left\{\nrm{\bdy_{\ell}^{\perp}p(D)^{-1}a(D)q}_{H^{1}_\om}\nrm{m_1(D)^{\frac{\gam}{2}}\Lam \bdy_{\ell}\tht}_{\Hdot^{\be-1}_\om}\right. \notag \\ &\hspace{5 em}\left.+\nrm{\Lam \bdy_{\ell}\tht}_{H^{\be-1}_\om}\nrm{\bdy_{\ell}^{\perp}m_1(D)^{\frac{\gam}{2}}p(D)^{-1}a(D)q}_{\Hdot^{1}_\om}\right. \notag \\ &\hspace{5 em}\left. +\nrm{\bdy_{\ell}^{\perp}p(D)^{-1}a(D)q}_{\Hdot^{1}_\om}\nrm{m_1(D)^{\frac{\gam}{2}}\Lam \bdy_{\ell}\tht}_{\Hdot^{\be-1}_\om} \right\}\Sob{{\Lam}^{\beta}_{\om,j}m_1(D)^{-\gam}H}{L^2}\notag\\
    \le Cc_{j}&\left(\nrm{q}_{H^\be_\om}\nrm{m_1(D)^{\frac{\gam}2}\tht}_{\Hdot^{1+\be}_\om}+C\nrm{\tht}_{H^{1+\be}_\om}\nrm{m_1(D)^{\frac{\gam}2}q}_{\Hdot^{\be}_\om}\right)\nrm{{\Lam}^{\be}_{\om,j}m_1(D)^{-\frac{\gam}{2}}H}_{L^2}.
\end{align*}
Lastly, we apply \cref{lem:commutator4} with $r=\be-1$, $(s,\bar{s})=(\be-1,\be-1)$, $(\om_1,\til{\om}_1)= (\om_2,\tilde{\om}_2)=(\om_3,\tilde{\om}_3)=(\om, m_1^{\frac{\gam}{2}}\om)$, and $\Gam_1=\Gam_2=\Gam_3=m^{\gam/2}_1p^{-\frac{1}{2}}$, to obtain
\begin{align*}
    |K_3|\le Cc_{j}&m_1(2^j)^{\frac{\gam}{2}}p(2^j)^{-1}2^{(2-\be)j}\left\{\nrm{\bdy_{\ell}^{\perp}q}_{H^{\be-1}_\om}\nrm{\bdy_{l}m_1(D)^{\frac{\gam}{2}}\Lam \tht}_{\Hdot^{\be-1}_\om}\right. \notag \\ 
    &\hspace{11 em}\left.+\nrm{\bdy_{l}\Lam \tht}_{H^{\be-1}_\om}\nrm{m_1(D)^{\frac{\gam}{2}}\bdy_{\ell}^{\perp}q}_{\Hdot^{\be-1}}\right. \notag \\ 
    &\hspace{11 em}\left. +\nrm{\bdy_{\ell}^{\perp}q}_{\Hdot^{\be-1}_\om}\nrm{\bdy_{l}m_1(D)^{\frac{\gam}{2}}\Lam \tht}_{\Hdot^{\be-1}_\om} \right\}\nrm{a(D){\Lam}^{\be}_{\om,j}m_1(D)^{-\gam}H}_{L^2}\notag\\
    \le Cc_{j}&\left(\nrm{q}_{H^\be_\om}\nrm{m_1(D)^{\frac{\gam}2}\tht}_{\Hdot^{1+\be}_\om}+C\nrm{\tht}_{H^{1+\be}_\om}\nrm{m_1(D)^{\frac{\gam}2}q}_{\Hdot^{\be}_\om}\right)\nrm{{\Lam}^{\be}_{\om,j}m_1(D)^{-\frac{\gam}{2}}H}_{L^2}.
\end{align*}
Collecting the estimates of $K_1$, $K_2$ and $K_3$, summing over $j$ and applying the Cauchy-Schwarz inequality, we obtain 
\begin{align}\label{est:stab:2}
    \Sob{ m_1(D)^{-\frac{\gam}{2}}\Div F_{q}(\tht)}{\Hdot^{\be}_\om}&\le C\nrm{q}_{H^\be_\om}\nrm{m_1(D)^{\frac{\gam}2}\tht}_{\Hdot^{1+\be}_\om}+C\nrm{\tht}_{H^{1+\be}_\om}\nrm{m_1(D)^{\frac{\gam}2}q}_{\Hdot^{\be}_\om}. 
\end{align}
when $\be\in(1,2]$. Finally, we combine \eqref{est:stab:2} with \eqref{est:H:L2} to arrive at
    \begin{align}\label{est:stab:2:final}
        \Sob{m_1(D)^{-\frac{\gam}2}\Div F_{q}(\tht)}{H^\be_\om}&\le C\nrm{q}_{ H^\be_\om}\nrm{m_1(D)^{\frac{\gam}2}\tht}_{\Hdot^{1+\be}_\om}+C\nrm{\tht}_{H^{1+\be}_\om}\nrm{m_1(D)^{\frac{\gam}2}q}_{H^{\be}_\om}.
    \end{align}
The estimates \eqref{est:stab:0:a}-\eqref{est:stab:1:endpoint:b}, and \eqref{est:stab:2:final} together complete the proof.
\end{proof}

\section{Well-posedness: Proof of Theorem \ref{thm:main:dgsqg}}\label{sect:wellposed}

We will now establish local existence of a unique solution to \eqref{eq:dgsqg}, which possesses the property of instantaneous smoothing, and is continuous with respect to initial data. The argument hinges on the simple observation that since $\nabla^{\perp}\tht$ is divergence-free, one can express equation \eqref{eq:dgsqg} as 
    \begin{align}\label{eq:dgsqg:mod:form}
        \partial_{t}\theta+ m(D) \tht+ \Div F_{-\tht}(\tht)=0 ,\quad\tht(0,x)=\tht_0(x),
    \end{align}
where $F$ is as defined in \eqref{def:mod:flux}, for all $\be\in[0,2]$. In particular, see that \eqref{eq:dgsqg:mod:form} has the structure of \eqref{eq:mod:claw} with $q=-\tht$ and $G\equiv0$. One may then formally deduce apriori estimates for \eqref{eq:dgsqg} from the protean system \eqref{eq:mod:claw}. A rigorous proof of well-posedness then reduces to construction of the solution; this is a straightforward matter and can be dealt with by considering a standard artificial viscosity approximation, for which all apriori estimates hold independently of the viscosity parameter. The relevant details of this argument are provided in \cref{app:well:posed}. In what follows, we perform the formal analysis to establish local well-posedness.

\subsection{Existence}\label{sect:exist}

Upon setting $q=-\tht$ and $G\equiv0$ in \eqref{eq:mod:claw}, we see from \eqref{est:Hs:final} that for any $\be \in [0,2]$, for $\lam\equiv0$, we have
	\begin{align}\label{est:balance:exist}
		   \frac{d}{dt}\Sob{\tht}{\Hdot^{1+\be}_{\om}\cap L^2_{\om}}^2+\frac{3}2\Sob{m_1(D)^{\frac{1}2}\tht}{\Hdot^{1+\be}_{\om}\cap L^2_{\om}}^{2}\leq C \left(1+\Sob{\tht}{H^{1+\be}_\om}^{\frac{1}{1-\gam}}+\Sob{m_1(D)^{\frac{\gam}2}\tht}{\Hdot^{1+\be}_\om}^{\frac{2}{2-\gam}}\right)\Sob{\tht}{\Hdot^{1+\be}_{\om}\cap L^2_{\om}}^2.
	\end{align}
Observe that by \eqref{est:interpolation:weighted} and Young's inequality (using the fact that $\gam<1$), we obtain
    \begin{align}
        C\left(\Sob{\tht}{H^{1+\be}_\om}^{\frac{1}{1-\gam}}+\Sob{m_1(D)^{\frac{\gam}2}\tht}{\Hdot^{1+\be}_\om}^{\frac{2}{2-\gam}}\right)\Sob{\tht}{H^{1+\be}_\om}^2&\leq C\left(\Sob{\tht}{H^{1+\be}_\om}^{\frac{3-2\gam}{1-\gam}}+\Sob{m_1(D)^{\frac{1}2}\tht}{\Hdot^{1+\be}_\om}^{\frac{2\gam}{2-\gam}}\Sob{\tht}{H^{1+\be}_\om}^{\frac{2(3-2\gam)}{2-\gam}}\right)\notag\\
        &\leq C\Sob{\tht}{H^{1+\be}_\om}^{\frac{3-2\gam}{1-\gam}}+\frac{1}2\Sob{m_1(D)^{\frac{1}2}\tht}{\Hdot^{1+\be}_\om}^2.\notag
    \end{align}
By Gronwall's inequality, we deduce the existence of a time $T=T(\Sob{\tht_0}{{H}^{1+\be}_\om})$ such that
 \begin{align}\label{est:apriori:dgsqg:b}
        \sup_{t\in[0,T]}\Sob{\tht(t)}{H^{1+\be}_\om}^2+\int_0^T\Sob{m(D)^{\frac{1}{2}}\tht(t)}{H^{1+\be}_\om}^2dt\le C(1+\Sob{\tht_{0}}{H^{1+\be}_\om}^2).
    \end{align}
Similarly, when $\be=0$, and $\tht_0 \in H^1_\om \cap \Hdot^{-1}_\om$, we invoke \eqref{est:Hminus1:final} with $q=\tht$ and \eqref{est:Hs:final} with $q=\tht$ and $\s=1$ to deduce the existence of a time $T=T(\Sob{\tht_0}{{H}^{1}_\om\cap \Hdot^{-1}_\om})$ such that 
    \begin{align}\label{est:apriori:dgsqg:a}
        \sup_{t\in[0,T]}\Sob{\tht(t)}{H^1_\om\cap\Hdot^{-1}_{\om}}^2+\int_0^T\Sob{m(D)^{\frac{1}{2}}\tht(t)}{H^1_\om\cap\Hdot^{-1}_{\om}}^2dt\le C(1+\Sob{\tht_{0}}{H^1_\om\cap\Hdot^{-1}_{\om}}^2),
    \end{align}
for some constant $C>0$
   
An artificial viscosity approximation can then be used to construct a solution $\tht\in C([0,T];H^1_\om\cap\Hdot^{-1}_\om)$, provided that $\tht_0\in H^1_\om\cap\Hdot^{-1}_\om$, when $\be=0$, and $\tht\in C([0,T];H^{1+\be}_\om)$, provided that $\tht_0\in H^{1+\be}_\om$, when $\be\in(0,2]$.

\subsection{Smoothing} 

Upon setting $q=-\tht$ and $G\equiv0$ in \eqref{eq:mod:claw}, we see from 
\eqref{est:Hs:final} and the corresponding inequality but for the case $\be \in (1,2]$ and performing the same analysis from \cref{sect:exist} that
\begin{align*}
          \frac{d}{dt}\Sob{\tilde{\tht}}{\Hdot^{1+\be}_{\om}}^2+\frac{3}2\Sob{m_1(D)^{\frac{1}2}\til{\tht}}{\Hdot^{1+\be}_{\om}}^{2}&\leq C \left(1+\Sob{\til{\tht}}{H^{1+\be}_\om}^{\frac{1}{1-\gam}}+\Sob{m_1(D)^{\frac{\gam}2}\til{\tht}}{\Hdot^{1+\be}_\om}^{\frac{2}{2-\gam}}\right)\Sob{\tilde{\tht}}{\Hdot^{1+\be}_{\om}\cap L^2_{\om}}^2\notag\\
        &\leq C\left(1+\Sob{\til{\tht}}{H^{1+\be}_\om}^2\right)^{\frac{3-2\gam}{2(1-\gam)}}+\frac{1}2\Sob{m_1(D)^{\frac{1}2}\til{\tht}}{\Hdot^{1+\be}_\om}^2.\notag
\end{align*}
Again, it follows from Gronwall's inequality that there exists $T>0$ such that
    \begin{align}\label{est:smoothing}
        \sup_{0\le t\le T}\Sob{E^{\lam_1t}_{ \nu}\tht(t)}{\Hdot^{1+\be}_\om}\le C(1+\Sob{\tht_0}{H^{1+\be}_\om}),
    \end{align}
as desired.

\subsection{Uniqueness} Let $\be \in [0,2]$. For $j=1,2$, let $\tht_0^{(j)}\in H^1_\om\cap \Hdot^{-1}_\om$ when $\be=0$ and $\tht_0^{(j)}\in H^{1+\be}_\om$ when $\be\in(0,2]$. Suppose that $\tht^{(1)},\tht^{(2)}$ are two solutions of \eqref{eq:dgsqg} corresponding to initial data $\tht_0^{(1)}, \tht_0^{(2)}$ such that $\tht^{(j)}\in C([0,T];\Hdot^{-1}_\om\cap H^1_\om)$, when $\be=0$, and $\tht^{(j)}\in C([0,T];H^{1+\be}_\om)$, when $\be\in(0,2]$, for $j=1,2$, where $T$ is the local existence time obtained from \cref{sect:exist}. To prove uniqueness, it will be convenient to introduce the following notation
    \begin{align}\label{def:Y}
        Y^\be:=
            \begin{cases}
                L^2_{\om}\cap\Hdot^{-1}_\om,&\be=0\\
                L^2_{\om},&\be\in(0,1)\\
                H^\be_\om,&\be\in[1,2].
            \end{cases}
    \end{align}

Let $\Tht:=\tht^{(1)}-\tht^{(2)}$. Then $\Tht$ is governed by
    \begin{align}\label{eq:difference:intro}
        \partial_{t}\Tht+m(D) \Tht+\Div F_{-\tht^{(1)}}(\Tht)=\Div F_{\Tht}(\tht^{(2)}),\quad\Tht(0,x)=\Tht_0(x).
    \end{align}
Observe that \eqref{eq:difference:intro} has the structure of \eqref{eq:mod:claw} with $q=-\tht^{(1)}$ and $G=\Div F_{\Tht}(\tht^{(2)})$. Then from \eqref{est:Hminus1:final}
    \begin{align}\label{est:uniqueness:a:pre}
            \frac{d}{dt}\Sob{\Tht}{\dot{H}^{-1}_{\om}\cap L^2_{\om}}^2+c\Sob{m(D)^{\frac{1}2}\Tht}{\dot{H}^{-1}_{\om}\cap L^2_{\om}}^{2}\leq& C\left(1+\Sob{\tht^{(1)}}{ \Hdot^{-1}_\om}^{\frac{1}{1-\gam}}+\Sob{\tht^{(1)}}{\Hdot^{1}_\om}^{\frac{1}{1-\gam}}+\Sob{m_1(D)^{\frac{\gam}2}\tht^{(1)}}{\Hdot^{1}_\om}^{\frac{2}{2-\gam}}\right)\Sob{\Tht}{\dot{H}^{-1}_{\om}\cap L^2_{\om}}^2\notag\\&+C\Sob{m_1(D)^{-\frac{1}{2}}\Div F_{\Tht}(\tht^{(2)})}{\Hdot^{-1}_\om \cap L^2_{\om}}^2.
    \end{align}
when $\be=0$. 
By interpolation inequality and Young's inequality, we obtain
\begin{align}\label{E:interpolation:uniqueness}
    \Sob{m_1(D)^{\frac{\gam}2}\tht^{(1)}}{\Hdot^{1}_\om}^{\frac{2}{2-\gam}}&\le C\Sob{\tht^{(1)}}{\Hdot^{1}_\om}^{\frac{2(1-\gam)}{2-\gam}}\Sob{m_1(D)^{\frac{1}{2}}\tht^{(1)}}{\Hdot^{1}_\om}^{\frac{2\gam}{2-\gam}}\notag\\
    &\le C\Sob{\tht^{(1)}}{\Hdot^{\s}_\om}+C\Sob{m_1(D)^{\frac{1}{2}}\tht^{(1)}}{\Hdot^{\s}_\om}^2.
\end{align}
From \eqref{est:uniqueness:a:pre} and \eqref{E:interpolation:uniqueness}, we obtain
 \begin{align}\label{est:uniqueness:a}
            \frac{d}{dt}\Sob{\Tht}{\dot{H}^{-1}_{\om}\cap L^2_{\om}}^2+c\Sob{m(D)^{\frac{1}2}\Tht}{\dot{H}^{-1}_{\om}\cap L^2_{\om}}^{2}\leq& C\left(1+\Sob{\tht^{(1)}}{ \Hdot^{-1}_\om}^{\frac{1}{1-\gam}}+\Sob{\tht^{(1)}}{\Hdot^{1}_\om}^{\frac{1}{1-\gam}}+\Sob{m_1(D)^{\frac{\gam}2}\tht^{(1)}}{\Hdot^{1}_\om}^{2}\right)\Sob{\Tht}{\dot{H}^{-1}_{\om}\cap L^2_{\om}}^2\notag\\&+C\Sob{m_1(D)^{-\frac{1}{2}}\Div F_{\Tht}(\tht^{(2)})}{\Hdot^{-1}_\om \cap L^2_{\om}}^2.
    \end{align}
Similarly, when $\be\in(0,1)$, $\s \in (-1,2]$ or when $\be\in[1,2]$, $\s \in [1,1+\be]$, we have from \eqref{est:Hs:final} that
    \begin{align}\label{est:uniqueness:b}
         \frac{d}{dt}\Sob{\Tht}{\dot{H}^{\s}_{\om}\cap L^2_{\om}}^2+c\Sob{m(D)^{\frac{1}2}\Tht}{\dot{H}^{\s}_{\om}\cap L^2_{\om}}^{2}\leq& C\left(1+\Sob{\tht^{(1)}}{H^{1+\be}_\om}^{\frac{1}{1-\gam}}+\Sob{m_1(D)^{\frac{\gam}2}\tht^{(1)}}{\Hdot^{1+\be}_\om}^{2}\right)\Sob{\Tht}{\dot{H}^{\s}_{\om}\cap L^2_{\om}}^2\notag\\&+C\Sob{m_1(D)^{-\frac{1}{2}}\Div F_{\Tht}(\tht^{(2)})}{\Hdot^{\s}_\om \cap L^2_{\om}}^2.
    \end{align}
    Applying \cref{lem:stab:G}, interpolation inequality, and Young's inequality, we see that when $\be=0$, we have
    \begin{align}\label{est:uniqueness:div:a}
        \Sob{m_1(D)^{-\frac{\gam}{2}}\Div F_{\Tht}(\tht^{(2)})}{\Hdot^{-1}_\om \cap L^2_{\om}}^2&\leq C\nrm{\Tht}_{\Hdot^
   {-1}_\om\cap L^2_{\om}}^2 \nrm{m_1(D)^{\frac{\gam}{2}}\tht^{(2)}}_{H^{1}_\om}^2+C\nrm{\tht^{(2)}}_{H^{1}_\om}^2\nrm{m_1(D)^{\frac{\gam}{2}}\Tht}_{L^2_{\om}}^2\notag\\
   &\le C\nrm{\Tht}_{\Hdot^
   {-1}_\om\cap L^2_{\om}}^2\nrm{m_1(D)^{\frac{\gam}{2}}\tht^{(2)}}_{H^{1}_\om}^2+C\nrm{\tht^{(2)}}_{H^{1}_\om}^2\nrm{\Tht}_{L^2_{\om}}^{2-2\gam}\nrm{m_1(D)^{\frac{1}{2}}\Tht}_{L^2_{\om}}^{2\gam}\notag\\
   &\le C\left(\nrm{m_1(D)^{\frac{\gam}{2}}\tht^{(2)}}_{H^{1}_\om}^2+\nrm{\tht^{(2)}}_{H^{1}_\om}^{\frac{2}{1-\gam}}\right)\nrm{\Tht}_{\dot{H}^{-1}_{\om}\cap L^2_{\om}}^2+\frac{c}{2}\nrm{m_1(D)^{\frac{1}{2}}\Tht}_{L^2_{\om}}^{2}.
    \end{align}
Similarly, when $\be\in(0,1)$, we have
    \begin{align}\label{est:uniqueness:div:b}
         \Sob{m_1(D)^{-\frac{\gam}{2}}\Div F_{\Tht}(\tht^{(2)})}{L^2_{\om}}&\leq C\left(\nrm{m_1(D)^{\frac{\gam}{2}}\tht^{(2)}}_{\Hdot^{1+\be}_\om}^2+\nrm{\tht^{(2)}}_{H^{1+\be}_\om}^{\frac{2}{1-\gam}}\right)\nrm{\Tht}_{L^2_{\om}}^2+\frac{c}{2}\nrm{m_1(D)^{\frac{1}{2}}\Tht}_{L^2_{\om}}^{2},
    \end{align}
and when $\be\in[1,2]$, we have
    \begin{align}\label{est:uniqueness:div:d}
        \Sob{m_1(D)^{-\frac{\gam}{2}}\Div F_{\Tht}(\tht^{(2)})}{H^{\be}_\om}
            &\leq C\left(\nrm{m_1(D)^{\frac{\gam}{2}}\tht^{(2)}}_{\Hdot^{1+\be}_\om}^2+\nrm{\tht^{(2)}}_{H^{1+\be}_\om}^{\frac{2}{1-\gam}}\right)\nrm{\Tht}_{{H}^{\be}_{\om}}^2+\frac{c}{2}\nrm{m_1(D)^{\frac{1}{2}}\Tht}_{{H}^{\be}_{\om}}^{2}. 
    \end{align}
Combining \eqref{est:uniqueness:a} and \eqref{est:uniqueness:div:a} yields
    \begin{align}
        \frac{d}{dt}&\Sob{\Tht}{\dot{H}^{-1}_{\om}\cap L^2_{\om}}^2+\frac{c}{2}\Sob{m(D)^{\frac{1}2}\Tht}{\dot{H}^{-1}_{\om}\cap L^2_{\om}}^{2}\notag\\
        \leq& C\left(1+\Sob{\tht^{(1)}}{\Hdot^{1}_\om \cap \Hdot^{-1}_\om}^{\frac{1}{1-\gam}}+\nrm{\tht^{(2)}}_{H^{1}_\om}^{\frac{2}{1-\gam}}+\Sob{m_1(D)^{\frac{\gam}2}\tht^{(1)}}{\Hdot^{1}_\om}^{2}+\nrm{m_1(D)^{\frac{\gam}{2}}\tht^{(2)}}_{H^{1}_\om}^2\right)\Sob{\Tht}{\dot{H}^{-1}_{\om}\cap L^2_{\om}}^2,\label{est:Hminus1:uniqueness}
    \end{align}
when $\be=0$. Combining \eqref{est:uniqueness:b} with $\s=0$ and \eqref{est:uniqueness:div:b} yields
    \begin{align}\label{est:Hsig:uniqueness:1}
          \frac{d}{dt}&\Sob{\Tht}{ L^2_{\om}}^2+\frac{c}{2}\Sob{m(D)^{\frac{1}2}\Tht}{L^2_{\om}}^{2}\notag\\&\leq C\left(1+\Sob{\tht^{(1)}}{H^{1+\be}_\om}^{\frac{1}{1-\gam}}+\nrm{\tht^{(2)}}_{H^{1+\be}_\om}^{\frac{2}{1-\gam}}+\Sob{m_1(D)^{\frac{\gam}2}\tht^{(1)}}{\Hdot^{1+\be}_\om}^{2}+\nrm{m_1(D)^{\frac{\gam}{2}}\tht^{(2)}}_{\Hdot^{1+\be}_\om}^2\right)\Sob{\Tht}{L^2_{\om}}^2,
    \end{align}
when $\be\in(0,1)$. 
Lastly, combining \eqref{est:uniqueness:b} with $\s=\be\in(1,2]$ and \eqref{est:uniqueness:div:d} yields
    \begin{align}\label{est:Hsig:uniqueness:3}
          \frac{d}{dt}&\Sob{\Tht}{ {H}^{\be}_{\om}}^2+\frac{c}{2}\Sob{m(D)^{\frac{1}2}\Tht}{{H}^{\be}_{\om}}^{2}\notag\\&\leq C\left(1+\Sob{\tht^{(1)}}{H^{1+\be}_\om}^{\frac{1}{1-\gam}}+\nrm{\tht^{(2)}}_{H^{1+\be}_\om}^{\frac{2}{1-\gam}}+\Sob{m_1(D)^{\frac{\gam}2}\tht^{(1)}}{\Hdot^{1+\be}_\om}^{2}+\nrm{m_1(D)^{\frac{\gam}{2}}\tht^{(2)}}_{\Hdot^{1+\be}_\om}^2\right)\Sob{\Tht}{{H}^{\be}_{\om}}^2,
    \end{align}
Recall that $\tht^{(j)}\in L^\infty_T(H^{1}_\om\cap\Hdot^{-1}_\om)\cap L^2_T(H^1_{\om m^{1/2}}\cap\Hdot^{-1}_{\om m^{1/2}})$, when $\be=0$, and $\tht^{(j)}\in L^\infty_T H^{1+\be}_\om\cap L^2_T H^{1+\be}_{\om m^{1/2}}$, when $\be\in(0,2]$, for $j=1,2$. In each of \eqref{est:Hminus1:uniqueness}, \eqref{est:Hsig:uniqueness:1},  \eqref{est:Hsig:uniqueness:3}, we may therefore apply Gronwall inequality to deduce that there exists a $C>0$, depending on $T>0$, such that
    \begin{align}
        \sup_{0\leq t\leq T}\Sob{\Tht(t)}{Y^\be}&\leq C\Sob{\Tht_0}{Y^\be}\label{eq:uniqueness}
    \end{align}
for all $\be\in[0,2]$. In particular, if $\tht_0^{(1)}=\tht_0^{(2)}$, then $\tht^{(1)}=\tht^{(2)}$ in $Y^\be$ over $[0,T]$.

\subsection{Continuous Dependence on Initial Data}\label{sect:cont:dep}
Let $\be\in [0,2]$ and denote the data-to-solution operator of \eqref{eq:dgsqg} by $\Phi_0:H^{1}_\om\cap\Hdot^{-1}_\om\goesto\bigcup_{T>0}C([0,T];H^{1}_\om\cap\Hdot^{-1}_\om)$, when $\be=0$, and $\Phi_\be: H^{1+\be}_\om\goesto\bigcup_{T>0}C([0,T]; H^{1+\be}_\om)$, when $\be\in(0,2]$. Existence and uniqueness of solutions to \eqref{eq:dgsqg} establishes that $\Phi$ is well-defined. It will be convenient to use of the notation $\tht(\cdotp;\tht_0)=\Phi(\tht_0)(\cdotp)$. 

We will now show that $\Phi_\be$ is continuous, for all $\be\in[0,2]$. To prove this, it will be convenient to introduce the following notation
    \begin{align}\label{def:Xb}
        X^\be:=\begin{cases}
            H^1_\om\cap\Hdot^{-1}_\om,&\be=0\\
            H^{1+\be}_\om,&\be\in(0,2].
        \end{cases}
    \end{align}
Ultimately, we will show that given $\tht_0\in H^1_\om\cap\Hdot^{-1}_\om$, when $\be=0$, and $\tht_{0} \in H^{1+\be}_\om$, when $\be\in(0,2]$, there exists a neighborhood $U_0\subset H^{1+\be}_\om$ of $\tht_{0}$ and a time, $T>0$, such that for any sequence of initial data $\{\tht^{n}_{0}\}\subset U_0$, we have
    \[
        \lim_{n \to \infty} \Sob{\tht^{n}_{0}-\tht_{0}}{X^\be}=0\quad\text{implies}\quad\lim_{n \to \infty}\Sob{\tht^n-\tht}{X^\be}=0.
    \]
We will make use of the convention that $\tht^\infty=\tht=\Phi(\tht_0)$.

Let $T_0>0$ denote the local existence time of $\tht$. We define the neighborhoods $U_\be\subset X^\be$  by
    \begin{align}\label{def:U}
            U_\be&:=\{f\in X^\be: \Sob{f-\tht_{0}}{X^\be}<\Sob{\tht_{0}}{X^\be}\}.
    \end{align}
Let 
    \begin{align}\label{def:K}
        K_\be:=\sup_{f \in U_\be}\Sob{f}{X^\be}.
    \end{align}
Denote by $\tht^{n}(\cdotp;\tht^n_0)$ the solution to \eqref{eq:dgsqg:mod:form} corresponding to initial data $\tht^n_0$. Then by \eqref{est:apriori:dgsqg:a} and \eqref{est:apriori:dgsqg:b} there exists a constant $C>0$ such that
    \[
        \sup_{n>0}\left(\Sob{\tht^n}{L^{\infty}_{T}X^\be}+\Sob{m(D)^{\frac{1}{2}}\tht^n}{L^{2}_{T}X^\be}\right)\le C K_\be,
    \]
for some $0<T\leq T_0$ dependent on $\Sob{\tht_0}{X^\be}$.

Upon returning to \eqref{est:Hminus1:uniqueness}, \eqref{est:Hsig:uniqueness:1}, \eqref{est:Hsig:uniqueness:3}, and applying \eqref{eq:uniqueness}, respectively in $\be$, we obtain
    \begin{align}\label{eq:stability}
         \sup_{0\leq t\leq T}\Sob{(\tht^n-\tht)(t)}{Y^\be}^2+c\int_0^T\Sob{m(D)^{\frac{1}{2}}(\tht^n-\tht)(s)}{Y^\be}^2ds\leq C(T, K_\be)\Sob{\tht^n_0-\tht_0}{Y^\be}^2,
    \end{align}
for some $C(T,K_\be)>0$, for all $\be\in[0,2]$, where $Y^\be$ was defined in \eqref{def:Y}. This implies
    \begin{align}\label{eq:convergence}
       \lim_{n\goesto\infty}\left( \Sob{\tht^n-\tht}{L^\infty_TY^\be}^2+\Sob{m(D)^{\frac{1}{2}}(\tht^n-\tht)}{L^2_TY^\be}^2\right)=0,
    \end{align}
for all $\be\in[0,2]$.
    
To complete the proof, it suffices to show that $\nabla \tht^n \to \nabla \tht$ in $L^\infty_{T}\Hdot^{\be}_\om$. For this, let 
    \begin{align}\label{def:vsigma:zeta}
       \vs^n:=(\vs^n_1,\vs^n_2),\quad\vs:=(\vs_{1},\vs_2),\quad \ze^n:=(\ze^n_1,\ze^n_2),\quad \ze:=(\ze_1,\ze_2).
    \end{align}
Then we decompose $\nabla\tht$ into $\bdy_\ell\tht=\vs_\ell+\ze_\ell$ and $\nabla\tht^n$ into $\bdy_{\ell}\tht^{n}=\vs_{\ell}^{n}+\ze_{\ell}^{n}$, for $\ell=1,2$, where we assume the components of ($\vs^{n},\ze^{n}$) are governed by the equations
    \begin{align}\label{eq:vs}
            \partial_{t}\vs^{n}_{\ell}+m(D)\vs^{n}_{\ell}+ \Div F_{-\tht^{n}}(\vs^{n}_{\ell})=G_{\ell},\quad\vs^{n}_{\ell}(0,x)=\bdy_{\ell}\tht_0(x),
    \end{align}
and 
    \begin{align}\label{eq:ze}
            \partial_{t}\ze^{n}_{\ell}+m(D)\zeta^{n}_{\ell}+ \Div F_{-\tht^n}(\ze^{n}_{\ell})=G_{\ell}^{n}-G_{\ell},\quad \ze^{n}_{\ell}(0,x)=\bdy_{\ell}\tht^{n}_0(x)-\bdy_{\ell}\tht_{0}(x),
    \end{align}
for $n\in\mathbb{N}\cup\{\infty\}$ and each $\ell=1,2$, where
    \begin{align}\label{def:G}
        G:=(G_1,G_2),\quad G_{\ell}:=\Div F_{\bdy_{\ell}\tht}(\tht), \quad G^n:=(G^n_1,G^n_2),\quad G^{n}_{\ell}:=\Div F_{\bdy_{\ell}\tht^n}(\tht^{n}),\quad\ell=1,2.
    \end{align}
Note that to be consistent with $\tht^\infty=\tht$, we also make use of $\vs^\infty=\vs$ and $\ze^\infty=\ze$. Now observe that both \eqref{eq:vs} and \eqref{eq:ze} have the structure of \eqref{eq:mod:claw}. Indeed, we see that \eqref{eq:vs} has the structure upon making the replacement $q\mapsto -\tht^n$ and $G\mapsto G_\ell$ and \eqref{eq:ze} has the structure upon making the replacement $q\mapsto -\tht^n$ and $G\mapsto G^n_\ell-G_\ell$. Since $\tht_0\in X^\be$, $\nabla\tht_0\in Y^\be$, $\tht^n\in L^\infty_TX^\be$, and $m(D)^{\frac{1}2}\tht^n\in L^2_TX^\be$, for all $n\in\NN\cup\{\infty\}$, in order to apply \cref{thm:modclaw:wellposed}, it suffices to check that $m(D)^{-1/2}G_\ell\in L^2_TY^\be$ and $m(D)^{-1/2}(G^n_\ell-G_\ell)\in L^2_TY^\be$. This can be checked with \cref{lem:stab:G}.

First, we see that \cref{lem:stab:G} implies
    \begin{align}\notag
        \Sob{m(D)^{-\frac{\gam}{2}}G}{Y^\be}
            &\leq C\Sob{\nabla\tht}{Y^{\be}}\Sob{m_1 (D)^{\frac{\gam}{2}}\tht}{H^{1+\be}_\om}+\nrm{\tht}_{H^{1+\be}_\om}\nrm{m(D)^{\frac{\gam}{2}}\nabla\tht}_{H^{\be}_\om}\le C\Sob{\tht}{H^{1+\be}_\om}\Sob{m_1 (D)^{\frac{\gam}{2}}\tht}{H^{1+\be}_\om}.
    \end{align}
for all $\be\in[0,2]$. On the other hand, to study $m(D)^{-1/2}(G^n_\ell-G_\ell)$, observe that
    \begin{align}\label{eq:G:difference}
        G_{\ell}^{n}-G_{\ell}=\Div F_{\bdy_{\ell}\tht^{n}}(\tht^n-\tht)+\Div F_{(\bdy_{\ell}\tht^{n}-\bdy_{\ell}\tht)}(\tht).
    \end{align}
Hence, \cref{lem:stab:G} implies
    \begin{align}\notag
        \Sob{m(D)^{-\frac{\gam}{2}}(G^n-G)}{Y^\be}
            &\leq \Sob{m(D)^{-\frac{\gam}{2}}\Div F_{\nabla\tht^n}(\tht^n-\tht)}{Y^\be}+\Sob{m(D)^{-\frac{\gam}{2}}\Div F_{\nabla(\tht^n-\tht)}(\tht)}{Y^\be}.
    \end{align}
By \cref{lem:stab:G} we see that
    \begin{align}\notag
        \Sob{m(D)^{-\frac{\gam}{2}}\Div F_{\nabla\tht^n}(\tht^n-\tht)}{Y^\be}&\leq C\Sob{\nabla\tht^{n}}{Y^{\be}}\Sob{m_1 (D)^{\frac{\gam}{2}}(\tht^n-\tht)}{H^{1+\be}_\om}+\nrm{\tht^n-\tht}_{H^{1+\be}_\om}\nrm{m(D)^{\frac{\gam}{2}}\nabla\tht^n}_{H^{\be}_\om}\\&\leq C\Sob{\tht^{n}}{H^{1+\be}_\om}\Sob{m_1 (D)^{\frac{\gam}{2}}(\tht^n-\tht)}{H^{1+\be}_\om}+\nrm{\tht^n-\tht}_{H^{1+\be}_\om}\nrm{m(D)^{\frac{\gam}{2}}\tht^n}_{H^{1+\be}_\om}\notag.
    \end{align}
Similarly, we have
\begin{align}\notag
        \Sob{m(D)^{-\frac{\gam}{2}}\Div F_{\nabla(\tht^n-\tht)}(\tht)}{Y^\be}&\leq C\Sob{\tht}{H^{1+\be}_\om}\Sob{m_1 (D)^{\frac{\gam}{2}}(\tht^n-\tht)}{H^{1+\be}_\om}+\nrm{\tht^n-\tht}_{H^{1+\be}_\om}\nrm{m(D)^{\frac{\gam}{2}}\tht}_{H^{1+\be}_\om}\notag.
    \end{align}
Thus
    \begin{align}\notag
        \Sob{m(D)^{-\frac{\gam}{2}}(G^n-G)}{Y^\be}\leq C\left(\Sob{\tht^n-\tht}{H^{1+\be}_\om}+\nrm{m(D)^{\frac{\gam}{2}}(\tht^n-\tht)}_{H^{1+\be}_\om}\right)\left(\Sob{\tht}{H^{1+\be}_\om \cap H^{1+\be}_{\om m^{1/2}}}+\nrm{\tht^n}_{H^{1+\be}_\om\cap 
H^{1+\be}_{\om m^1/2}}\right)
    \end{align}
Hence, \cref{thm:modclaw:wellposed} applies to guarantee a unique solution $\vs^n\in C([0,T];Y^\be)$ to \eqref{eq:vs}, for all $\be\in[0,2]$.

By \eqref{eq:convergence}, we may invoke \cref{thm:stab} to ensure that
    \begin{align}\notag
        \lim_{n \to \infty} \left(\Sob{\vs^{n}-\vs}{L^{\infty}_{T}(H^1_\om\cap\Hdot^{-1}_\om)}+ \Sob{m(D)^{\gam/2}(\vs^{n}-\vs)}{L^{2}_{T}(H^1_\om\cap\Hdot^{-1}_\om)} \right)=0,
    \end{align}
when $\be=0$, and
    \begin{align}\notag
        \lim_{n \to \infty} \left(\Sob{\vs^{n}-\vs}{L^{\infty}_{T}H^\be_\om}+ \Sob{m(D)^{\gam/2}(\vs^{n}-\vs)}{L^{2}_{T}H^\be_\om} \right)=0,
    \end{align}
when $\be\in(0,2]$. Similarly, the same relations hold for $\ze^n-\ze$.

To conclude the proof, we apply the triangle inequality and observe that
    \[
       \limsup_{n\goesto\infty}\Sob{\nabla \tht^{n}-\nabla \tht}{L^{\infty}_{T}\Hdot^\be_\om}\leq \limsup_{n\goesto\infty}\Sob{\vs^n-\vs}{L^\infty_T \Hdot^\be_\om}+\limsup_{n\goesto\infty}\Sob{\ze^n-\ze}{L^\infty_T \Hdot^{\be}_\om}=0,
    \]
as desired.

\subsection*{Acknowledgments} The authors would like to thank Michael Jolly for his encouragement in the course of this work. V.R.M. also thanks the support of the National Science Foundation through NSF-DMS 2213363, NSF-DMS 2206491, and NSF-DMS 2511403.

\appendix

\section{Proof of Lemma \ref{lem:max:principle}}\label{app:proof:maxp}
\begin{proof}
Now suppose that $p\geq2$ is an even integer and let $\Phi(f)=f^p/p$. { Upon multiplying \eqref{eq:Euler:endpoint} by $\tht^{p-1}$}, we obtain
    \begin{align}
        \frac{1}p\frac{d}{dt}\Sob{\tht}{L^p}^p+\int_{\RR^2}\left(\Phi'(\tht(x)){L}\tht(x)-({L}\Phi(\tht))(x)\right)dx+\int_{\RR^2}({L}\Phi(\tht))(x)dx=0.\notag
    \end{align}
Observe that
    \begin{align}
        \int_{\RR^2}({L}\Phi(\tht))(x)dx=\Ft\left({L}\Phi(\tht)\right)(0)=\ln(1)\Ft\left(\Phi(\tht)\right)(\xi)=0.\notag
    \end{align}
Hence
    \[
      \frac{1}p\frac{d}{dt}\Sob{\tht}{L^p}^p\leq0,
    \]
which implies $\Sob{\tht(t)}{L^p} \leq \Sob{\tht(0)}{L^p}$.  It follows that
    \[
        \int_{B(R)}\tht(y)^pdy\leq\Sob{\tht_0}{L^p}^p,
    \]
for all $R>0$, where $B(R)$ denotes the ball of radius $R$ centered at the origin. In particular, by the Lebesgue Differentiation Theorem, for almost every $x\in\RR^d$, we may fix $R_0(x)>0$ sufficiently small satisfying $|B(R_0(x))|<1$ and
    \begin{align}
        |\tht(x)|&\leq \Sob{\tht_0}{L^\infty}+\frac{1}{|B(R_0(x))|}{\int_{B(R_0(x))}|\tht(y)|dy}\notag\\
        &\leq\Sob{\tht_0}{L^\infty}+ |B(R_0(x))|^{-1/p}\Sob{\tht_0}{L^p},\notag\\
        &\leq\Sob{\tht_0}{L^\infty}+|B(R_0(x))|^{-1/p}\Sob{\tht_0}{L^2}^{2/p}\Sob{\tht_0}{L^\infty}^{(p-2)/p}\notag
    \end{align}
holds, for any even integer $p$. Note that we applied H\"older's inequality to obtain the penultimate inequality. Upon choosing $R_0(x)$ appropriately so that $p_0(x)=-\ln|B(R_0(x))|$ is an even integer, we deduce
    \begin{align}\label{est:omega:Linfty}
        \Sob{\tht(t)}{L^\infty}\leq\Sob{\tht_0}{L^\infty}+ { e}\sup_{p\geq2}\left(\frac{\Sob{\tht_0}{L^2}}{\Sob{\tht_0}{L^\infty}}\right)^{2/p}\Sob{\tht_0}{L^\infty},
    \end{align}
as desired
\end{proof}

\section{Product Estimates}\label{app:prod}

We will now prove \cref{thm:prod}. Following Bony's decomposition, we can formally decompose the product of $f$ and $g$ as
\begin{align}\label{eq:Bony}
		   fg={T}_{f}g+{T}_{g}f+{R}(f,g),
		\end{align}
where we denote by
 		\begin{align}\label{def:para:factors}
 		    {T}_{f}g:=\sum_{k}S_{k-3}f\triangle_{k}g\quad 	and \quad {R}(f,g):=\sum_{k}\triangle_{k}f{\tlpk}g,\quad \tlpk g=\sum_{|i-k|\le3}\triangle_{i}g.
 		\end{align}
Note that the superscript $\tlpk$ denotes the \textit{extended} Littlewood-Paley block. In the following lemma, we obtain estimates for the term denoted by $T_{f}g$ that represents the interactions of low frequencies of $f$ and high frequencies of $g$. 

\begin{Lem}
\label{lem:prod:log:A}
 Given $d\geq2$, let ${s}, \bar{s} \in \RR$ such that ${s}\leq d/2$. Let $\om,\om_1,\tilde{\om}_1\in\mathscr{M}_W$. Let $\Gam$ be a function satisfying the following inequality for all $y>0$
    \begin{align}\label{cond:Gam:A1}
        \frac{\om(y)}{{\tilde{\om}_1(y)}}\left({\mathbbm{1}}_{(-\infty,d/2)}({s})\int_0^1\frac{r^{d-2{s}-1}}{\om_1^2(y r)}dr+{\mathbbm{1}}_{\{d/2\}}({s})\int_{0}^{y}\frac{r^{d-1}dr}{(1+r^2)^{d/2}\om_1^{2}(r)}\right)^{1/2}\le C\Gam(y).
    \end{align}
When ${s}<d/2$, there exists a constant $C>0$ and sequence $\{c_j\}\in\ell^2(\ZZ)$ with $\Sob{\{c_j\}}{\ell^2}\leq1$ such that
		\begin{align}\label{est:prod:log:A1}
		    &\Sob{\triangle_j(T_{f}g)}{L^2(\RR^d)}\le Cc_j2^{-({s}+\bar{s}-d/2) j}\om(2^j)^{-1}\Gam(2^{j})\Sob{f}{{\Hdot}^{{s}}_{\om_1} (\RR^d)}\Sob{g}{\Hdot^{\bar{s}}_{\tilde{\om}_1}(\RR^d)}.
		\end{align}
When ${s}=d/2$, there exists a constant $C>0$ and sequence $\{c_j\}\in\ell^2(\ZZ)$ with $\Sob{\{c_j\}}{\ell^2}\leq1$ such that
    	\begin{align}\label{est:prod:log:A2}
		    &\Sob{\triangle_j(T_{f}g)}{L^2(\RR^d)}\le Cc_j2^{-\bar{s}j}\om(2^j)^{-1}\Gam(2^{j})
		\Sob{f}{H^{d/2}_{\om_1}(\RR^d)}\Sob{g}{\Hdot^{\bar{s}}_{\tilde{\om}_1}(\RR^d)}.
		\end{align}
\end{Lem}

\begin{proof}
First, we apply H\"older's inequality and obtain
		    \begin{align}\label{est:low:high:a}
		      \Sob{\triangle_j({T}_{f}g)}{L^2}\le \sum_{\abs{k-j}\le 2}\Sob{S_{k-3}f}{L^{\infty}}\Sob{\triangle_{k}g}{L^2}\le C\sum_{\abs{k-j}\le 2}\Sob{\chi_{k-3}\hat{f}}{L^{1}}\Sob{\triangle_{k}g}{L^2} .
		    \end{align}
Let
    	\begin{align}\label{def:comm:A:cj}
		    c_{j}:=\frac{\sum_{\abs{k-j}\le 2}2^{\bar{s} k}\Sob{\tilde{\om}_1(D)\triangle_{k}g}{L^2}}{\Sob{g}{\Hdot^{\bar{s}}_{\tilde{\om}_1}}}.
		\end{align}
Observe that $\{c_j\}\in\ell^2(\ZZ)$ and $\Sob{\{c_j\}}{\ell^2}\leq1$. Now, by \eqref{eq:bernstein:om} and \cref{lem:Bernstein}, we may estimate $\Sob{\lpk g}{L^2}$ to obtain
    \begin{align}\label{est:lpkg:a}
        \Sob{\lpj(T_fg)}{L^2}\leq  Cc_{j}\tilde{\om}_1(2^j)^{-1}2^{-\bar{s} j}\Sob{\chi_{j-3}\hat{f}}{L^1}\Sob{g}{\Hdot^{\bar{s}}_{\tilde{\om}_1}(\RR^d)}.
    \end{align}
We are left to estimate $\Sob{\chi_{k-3}\hat{f}}{L^{1}}$. We will treat the cases ${s}<d/2$ and ${s}=d/2$ separately.

\subsubsection*{Case: ${s}<d/2$}
By the Cauchy-Schwarz inequality, we obtain
    \begin{align}
        \int |\chi_{j-1}(\xi)\hat{f}(\xi)|d\xi&\leq \left(\int_{\Bcal_j}\frac{1}{|\xi|^{2{s}}\om_1^{2}(|\xi|)}d\xi\right)^{1/2}\left(\int |\xi|^{2{s}}\om_1^{2}(|\xi|)|\hat{f}(\xi)|^2d\xi\right)^{1/2}\notag.
    \end{align}
We have
    \begin{align}
        \int_{\Bcal_j}\frac{1}{|\xi|^{2{s}}\om_1^{2}(|\xi|)}d\xi&\leq C 2^{(d-2{s})j}\int_0^{2^j}\frac{1}{\om_1^2(2^j(r2^{-j}))}\left(\frac{r}{2^j}\right)^{d-2{s}-1}\frac{dr}{2^j}= C2^{(d-2{s})j}\int_0^{1}\frac{r^{d-2{s}-1}}{\om_1^2(2^jr)}dr.\notag
    \end{align}
It follows that
    \begin{align}\label{est:Skf:prelim:a1}
        \int |\chi_{j-1}(\xi)\hat{f}(\xi)|d\xi\leq C2^{(d/2-{s})j}\til{\om}_1(2^j)\om(2^j)^{-1}\Gam(2^j)\Sob{f}{\Hdot^{{s}}_{\om_1}},
    \end{align}
as desired.

\subsubsection*{Case: ${s}=d/2$}
Similarly, we estimate
    \begin{align}\label{est:Skf:prelim:a2}
        \int |\chi_{j-1}(\xi)\hat{f}(\xi)|d\xi&\leq \left(\int_{\Bcal_j}\frac{1}{(1+|\xi|^2)^{d/2}\om_1^{2}(|\xi|)}d\xi\right)^{1/2}\left(\int (1+|\xi|^2)^{d/2}\om_1^{2}(|\xi|)|\hat{f}(\xi)|^2d\xi\right)^{1/2}\notag \\
        &\le C\left(\int_{0}^{2^{j}}\frac{r^{d-1}dr}{(1+r^2)^{d/2}\om_1^{2}(r)}\right)^{1/2}\Sob{f}{{H}^{d/2}_{\om_1}}\notag\\
        &\leq C\til{\om}_1(2^j)\om(2^j)^{-1}\Gam(2^j)\Sob{f}{H^{d/2}_{\om_1}}.
    \end{align}
From \eqref{est:lpkg:a}, \eqref{est:Skf:prelim:a1}, and \eqref{est:Skf:prelim:a2}, we obtain the desired estimate.
\end{proof}

Next we obtain estimates for the high frequency interactions represented by $R(f,g)$.

\begin{Lem}\label{lem:prod:log:B}
Given $d\geq2$, let ${s},\bar{s}\in\RR$ such that ${s},\bar{s}\leq d/2$ and ${s} +\bar{s}>0$. Let $\om,\om_1,\tilde{\om}_1\in\mathscr{M}_W$. Let $\Gam$ be a function satisfying the following inequality for all $y\ge 0$:
    \begin{align*}\notag
        \frac{\om(y)}{{\om_1(y)}{\tilde{\om}_1(y)}}\le C\Gam(y).
    \end{align*}
Then, there exists a constant $C>0$ and a sequence $\{c_j\}\in\ell^2(\ZZ)$ with $\Sob{\{c_j\}}{\ell^2}\leq1$ such that
    \begin{align}\notag
        \Sob{\triangle_jR(fg)}{L^2(\RR^d)}\le  Cc_j2^{-({s}+\bar{s}-\frac{d}{2})j}  \om(2^j)^{-1}\Gam(2^j)\Sob{f}{\Hdot^{{s}}_{\om_1}(\RR^d)}\Sob{g}{\Hdot^{\bar{s}}_{\tilde{\om}_1}(\RR^d)},
    \end{align}
for all $j\in\ZZ$.
\end{Lem}
\begin{proof}
Let $\rho={s}+\bar{s}$. By Bernstein's inequality and \eqref{est:omega}, we have
    \begin{align*}
        &\Sob{\triangle_{j}{R}(f,g)}{L^2}\\&\le C\sum_{k\ge j-5}2^{(d/2)j}\Sob{\triangle_{k}f}{L^2}\Sob{{\tlpk}g}{L^2}\\
        &\leq C(\om_1(2^j)\tilde{\om}_1(2^j))^{-1}2^{(d/2-\rho)j}\sum_{k\geq j-5}2^{-\rho(k-j)}\left(\frac{\om_1(2^j)\tilde{\om}_1(2^j)}{\om_1(2^k)\tilde{\om}_1(2^k)}\right)2^{{s} k}\Sob{\triangle_{k}f}{L^2_{\om_1}}2^{\bar{s} k}\Sob{{\tlpk}g}{L^2_{\til{\om}_1}}\notag\\
        &\leq Cc_j2^{-(\rho-d/2)j}  \om(2^j)^{-1}\Gam(2^j)\Sob{f}{\Hdot^{{s}}_{\om_1}}\Sob{g}{\Hdot^{\bar{s}}_{\tilde{\om}_1}},
    \end{align*}
where 
    \[
        c_{j}:=\sum_{k\ge j-5}\frac{2^{-\rho(k-j)}\left(\frac{\om_1(2^j)\tilde{\om}_1(2^j)}{\om_1(2^k)\tilde{\om}_1(2^k)}\right)2^{{s} k}\Sob{\triangle_{k}f}{L^2_{\om}}2^{\bar{s} k}\Sob{{\tlpk}g}{L^2_{\til{\om}}}}{C_0\Sob{f}{\Hdot^{{s}}_{\om_1}}\Sob{g}{\Hdot^{\bar{s}}_{\tilde{\om}_1}}}. 
    \]
It remains to show that $\{c_j\}\in\ell^2(\ZZ)$. Since $\om_1,\til{\om}_1\in\mathscr{M}_W$, we recall that
    \[
        \om_1=\frac{\om_1^a}{\om_1^b},\quad \til{\om}_1=\frac{\til{\om}_1^a}{\til{\om}_1^b}. 
    \]
Thus
    \[
    \left(\frac{\om_1(2^j)\tilde{\om}_1(2^j)}{\om_1(2^k)\tilde{\om}_1(2^k)}\right)=\left(\frac{\om_1^a(2^j)\tilde{\om}_1^a(2^j)}{\om_1^a(2^k)\tilde{\om}_1^a(2^k)}\right)\left(\frac{\om_1^b(2^k)\tilde{\om}_1^b(2^k)}{\om_1^b(2^j)\tilde{\om}_1^b(2^j)}\right).
    \]
For the first factor, we observe that $\om_1^a,\til{\om}_1^a$ are increasing by \ref{item:O1} and satisfy \eqref{est:omega}, it follows that
    \begin{align}
	    \om_1^a(2^j)\leq \om_1^a(2^{k+5})\le C\om_1^a(2^k),\quad \text{whenever}\ k-j\geq-5,\notag
	 \end{align}
for some $C$ independent of $k,j$; the same argument applied to $\til{\om}_1^a$. This implies that
    \[
        \frac{\om_1^a(2^j)\tilde{\om}_1^a(2^j)}{\om_1^a(2^k)\tilde{\om}_1^a(2^k)}\leq C.
    \]
For the second factor, we observe that since $\om_1^b,\til{\om}_1^b$ satisfy \ref{item:O1}-\ref{item:O3} and \eqref{eq:omega:eps:bdd}, it follows that
        \begin{align}
	    \om_1^b(2^k)\leq C(\om_1^b(2^{j})+\om_1^b(2^{k-j}))\le C\om_1^b(2^j)(1+2^{\eps(k-j)}),\quad \text{for}\ \eps\in(0,\rho),\notag
        \end{align}
A similar estimate follows for $\til{\om}_1^b$. This implies that
    \[
        \frac{\om_1^b(2^k)\tilde{\om}_1^b(2^k)}{\om_1^b(2^j)\tilde{\om}_1^b(2^j)}\leq C(1+2^{\eps(k-j)}).
    \]
We conclude that for an appropriate normalizing constant $C_0$, we have $\Sob{\{c_j\}}{\ell^2}\leq1$, as desired.
	\end{proof}

Applying \cref{lem:prod:log:A} and \cref{lem:prod:log:B} to the terms in the decomposition \eqref{eq:Bony}, we obtain the estimate \eqref{est:product:Thm} claimed in   \cref{thm:prod}.  

\section{Proof of Theorem \ref{thm:modclaw:wellposed}}\label{app:well:posed}
We will now provide a sketch of the proof of \cref{thm:modclaw:wellposed}.

\begin{proof}[Proof of \cref{thm:modclaw:wellposed}] Let the space  $X^\be$ be as defined in (\ref{def:Xb}). 
We mollify $q$ and $G$ with respect to time by setting
\[q^{n}=\rho_{n}* q,\quad G^{n}=\rho_{n}* G,\]
where $\{\rho_{n}(t)\}_n$ is a sequence of standard mollifiers. It follows that
    \begin{align*}
        &q^{n}\in C([0,T];X^\be)\\
        &m_1(D)^{\gam/2}q^{n}\in \Hdot^{1+\be}_\om\\
        &m_1(D)^{-{1}/{2}}G^{n}\in C([0,T];\Hdot^{\s}_\om \cap L^2_{\om})
    \end{align*}
Moreover, $\{q^{n}\}_n$ is uniformly bounded in $L^{\infty}(0,T;X^\be)$, $\{m_1(D)^{\gam/2}q^{n}\}_n$ is uniformly bounded in $L^{p_0}(0,T;\Hdot^{1+\be}_\om)$ and $\{m_1(D)^{-{1}/{2}}G^{n}\}_n$ is uniformly bounded in $L^{2}(0,T;\Hdot^{\s}_\om \cap L^2_{\om})$. 

Let us consider an artificial viscosity regularization of \eqref{eq:mod:claw}: 

\begin{align}\label{eq:forced-transport:reg}
\begin{cases}
\partial_{t}\theta^{n}+m(D)\tht^n-\frac{1}{n}\De \tht^{n}+ \Div F_{q^{n}}(\tht^{n})=G^{n}.   \\
\tht^{n}(0,x)=\tht_{0}(x).
\end{cases}
\end{align}
For $0\le t\le T$, define
\begin{align*}
    {\mathcal{L}}_{1}(G^{n})&:=\int^{t}_{0}e^{\frac{1}{n}\De(t-s)}G^{n}(s)\, ds,\\
    {\mathcal{L}}_{2}(\tht^{n};q^{n})&:=\int^{t}_{0}e^{\frac{1}{n}\De(t-s)}\Div F_{q^n}(\tht^n)\,ds.
\end{align*}

Since $m_1$ satisfies \eqref{eq:omega:eps:bdd}, it follows that for any positive small $\eps_0$, we have
\begin{align*}
   \Sob{\mathcal{L}_1(G^n)(t)}{\Hdot^{\s}_\om}&\le C\int_{0}^{t}\Sob{(I+{\Lam}^{\eps_0})m_1(D)^{-\frac{1}{2}}e^{\frac{1}{n}\De(t-s)}G^n(s)}{\Hdot^{\s}_\om}ds\\
   &\le C_n(T+T^{1-\eps_0/2})\Sob{m_1(D)^{-\frac{1}{2}}G^n}{L^{\infty}_T \Hdot^{\s}_\om}. 
\end{align*}
To estimate $\Sob{\mathcal{L}_2(\tht^n;q^n)}{\Hdot^{\s}_\om}$, we consider the two cases $\be \in [0,1]$ and $\be \in (1,2]$ separately. Henceforth, $\eps$ will denote a sufficiently small positive number and $\eps_1, \eps_2$, and $\eps_3$ will denote some appropriately chosen real numbers. 
\subsection*{\textbf{Case: ${\be}\in [0,1]$}} Let $\eps_1$ be chosen such that 
    \[
        \eps_1 \in \begin{cases}(0,\be/2),\quad &\text{if}\quad \be\in (0,1],\\
                    (-1,0),\quad &\text{if}\quad \be=0.
                    \end{cases}
    \]
    
\subsubsection*{Subcase: $\s \in [0,1) $} Applying \cref{cor:prod} with $(s,\bar{s})=(1-\be+\eps_1, \s)$, $(\om_1,\til{\om}_1)=(p^{-1}\om,\om)$, $(\om_2, \til{\om}_2)=(\om, p^{-1}\om)$,  $(\om_1,\til{\om}_1)=(p^{-1}\om,\om)$, and $\Gam=m_1^{\gam}$, we have
\begin{align*}
    \Sob{\mathcal{L}_2(\tht^n;q^n)}{\Hdot^{\s}_\om}&\le \int_0^{t}\Sob{({\Lam}^{1+\be-\eps_1}(I+\Lam^\eps)e^{\frac{1}{n}\De(t-s)})\nabla {\Lam}^{-1}\cdot (v^n\tht^n)(s)}{\Hdot^{\s-\be+\eps_1}_{\om m_1^{-\gam}}}ds\\
    &\le C_n T^{(1-\be+\eps_1)/2}(1+T^{-\eps/2})\Sob{q^n}{L^{\infty}_T X^{\be}_\om}\Sob{\tht^n}{L^\infty_T \Hdot^{\s}_\om}.
\end{align*}

\subsubsection*{Subcase: $\s \in [1,2] $} Applying \cref{cor:prod} with $(s,\bar{s})=(1-\be+\eps_1, \s-1)$, $(\om_1,\til{\om}_1)=(p^{-1}\om,\om)$, $(\om_2, \til{\om}_2)=(\om, p^{-1}\om)$,  $(\om_1,\til{\om}_1)=(p^{-1}\om,\om)$, and $\Gam=m_1^{\gam}$, we have
\begin{align*}
    \Sob{\mathcal{L}_2(\tht^n;q^n)}{\Hdot^{\s}_\om}&\le \int_0^{t}\Sob{({\Lam}^{1+\be-\eps_1}(I+\Lam^\eps)e^{\frac{1}{n}\De(t-s)}) (v^n\cdot\nabla\tht^n)(s)}{\Hdot^{\s-1+\be+\eps_1}_{\om m_1^{-\gam}}}ds\\
    &\le C_n T^{(1-\be+\eps_1)/2}(1+T^{-\eps/2})\Sob{q^n}{L^{\infty}_T X^{\be}_\om}\Sob{\tht^n}{L^\infty_T (\Hdot^{\s}_\om \cap L^2_{\om})}.
\end{align*}

\subsection*{Case: ${\be}\in (1,2]$} Let $\eps_2$ be chosen such that
        \[
        \eps_2 \in (0,\be-1).
        \]

\subsubsection*{Subcase: $\s \in [0,1)$}
Applying \cref{cor:prod} with $(s,\bar{s})=(1-\eps_2,\s)$, $(\om_1,\til{\om}_1)=(p^{-1}\om,\om)$, $(\om_2, \til{\om}_2)=(\om, p^{-1}\om)$,  $(\om_1,\til{\om}_1)=(p^{-1}\om,\om)$, $\Gam=m_1^{\gam}$, and $(s,\bar{s})=(\be-1-\eps_2,\s)$, $(\om_1,\til{\om}_1)=(p^{-1}\om,\om)$, $(\om_2, \til{\om}_2)=(\om, p^{-1}\om)$,  $(\om_1,\til{\om}_1)=(\om,\om)$, $\Gam=p^{-1}m_1^{\gam}$, we have
\begin{align*}
    \Sob{\mathcal{L}_2(\tht^n;q^n)(t)}{\Hdot^{\s}_\om}&\le C\int_{0}^{t}\left(\Sob{({\Lam}^{1+\eps_2}(I+\Lam^\eps)e^{\frac{1}{n}\De(t-s)})\nabla {\Lam}^{-1} \cdot((\nabla^{\perp}a(D){q^n}) \tht^n)}{\Hdot^{\s-\eps_2}_{\om m_1^{-\gam}}}\right.\\&\left.\hspace{5 em}+\Sob{({\Lam}^{1+\eps_2}(I+\Lam^\eps)e^{\frac{1}{n}\De(t-s)})\nabla {\Lam}^{-1}\cdot((\nabla^{\perp}{q^n})\tht^n)}{\Hdot^{\s-\eps_2+\be-2}_{\om p m_1^{-\gam}}}\right)ds\\
    &\le C_n T^{(1-\eps_2)/2}(1+T^{-\eps/2})\Sob{q^n}{L^{\infty}_T H^\be_\om}\Sob{\tht^n}{L^\infty_T \Hdot^{\s}_\om}.
\end{align*}

\subsubsection*{Subcase: $\s \in [1,2] $}
Applying \cref{cor:prod} with $(s,\bar{s})=(1-\eps_2,\s-1)$, $(\om_1,\til{\om}_1)=(p^{-1}\om,\om)$, $(\om_2, \til{\om}_2)=(\om, p^{-1}\om)$,  $(\om_1,\til{\om}_1)=(p^{-1}\om,\om)$, $\Gam=m_1^{\gam}$, and $(s,\bar{s})=(\be-1-\eps_2,\s-1)$, $(\om_1,\til{\om}_1)=(p^{-1}\om,\om)$, $(\om_2, \til{\om}_2)=(\om, p^{-1}\om)$,  $(\om_1,\til{\om}_1)=(\om,\om)$, $\Gam=p^{-1}m_1^{\gam}$, we have
\begin{align*}
    \Sob{\mathcal{L}_2(\tht^n;q^n)(t)}{\Hdot^{\s}_\om}&\le C\int_{0}^{t}\left(\Sob{({\Lam}^{1+\eps_2}(I+\Lam^\eps)e^{\frac{1}{n}\De(t-s)})(\nabla^{\perp}a(D){q^n}\cdot \nabla \tht^n)}{\Hdot^{\s-1-\eps_2}_{\om m_1^{-\gam}}}\right.\\&\left.\hspace{5 em}+\Sob{({\Lam}^{1+\eps_2}(I+\Lam^\eps)e^{\frac{1}{n}\De(t-s)}) (\nabla^{\perp}{q^n}\cdot\nabla\tht^n)}{\Hdot^{\s-\eps_2+\be-3}_{\om p m_1^{-\gam}}}\right)ds\\
    &\le C_nT^{(1-\eps_2)/2}(1+T^{-\eps/2})\Sob{q^n}{L^{\infty}_T H^\be_\om}\Sob{\tht^n}{L^\infty_T H^{\s}_\om}.
\end{align*}

\subsubsection*{Subcase: $\s \in (2,1+\be]$} Let $\eps_3$ be chosen such that
        \[
            \eps_3 \in (0,\s-2).
        \]
Using Plancherel's theorem and the fact that $H^s$ is a Banach algebra if $s>1$, we have
\begin{align*}
    \Sob{\mathcal{L}_2(\tht^n;q^n)(t)}{\Hdot^{\s}_\om}&\le C\int_{0}^{t}\left(\Sob{({\Lam}^{1+\eps_3}e^{\frac{1}{n}\De(t-s)})(\nabla^{\perp}a(D){q^n}\cdot \nabla \tht^n)}{H^{\s-1-\eps_3+\eps}}\right.\\&\left. \hspace{5 em}+\Sob{({\Lam}^{1+\eps_3+\be-2}e^{\frac{1}{n}\De(t-s)}) (\nabla^{\perp}{q^n}\cdot\nabla\tht^n)}{H^{\s-1+\eps-\eps_3}}\right)ds\\
    &\le C_nT^{(1-\eps_3)/2}\Sob{q^n}{L^{\infty}_T \Hdot^{1+\be}_\om}\Sob{\tht^n}{L^\infty_T H^{\s}_\om}+ C_nT^{(3-\eps_3-\be)/2}\Sob{q^n}{L^{\infty}_T \Hdot^{1+\be}_\om}\Sob{\tht^n}{L^\infty_T H^{\s}_\om}.
\end{align*}

Applying Picard's theorem \cite{Lemarie-Rieusset2002}, we obtain the existence of a unique solution $\tht^{n}$ to \eqref{eq:forced-transport:reg} such that $\tht^{n}\in L^{\infty}(0,T_n;\Hdot^{\s}_\om \cap L^2_{\om})$ for some time $T_{n}>0$. However, owing to the uniform estimates developed in \cref{sect:apriori}, we can therefore assume that
\begin{align*}
    T_{n}=T,\quad \text{for all}\ n.
\end{align*}
Let us denote by
    \[
    \Tht^{n}(t)=\tht^{n}(t)-\int_{0}^{t}G^{n}(s)ds.
    \]
Then, $\Sob{\Tht^{n}}{L^{\infty}_{T}\Hdot^{\s}_\om \cap L^2_{\om}}$ is bounded uniformly in $n$.
Using similar estimates as above, it is easy to establish that $\Sob{\partial_{t}\Tht^{n}}{L^{\infty}_{T}\Hdot^{-k}_\om}$ is bounded uniformly in $n$, for some sufficiently large $k>0$. By an application of the classical Aubin-Lions lemma (see \cite{ConstantinFoiasBook1988}), there exists $\Tht\in L^{\infty}(0,T;\Hdot^{\s}_\om)$ such that for any given test function $\varphi \in C_{c}^{\infty}([0,T] \times \RR^2)$, one can extract a subsequence of $\{\Tht^n\}$, denoted by $\{\Tht^{n_{k}}\}$ satisfying
\begin{align*}
   & \Tht^{n_{k}}\xrightharpoonup {\text{w*}}\Tht\quad\text{in}\quad L^{\infty}([0,T];\Hdot^{\s}_\om \cap \Hdot^0_\om),\\
    & \varphi\Tht^{n_{k}} \longrightarrow \varphi\Tht\quad \text{in} \quad C([0,T];\Hdot^{\s-\delta}_\om \cap \Hdot^{-\delta}_\om),
\end{align*}
for any $\delta>0$. It then follows that $\tht(t)=\Tht(t)+\int_{0}^{t}G(s)ds$ is a weak solution of \eqref{eq:mod:claw}.  
\end{proof}

\begin{footnotesize}

\end{footnotesize}

\vspace{.3in}
\begin{multicols}{2}

\noindent Anuj Kumar\\ 
{\footnotesize
Department of Mathematics\\
Florida State University\footnote[1]{Affiliation during this course of the work}\\
Department of Mathematics\\
IIT Jodhpur\footnote[7]{Current address}\\
Email: \url{akumar241@outlook.com}} \\[.2cm]

\columnbreak 

\noindent Vincent R. Martinez\footnote[2]{Corresponding author}\\
{\footnotesize
Department of Mathematics \& Statistics\\
CUNY Hunter College \\
Department of Mathematics \\
CUNY Graduate Center \\
Web: \url{http://math.hunter.cuny.edu/vmartine/}\\
Email: \url{vrmartinez@hunter.cuny.edu}\\
}

\end{multicols}

\end{document}